\documentclass[12pt,dvips]{amsart}
\usepackage{amsfonts, amssymb, latexsym, epsfig}

\theoremstyle{plain}
  \newtheorem{theorem}{Theorem}[section]
  \newtheorem{maintheorem}{Main Theorem}[section]
  \newtheorem{fact}{Fact}[section]
  \newtheorem{proposition}[theorem]{Proposition}
  \newtheorem{defthm}[theorem]{Definition--Theorem}
  \newtheorem{lemma}[theorem]{Lemma}
  \newtheorem{deflemma}[theorem]{Definition--Lemma}
  \newtheorem{corollary}[theorem]{Corollary}
  \newtheorem{conjecture}[theorem]{Conjecture}
  \newtheorem{problem}[theorem]{Problem}
  \newtheorem{claim}[theorem]{Claim}
\theoremstyle{definition}
  
  \newtheorem{example}[theorem]{Example}

  \newtheorem{warning}[theorem]{Warning}

 \theoremstyle{remark}
  
\numberwithin{equation}{section}

\def\Ess{{\mathcal{E}}}

\def\Groth{\mathfrak{G}}
\def\Schub{\mathfrak{S}}

\def\xx{\mathbf{x}}
\def\yy{\mathbf{y}}
\def\Flags{{\rm Flags}}
\def\del{\mathrm{del}}
\def\link{\mathrm{link}}

\newcommand{\cross}{+}

\newcommand{\cellsize}{15}
\newlength{\cellsz} 
\newsavebox{\cell}
\sbox{\cell}{\begin{picture}(\cellsize,\cellsize)
\put(0,0){\line(1,0){\cellsize}}
\put(0,0){\line(0,1){\cellsize}}
\put(\cellsize,0){\line(0,1){\cellsize}}
\put(0,\cellsize){\line(1,0){\cellsize}}
\end{picture}}
\newcommand\cellify[1]{\def\thearg{#1}\def\nothing{}%
\ifx\thearg\nothing
\vrule width0pt height\cellsz depth0pt\else
\hbox to 0pt{\usebox{\cell} \hss}\fi%
\vbox to \cellsz{
\vss
\hbox to \cellsz{\hss$#1$\hss}
\vss}}
\newcommand\tableau[1]{\vtop{\let\\\cr
\baselineskip -16000pt \lineskiplimit 16000pt \lineskip 0pt
\ialign{&\cellify{##}\cr#1\crcr}}}
\font\co=lcircle10
\def\petit#1{{\scriptstyle #1}}

\def\jr{\smash{\raise2pt\hbox{\co \rlap{\rlap{\char'005} \char'007}}
               \raise6pt\hbox{\rlap{\vrule height5pt}}
               \raise2pt\hbox{\rlap{\hskip4pt \vrule height0.4pt depth0pt
                width5.7pt}}
               \raise2pt\hbox{\rlap{\hskip-9.5pt \vrule height.4pt depth0pt
                width6.2pt}}
               \lower6pt\hbox{\rlap{\vrule height4.5pt}}}}
\def\rj{\smash{\raise2pt\hbox{\co \rlap{\rlap{\char'004} \char'006}}
               \raise6pt\hbox{\rlap{\vrule height5pt}}
               \raise2pt\hbox{\rlap{\hskip4pt \vrule height0.4pt depth0pt
                width5.7pt}}
               \raise2pt\hbox{\rlap{\hskip-9.5pt \vrule height.4pt depth0pt
                width6.2pt}}
               \lower6pt\hbox{\rlap{\vrule height4.5pt}}}}
\def\je{\smash{\raise2pt\hbox{\co \rlap{\rlap{\char'005}
                \phantom{\char'007}}}\raise6pt\hbox{\rlap{\vrule height5pt}}
               \raise2pt\hbox{\rlap{\hskip-9.5pt \vrule height.4pt depth0pt
                width6.2pt}}}}
\def\ej{\smash{\raise2pt\hbox{\co \rlap{\rlap{\char'004}\phantom{\char'006}}}
               \raise2pt\hbox{\rlap{\hskip-9.5pt \vrule height.4pt depth0pt
                width6.2pt}}
               \lower6pt\hbox{\rlap{\vrule height4.5pt}}}}
\def\er{\smash{\raise2pt\hbox{\co \rlap{\rlap{\phantom{\char'005}} \char'007}}
               \raise2pt\hbox{\rlap{\hskip4pt \vrule height0.4pt depth0pt
                width5.7pt}}
               \lower6pt\hbox{\rlap{\vrule height4.5pt}}}}
\def\re{\smash{\raise2pt\hbox{\co \rlap{\rlap{\phantom{\char'004}} \char'006}}
               \raise6pt\hbox{\rlap{\vrule height5pt}}
               \raise2pt\hbox{\rlap{\hskip4pt \vrule height0.4pt depth0pt
                width5.7pt}}}}
\def\+{\smash{\lower6pt\hbox{\rlap{\vrule height17pt}}
                \raise2pt
                \hbox{\rlap{\hskip-9pt \vrule height.4pt depth0pt
                width18.7pt}}}}
\def\hor{\smash{\raise2pt\hbox{\rlap{\hskip-9.5pt \vrule height.4pt depth0pt
                width19.2pt}}}}
\def\ver{\smash{\lower6pt\hbox{\rlap{\vrule height17pt}}}}
\def\ho{\smash{\hbox{\rlap{\vrule height5pt}}
                \raise2pt
                \hbox{\rlap{\hskip-9pt \vrule height.4pt depth0pt
                width18.7pt}}}}

\def\perm#1#2{\hbox{\rlap{$\petit {#1}_{\scriptscriptstyle #2}$}}%
                \phantom{\petit 1}}

\def\textcross{\ \smash{\lower4pt\hbox{\rlap{\hskip4.15pt\vrule height14pt}}
                \raise2.8pt\hbox{\rlap{\hskip-3pt \vrule height.4pt depth0pt
                width14.7pt}}}\hskip12.7pt}
\def\textelbow{\ \hskip.1pt\smash{\raise2.75pt%
                \hbox{\co \hskip 4.15pt\rlap{\rlap{\char'004} \char'006}
                \lower6.8pt\rlap{\vrule height3.5pt}
                \raise3.6pt\rlap{\vrule height3.5pt}}
                \raise2.8pt\hbox{%
                  \rlap{\hskip-7.15pt \vrule height.4pt depth0pt width3.5pt}%
                  \rlap{\hskip4.05pt \vrule height.4pt depth0pt width3.5pt}}}
                \hskip8.7pt}

\begin{document}
\pagestyle{plain}
\title{A Gr\"{o}bner basis for Kazhdan-Lusztig ideals}
\author{Alexander Woo}
\address{Department of Mathematics\\
University of Idaho\\
Moscow, ID 83844-1103 }
\email{awoo@uidaho.edu}

\author{Alexander Yong}
\address{Department of Mathematics\\
University of Illinois at Urbana-Champaign\\
Urbana, IL 61801}
\email{ayong@illinois.edu}



\date{November 11, 2011}

\begin{abstract}
\emph{Kazhdan--Lusztig ideals}, a family of generalized determinantal ideals
investigated in [Woo--Yong '08], provide an explicit choice of coordinates and
equations encoding a neighbourhood of a torus-fixed point of a Schubert
variety on a type $A$ flag variety.  Our main result is a Gr\"{o}bner basis
for these ideals. This provides a single geometric
setting to transparently explain
the naturality of pipe dreams on the \emph{Rothe diagram of a
permutation}, and their appearance in:
\begin{itemize}
\item[$\bullet$] combinatorial formulas [Fomin--Kirillov '94]
for Schubert and Grothendieck polynomials of
[Lascoux--Sch\"{u}tzenberger '82];
\item the equivariant $K$-theory specialization formula of
[Buch--Rim\'{a}nyi '04]; and
\item a positive combinatorial formula for multiplicities of Schubert
  varieties in good cases, including those for which the associated Kazhdan--Lusztig ideal is
  homogeneous under the standard grading.
\end{itemize}
Our results generalize (with alternate proofs)
[Knutson--Miller '05]'s Gr\"{o}bner basis
theorem for Schubert determinantal ideals
and their geometric interpretation of the monomial positivity of
Schubert polynomials. We also complement recent work of [Knutson '08 $\&$ '09]
on degenerations of Kazhdan--Lusztig varieties in general Lie type, as well as
work of [Goldin '01] on equivariant localization and of [Lakshmibai--Weyman
  '90], [Rosenthal--Zelevinsky '01], and [Krattenthaler '01] on Grassmannian
multiplicity formulas.
\end{abstract}

\maketitle

\tableofcontents

\section{Introduction}

In our study \cite{WYII} of singularities of Schubert
varieties, we investigated \emph{Kazhdan--Lusztig ideals}. These
encode coordinates and equations
for neighbourhoods of type $A$ Schubert
varieties at torus fixed points.  The following problem naturally arises:
\begin{problem}
\label{problem:intro}
Determine a Gr\"{o}bner basis for the Kazhdan--Lusztig
ideals.
\end{problem}
The main result of this paper solves Problem~\ref{problem:intro} by giving an
explicit Gr\"obner basis with squarefree lead terms, thereby producing a single
degeneration from the (reduced and irreducible) variety of the Kazhdan--Lusztig
ideal to a reduced union of coordinate subspaces.  The class of
Kazhdan--Lusztig ideals includes determinantal ideals and many of their
generalizations. Gr\"{o}bner bases for determinantal ideals and their
generalizations have been of interest, as their study requires and has
applications involving simultaneously commutative algebra, algebraic geometry,
representation theory, and combinatorics.  Various aspects of determinantal
ideals and their generalizations have been extensively treated in the books
\cite{Bruns.Vetter, Miller.Sturmfels, MiroRoig}.

We use our Gr\"{o}bner degeneration to give new explanations for some Schubert
combinatorics.  We point out a geometric origin for various
combinatorial formulas for Schubert and Grothendieck polynomials
\cite{Lascoux.Schutzenberger1, Lascoux.Schutzenberger2}.  Specifically, we
explain the combinatorics of \emph{pipe dreams} on the \emph{Rothe diagram} of
a permutation, rather than just the $n\times n$ grid itself, which was the
focus of \cite{Knutson.Miller, KMY}. We thus obtain a geometric
explanation and new proof of the formula of A.~Buch and R.~Rim\'{a}nyi
\cite{Buch.Rimanyi} for the equivariant $K$-theory specializations of
these polynomials (as well as for older formulas of
\cite{Fomin.Kirillov}). This answers a question raised in
\cite[Section~2]{Buch.Rimanyi}.

Moreover, our results generalize both the earlier Gr\"{o}bner geometry
explanation of the monomial positivity of Schubert polynomials from
\cite[Theorem~A]{Knutson.Miller} as well as their companion
Gr\"{o}bner basis theorem \cite[Theorem~B]{Knutson.Miller} for
\emph{Schubert determinantal ideals}.  In brief, as suggested in
\cite{Fulton:Duke92}, Schubert determinantal ideals can be realized as
special cases of Kazhdan--Lusztig ideals, allowing us to use
recursions for equivariant $K$-theory Schubert classes due to
B.~Kostant and S.~Kumar \cite{Kostant.Kumar}.  Although this approach
assumes older background, it streamlines proofs by allowing induction
among a larger class of ideals.

This paper is related to developments in
\cite{Knutson.Miller:subword}, \cite{KMY}, \cite{Knutson:patches}, and
\cite{Knutson:frob}.  Specifically, our results complement recent work of
A.~Knutson \cite{Knutson:patches} on reduced degenerations of
\emph{Kazhdan--Lusztig varieties} \cite{Kazhdan.Lusztig} for generalized flag
varieties of arbitrary Lie type.  His paper provided inspiration for the present one. He gives many options for iteratively degenerating a
Kazhdan--Lusztig variety into a reduced union of affine spaces; this union is
described by a \emph{subword complex} \cite{Knutson.Miller:subword}.
More recently, Knutson \cite{Knutson:frob} describes an alternate approach to this
degeneration using Frobenius splitting and Bott--Samelson coordinates, though
he does not explicitly provide equations.

By contrast, we work only in type $A$ and single out a specific
choice of coordinates and a specific Gr\"obner degeneration, which we view as
especially natural in light of the aforementioned Schubert polynomial
combinatorics.  Our Gr\"{o}bner basis theorem provides an explicit,
non-iterative realization of one of the degenerations of
\cite{Knutson:patches} and \cite{Knutson:frob}, using explicit coordinates and
explicit equations in those coordinates.  However, our techniques are distinct
from those of Knutson's aforementioned work, though parallel to them.  We do
not use \cite{Knutson:frob}, nor specifically any of
the new results from \cite{Knutson:patches}, substituting instead direct
combinatorial commutative algebra arguments.  Our Gr\"{o}bner basis
theorem seems difficult to adapt for other degenerations
in \cite{Knutson:patches} and \cite{Knutson:frob}, highlighting that the
choice of coordinates and term order is somewhat delicate.

We expect our methods to extend to finding a generalization of our main
result to other Lie types; we hope to address this in a future paper.
However, type $A$ warrants special attention
since we would like to further our understanding of the geometric genesis of
the beautiful properties of Schubert
polynomial and pipe dream combinatorics.
As noted in \cite[Section~10]{Fomin.Kirillov:B}, the ensemble
of their many properties is specific to type~$A$.

\subsection{Organization of the paper}
This paper has three central, closely interconnected results (Theorems~\ref{thm:main1},~\ref{thm:prime}
and~\ref{thm:specialization}) and an application, as summarized below.

In Section~2, we recall preliminaries about Kazhdan--Lusztig ideals and
state the main result (Theorem~\ref{thm:main1}),
our Gr\"obner basis theorem.

In Section~3, we recall \emph{pipe dreams on Rothe diagrams}
\cite{Fomin.Kirillov, Buch.Rimanyi}.  Our second main result,
Theorem~\ref{thm:prime}, is that they transparently index components of the
prime decomposition of the initial ideal (under our term order) of a
Kazhdan--Lusztig ideal.  We present the associated Stanley--Reisner simplicial
complex with faces labeled by these pipe dreams. This complex has the
structure of an abstract subword complex \cite{Knutson.Miller:subword} and
thus inherits good properties we use; in particular, it is homeomorphic to a
vertex decomposible (and hence shellable) ball or sphere. An
additional advantage of using pipe dreams is that they admit an even more
``graphical'' description as interlacing strand diagrams, akin to the original
objects of \cite{Fomin.Kirillov}; this is not seen at the abstract subword
level.  The proof of Theorem~\ref{thm:main1} requires the definitions but not
the results of this section.

In Section~4, we introduce \emph{unspecialized Grothendieck polynomials}, a
synthesis of results of
\cite{Fomin.Kirillov, Buch.Rimanyi, Knutson.Miller}. In terms of these
polynomials, our third main result (Theorem~\ref{thm:specialization}) shows
that their specializations to Schubert and Grothendieck polynomials and thus
the formulas of \cite{Fomin.Kirillov, Buch.Rimanyi} arise from Gr\"{o}bner
geometry and combinatorial commutative algebra as multidegrees and
$K$-polynomials as defined in, for example, \cite{Knutson.Miller,
  Miller.Sturmfels}.
We prove this result assuming Theorem~\ref{thm:main1}.  We also include a proof of a theorem previously known as folklore,
giving a geometric interpretation of the aforementioned specializations in
terms of equivariant $K$-theory on the flag variety.  As with Section~3,
the proof of Theorem~\ref{thm:main1} requires some definitions but not the
results from this section.

In Section~5, we return to our initial motivation in \cite{WYII} of
understanding invariants of singularities of Schubert varieties. We
relate our work to the open problem of finding an
explicit, nonrecursive combinatorial rule for the multiplicity of a Schubert
variety at a torus fixed point.
We show that in good cases, including those when the Kazhdan--Lusztig ideal is
homogeneous under the standard grading, one can positively calculate the
multiplicity by counting the pipe dreams of Section~3.  In particular, this
leads to a simple proof of a
determinantal formula related but not identical to known
formulas~\cite{Lakshmibai.Weyman, Rosenthal.Zelevinsky} for multiplicities
of torus fixed points of Grassmannian Schubert varieties.
We point out the efficacy of using of random sampling
methods in our study of Schubert varieties, complementing the computational
commutative algebra methods of \cite{WYII}.  

Finally in Section~6 we give proofs for Theorems~\ref{thm:main1} and~\ref{thm:prime}.

\section{Kazhdan--Lusztig ideals and the main theorem}

\subsection{Flag, Schubert and Kazhdan--Lusztig varieties}
Let $G=GL_n({\mathbb C})$, $B$ be the Borel subgroup of strictly upper triangular matrices,
$T\subset B$ the maximal torus of diagonal matrices,
and $B_{-}$ the corresponding opposite Borel subgroup of strictly lower triangular matrices.
The {\bf complete flag  variety} is
\[\Flags({\mathbb C}^n):=G/B.\]
The fixed points of $G/B$ under
the left action of $T$ are naturally indexed by the symmetric group $S_n$ in
its role as the Weyl group of $G$; we denote these points $e_v$ for $v\in
S_n$.  One has a cell decomposition
\[G/B=\coprod_{w\in S_n} Be_w\]
known as the {\bf Bruhat decomposition}.
The $B$-orbit $X_{w}^{\circ}:=Be_w$
is a cell known as the {\bf Schubert cell}, and its
closure $X_w:=\overline{X_{w}^{\circ}}$
is the {\bf Schubert variety}.  It is a subvariety of dimension $\ell(w)$,
where $\ell(w)$ is the length of any reduced word of $w$.  Each Schubert variety
$X_w$ is a union of Schubert cells; {\bf Bruhat order} is the partial order on
$S_n$ defined by declaring that
\[v\leq w \mbox{\ \ \ if $X^\circ_v\subseteq X_w$.}\]

Since every point on $X_w$ is in the $B$-orbit of some $e_v$ (for $v\leq w$ in
Bruhat order), it follows that the study of local questions on Schubert
varieties reduces to the case of these fixed points.  An affine neighbourhood
of $e_v$ is given by $v\Omega^{\circ}_{id}$, where in general
$\Omega_{u}^{\circ}:=B_{-}e_u$
is known as an {\bf opposite Schubert cell}.  Hence to study $X_w$ locally at
$e_v$ one only needs to understand $X_w\cap v\Omega_{id}^{\circ}$. However, by
\cite[Lemma A.4]{Kazhdan.Lusztig}, one has the isomorphism
\begin{equation}
\label{eqn:KL}
X_{w}\cap v\Omega_{id}^{\circ}\cong (X_w\cap \Omega_{v}^{\circ})\times
{\mathbb A}^{\ell(v)}.
\end{equation}

Hence, it is essentially equivalent to
study the (reduced and irreducible)
{\bf Kazhdan--Lusztig variety}
\[{\mathcal N}_{v,w}=X_w\cap
\Omega_{v}^{\circ},\]
forgetting the factor of affine space.

\subsection{A choice of coordinate system, equations, and Kazhdan--Lusztig ideals}
\label{subsect:eqns}

We now define coordinates on $\Omega^\circ_v$, the {\bf
  Kazhdan--Lusztig ideal} $I_{v,w}$ in these coordinates, and
summarize why this ideal vanishes on $\mathcal{N}_{v,w}$.  Let $M_n$
be the set of all $n\times n$ matrices with entries in ${\mathbb C}$,
with coordinate ring ${\mathbb C}[{\bf z}]$ where ${\bf
  z}=\{z_{i,j}\ \}_{i,j=1}^{n}$ are the coordinate functions on the
entries of a generic matrix $Z$.

\begin{warning}
We index our variables so that $z_{i,j}$ is in the $i$-th row from the
{\em bottom} of the matrix and $j$-th column from the left.
\end{warning}

This indexing is consistent with our notation in~\cite{WYII}.  We made
this admittedly ugly choice of notation as a compromise between
inconsistent choices in the literature.  Our notation allows the
Schubert variety $X_w$ to be concretely realized as the closure of
$Be_w$, so that $\dim X_w=\ell(w)$.  At the same time, our choice
allows the ideals defining matrix Schubert varieties use the same
indexed variables as in~\cite{Knutson.Miller}, which
following~\cite{Fulton:Duke92} defines matrix Schubert varieties as the
closures (in $M_n$ of the pullback to $G$) of $B_{-}e_w$.

If we concretely realize $G$, $B$, $B_{-}$, and $T$ as invertible,
upper triangular, lower triangular, and diagonal matrices
respectively, then, as explained in \cite{Fulton:YT}, we can realize
the opposite Schubert cell $\Omega_v^\circ$ as an affine subspace of
$M_n$.  Specifically, a matrix is in (our realization of)
$\Omega_v^\circ$ if, for all $i$,
\[z_{n-v(i)+1,i}=1\]
and, for all $i$,
\[z_{n-v(i)+1,a}=0  \mbox{\ \ and $z_{b,i}=0$ for $a>i$ and $b>n-v(i)+1$.}\]
Let
\[{\bf  z}^{(v)}\subseteq {\bf z}\]
denote the remaining unspecialized variables, and
$Z^{(v)}$ the specialized generic matrix representing a generic element of
$\Omega_v^\circ$.

\begin{example}
\label{exa:261345}
If $n=6$ and $v=261345$ we have:
\[Z=\left(\begin{matrix}
z_{61} & z_{62} & z_{63} & z_{64} & z_{65} & z_{66}\\
z_{51} & z_{52} & z_{53} & z_{54} & z_{55} & z_{56}\\
z_{41} & z_{42} & z_{43} & z_{44} & z_{45} & z_{46}\\
z_{31} & z_{32} & z_{33} & z_{34} & z_{35} & z_{36}\\
z_{21} & z_{22} & z_{23} & z_{24} & z_{25} & z_{26}\\
z_{11} & z_{12} & z_{13} & z_{14} & z_{15} & z_{16}
\end{matrix}\right), \mbox{ \ and \ }
Z^{(v)}=\left(\begin{matrix}
0 & 0 & 1 & 0 & 0 & 0 \\
1 & 0 & 0 & 0 & 0 & 0\\
z_{41} & 0 & z_{43} & 1 & 0 & 0 \\
z_{31} & 0 & z_{33} & z_{34} & 1 & 0\\
z_{21} & 0 & z_{23} & z_{24} & z_{25} & 1\\
z_{11} & 1 & 0 & 0 & 0 & 0
\end{matrix}\right).
\]
\qed \end{example}

The ideal $I_{v,w}$ will be an ideal in the polynomial ring
$\mathbb{C}[\mathbf{z}^{(v)}]$.  To describe $I_{v,w}$, let $Z_{ab}^{(v)}$
denote the southwest $a\times b$ submatrix of $Z^{(v)}$.  We also let
\[R^w=[r_{ij}^{w}]_{i,j=1}^{n}\]
be the {\bf rank matrix} (which we index
similarly) defined by
$$r_{ij}^{w}=\#\{k\ | \ w(k)\geq n-i+1, k\leq j\}.$$

\begin{example} \label{exa:365124}
If $w=365124\in S_6$ then
$$
w=\left(\begin{matrix}
0 & 0 & 0 & 1 & 0 & 0\\
0 & 0 & 0 & 0 & 1 & 0\\
1 & 0 & 0 & 0 & 0 & 0\\
0 & 0 & 0 & 0 & 0 & 1\\
0 & 0 & 1 & 0 & 0 & 0\\
0 & 1 & 0 & 0 & 0 & 0
\end{matrix}\right) \mbox{ and }
R^w = \left(\begin{matrix}
1 & 2 & 3 & 4 & 5 & 6\\
1 & 2 & 3 & 3 & 4 & 5 \\
1 & 2 & 3 & 3 & 3 & 4 \\
0 & 1 & 2 & 2 & 2 & 3 \\
0 & 1 & 2 & 2 & 2 & 2\\
0 & 1 & 1 & 1 & 1 & 1
\end{matrix}\right).$$
\qed
\end{example}

Define the {\bf Kazhdan--Lusztig ideal} $I_{v,w}$ to be the ideal of
${\mathbb C}[{\bf z}^{(v)}]$ generated by all of the {\bf defining
  minors}, which are the size $1+r_{ij}^{w}$ minors of $Z_{ij}^{(v)}$
for all $i$ and $j$.  It follows from \cite[Lemma~6.1]{Fulton:Duke92}
that the defining minors vanish on any point of the Schubert variety
$X_w$ (no matter what coset representative is chosen to write the
point).  Therefore, $I_{v,w}$ vanishes on $\mathcal{N}_{v,w}$.  We
explained in \cite{WYII} that $\mathcal{N}_{v,w}$ is in fact defined
scheme-theoretically by $I_{v,w}$.  Our argument there required the
result of Fulton that these minors actually suffice to define
(scheme-theoretically) the matrix Schubert variety.  However, using
only the theorem that $I_{v,w}$ vanishes on $\mathcal{N}_{v,w}$, the
arguments of this paper give an independent proof of scheme-theoretic
equality, by establishing an equality of Hilbert series
(Theorem~\ref{thm:subwordKpoly}).

Not all these minors are needed to generate $I_{v,w}$. To state a smaller generating
set, we state some further definitions.
Give coordinates to an ambient $n\times n$ grid so that $(1,1)$ refers to
the southwest corner, $(n,1)$ refers to the northwest corner, and so on.  Now,
to each $v\in S_n$, the {\bf Rothe diagram} $D(v)$ is the following subset of
the $n\times n$ grid:
\begin{equation}
\label{eqn:diagramdef}
D(v)=\{(i,j): i<n-w(j)+1 \mbox{\ and \ } j<w^{-1}(n-i+1)\}.
\end{equation}
Alternatively, this set is described as follows.  Place a dot $\bullet$ in
position $(n-w(j)+1,j)$ for $1\leq j\leq n$. For each dot draw the ``hook''
that extends to the right and above that dot. The boxes that are not in
any hook are the boxes of $D(v)$.

Notice that with these conventions, the coordinates of
the boxes of $D(v)$ are exactly the labels for the (unspecialized)
variables appearing in $Z^{(v)}$.

The {\bf essential set} $\Ess(v)$ can be described as the set of those
boxes which are on the northeast edge of some connected components of $D(v)$.
To be precise,
\begin{equation}
\label{eqn:essentialsetprecisely}
(i,j)\in \Ess(v) \mbox{\ \ \ if \  } (i,j)\in D(v)  \mbox{\ but both \ }
(i+1,j)\not\in D(v) \mbox{\  and \ } (i,j+1)\not\in D(v).
\end{equation}

\begin{example}
Continuing the above example, we have
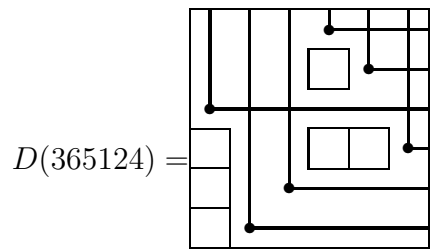
\[\begin{picture}(180,105)
\put(-30,30){$D(365124)=$}
\put(37.5,0){\makebox[0pt][l]{\framebox(90,90)}}
\put(37.5,30){\line(1,0){15}}
\put(52.5,45){\line(0,-1){45}}
\put(52.5,45){\line(-1,0){15}}
\put(37.5,15){\line(1,0){15}}
\put(82.5,60){\makebox[0pt][l]{\framebox(15,15)}}
\put(82.5,30){\makebox[0pt][l]{\framebox(30,15)}}
\put(97.5,30){\line(0,1){15}}
\thicklines
\put(75,22.5){\circle*{4}}
\put(75,22.5){\line(1,0){52.5}}
\put(75,22.5){\line(0,1){67.5}}
\put(45,52.5){\circle*{4}}
\put(45,52.5){\line(1,0){82.5}}
\put(45,52.5){\line(0,1){37.5}}
\put(60,7.5){\circle*{4}}
\put(60,7.5){\line(1,0){67.5}}
\put(60,7.5){\line(0,1){82.5}}
\put(90,82.5){\circle*{4}}
\put(90,82.5){\line(1,0){37.5}}
\put(90,82.5){\line(0,1){7.5}}
\put(105,67.5){\circle*{4}}
\put(105,67.5){\line(1,0){22.5}}
\put(105,67.5){\line(0,1){22.5}}
\put(120,37.6){\circle*{4}}
\put(120,37.5){\line(1,0){7.5}}
\put(120,37.5){\line(0,1){52.5}}
\end{picture}
\]
The essential set is ${\mathcal E}(w)=\{(3,1),(5,4),(3,5)\}$.\qed
\end{example}

Finally, the {\bf essential minors} are the size $1+r_{ij}^w$ minors
of $Z_{ij}^{(v)}$ for $(i,j)\in\Ess(w)$. It follows from
\cite[Lemma~3.10]{Fulton:Duke92} that this smaller subset of the defining minors of
$I_{v,w}$ also generates this ideal.

\subsection{A Gr\"obner basis for Kazhdan--Lusztig ideals}

Let $\prec$ be the pure lexicographic term order on monomials in ${\bf
  z}^{(v)}$ induced by favoring variables further to the right in $Z^{(v)}$,
breaking ties by favoring variables further down in a given column.  In
symbols,
\[z_{ij}\prec z_{kl} \mbox{\ if $j<l$, or if $j=l$ and $i<k$.}\]
(We use the
convention that the leading term is the one which is largest in our order.) For instance, in
Example~\ref{exa:261345} we have
\[z_{25}\succ z_{24} \succ z_{34} \succ z_{23}\succ
z_{33}\succ z_{43} \succ z_{11}\succ z_{21} \succ
z_{31}\succ z_{41}.\]

Let $\Bbbk$ be a field of arbitrary characteristic.
We now state our main result:

\begin{maintheorem}
\label{thm:main1}
Under the term order $\prec$, the essential (and therefore defining)
minors form a Gr\"{o}bner basis for $I_{v,w}\subseteq {\Bbbk}[{\bf z}^{(v)}]$.
\end{maintheorem}

\begin{example}
Continuing Examples~\ref{exa:261345} and~\ref{exa:365124}, the set of
essential minors of $Z^{(261345)}$ consist of three $1\times 1$ minors, ten
$3\times 3$ minors and five $4\times 4$ minors, corresponding to the three
essential set boxes associated to $w$. These form a Gr\"{o}bner basis for
the ideal $I_{261345,365124}$. A subset of these essential minors are:
\[z_{11}, z_{21}, z_{31},
\left|
\begin{matrix}
z_{31} & 0 & z_{33}\\
z_{21} & 0 & z_{23}\\
z_{11} & 1 & 0
\end{matrix}\right|=-z_{23}z_{31}+z_{33}z_{21},
\]
\[
\left|\begin{matrix}
z_{41} & 0 & z_{43} & 1\\
z_{31} & 0 & z_{33} & z_{34}\\
z_{21} & 0 & z_{23} & z_{24}\\
z_{11} & 1 & 0 & 0\\
\end{matrix}
\right|=z_{41}z_{33}z_{24}-z_{41}z_{23}z_{34}-z_{43}z_{31}z_{24}
+z_{43}z_{34}z_{21}+z_{31}z_{23}-z_{23}z_{21}.\]
Note some of the generators may be inhomogeneous (with respect to the standard grading). Other minors may be identically equal to $0$.
\qed \end{example}

Our main theorem specializes to the Gr\"{o}bner basis theorem of
\cite[Theorem~B]{Knutson.Miller} (also valid for a field ${\Bbbk}$ of
arbitrary characteristic). The {\bf Schubert determinantal ideal}
$I_{w}$ is generated by all size $1+r_{ij}^{w}$ minors of the southwest
$i\times j$ submatrix $Z_{ij}$ of $Z$, for all $i,j$.  A {\bf
  diagonal term order} is one which selects the diagonal (meaning northwest to
southeast) term in any minor of $Z$.
(Owing to the indexing of
  Schubert varieties we are using,
  which is upside-down and transposed from that of
  \cite{Knutson.Miller}, our
  diagonal term order is the same as their
  \emph{anti}diagonal term order.
  Moreover, our $z_{ij}$ is their $z_{ji}$.)

\begin{corollary}
\label{cor:Knutson.Miller}
The essential minors of the Schubert determinantal ideal $I_w$
are a Gr\"{o}bner basis with respect to any diagonal term order.
\end{corollary}

The {\bf matrix Schubert variety} ${\overline X}_w$ is the (reduced and
irreducible) variety in $M_n$ defined by $I_w$.  Matrix Schubert varieties
were introduced in \cite{Fulton:Duke92}.  The proof of the corollary
introduces a construction that shows matrix Schubert
varieties are special cases of Kazhdan--Lusztig varieties, as noted in the
introduction.  We will refer again to this construction in Section 4.3.

\medskip
\noindent
\emph{Proof of Corollary~\ref{cor:Knutson.Miller}:}
Let $w_0 \star w_0\in S_{2n}$
be the permutation such that
\[(w_0 \star w_0)(i)=w_0^{(n)}(i) \mbox{\ and \ }
(w_0 \star w_0)(i+n)=w_0^{(n)}(i)+n \mbox{\ for \ } 1\leq i\leq n\]
and where $w_0^{(n)}$ is the permutation
of longest length in $S_n$. Furthermore,
given $w\in S_n$, let
\[w\times 1_{n}\in S_{2n}\]
be the standard embedding into
$S_{2n}$, where
\[(w\times 1_n)(i)=w(i) \mbox{\ and \ } (w\times 1_n)(i+n)=i+n \mbox{\ for \ }
1\leq i\leq n.\]
Finally, set
\[{\hat w}= w_0^{(2n)}(w\times 1_n)\in S_{2n}.\]

Notice that $Z^{(w_0 \star w_0)}$ only involves the variables $z_{ij}$ for
$1\leq i,j\leq n$.  Now observe that $I_{w_0 \star w_0,{\hat w}}$ has the
exact same minors as $I_w$.
Under the term order $\prec$, the diagonal term of any minor is the leading
term as long as it is nonzero.  Hence by Theorem~\ref{thm:main1}, the diagonal
terms $\{d_1,\ldots,d_N\}$ of the essential minors of $I_{w_0 \star w_0, {\hat
    w}}$ generate the initial ideal of $I_{w_0\star w_0, {\hat w}}$ under
$\prec$, denoted ${\rm in}_{\prec} I_{w_0\star w_0, {\hat w}}$.  For the term
order $\prec$, the result now follows as a special case of
Theorem~\ref{thm:main1}.

Now suppose $\prec'$ is another diagonal term order.  Then if ${\rm
  in}_{\prec'} I_{w_0\star w_0, {\hat w}}$ is the initial ideal with respect
to $\prec'$, we have
\[{\rm in}_{\prec} I_{w_0\star w_0, {\hat w}}=\langle d_1,\ldots, d_N\rangle\subseteq
{\rm in}_{\prec'} I_{w_0\star w_0, {\hat w}}.\] However, since both ${\rm
  in}_\prec I_{w_0\star w_0, {\hat w}}$ and ${\rm in}_{\prec'} I_{w_0\star
  w_0, {\hat w}}$ are Gr\"{o}bner degenerations of $I_{w_0\star w_0, {\hat
    w}}$, they have the same Hilbert series (under the grading coming from the
    ``usual action'', as described in Section~4.1).  Hence they are equal and the theorem holds for
$\prec^\prime$.\qed

\begin{example}
Let $w=2143\in S_4$. Then since $n=4$, we have
\[w_0 \star w_0=43218765\in S_8, \ w\times 1_4=21435678, \mbox{\ and ${\hat w}=w_0^{(8)}(w\times 1_4)=78564321$.}\]
Hence
\[Z^{(w_0 \star w_0)}=Z^{(43218765)}=\left(\begin{matrix}
0 & 0 & 0 & 1 & 0 & 0 & 0 & 0  \\
0 & 0 & 1 & 0 & 0 & 0 & 0 & 0  \\
0 & 1 & 0 & 0 & 0 & 0 & 0 & 0  \\
1 & 0 & 0 & 0 & 0 & 0 & 0 & 0  \\
z_{41} & z_{42} & z_{43} & z_{44} & 0 & 0 & 0 & 1  \\
z_{31} & z_{32} & z_{33} & z_{34} & 0 & 0 & 1 & 0  \\
z_{21} & z_{22} & z_{23} & z_{24} & 0 & 1 & 0 & 0  \\
z_{11} & z_{12} & z_{13} & z_{14} & 1 & 0 & 0 & 0  \\
\end{matrix}\right).\]

The reader can check that
\[I_{w_0 \star w_0,{\hat w}}=\left\langle z_{11},
\left|
\begin{matrix}
z_{31} & z_{32} & z_{33}\\
z_{21} & z_{22} & z_{23}\\
z_{11} & z_{12} & z_{13}
\end{matrix}
\right|\right\rangle,\] which are the essential minors of $I_{2143}$.  (Compare
with the first example of \cite{Knutson.Miller}.)
\qed \end{example}

\section{Pipe dreams on Rothe diagrams and the initial ideal of
$I_{v,w}$}

\subsection{The prime decomposition theorem for $I_{v,w}$}
Let $LT_{\prec}(f)$ denote the leading term of $f$ under the order $\prec$.
The {\bf initial ideal}
\[{\rm in}_{\prec}I_{v,w}=\langle LT_{\prec}(f):f\in I_{v,w}\rangle\]
is by Theorem~\ref{thm:main1} a squarefree monomial ideal.  We will
use the {\bf Stanley--Reisner correspondence}, a bijective
correspondence between squarefree monomial ideals and simplicial
complexes.

Recall that given a simplicial complex $\Delta$ on a vertex set $A$ (so our faces are
subsets of $A$), the {\bf Stanley--Reisner ideal} $I(\Delta)$ is the squarefree
monomial ideal in
\[R=\Bbbk[x_a\mid a\in A]\]
generated by monomials corresponding to faces not in $\Delta$.  In symbols,
\[I(\Delta)=\left\langle \prod_{a\in F} x_a \mid F\not\in\Delta\right\rangle.\] The {\bf Stanley--Reisner ring} of $\Delta$ is defined to be $R/I(\Delta)$.
Conversely, given a squarefree monomial ideal $I$ in the ring
$R$, we can associate to it the {\bf
  Stanley--Reisner complex} $\Delta(I)$, which is the simplicial complex on $A$
such that
\[F\in\Delta(I) \mbox{\ if  \ } \prod_{a\in F} x_a\not\in I.\]

We now give a combinatorial description of a simplicial complex
$\Delta_{v,w}$, which we call the {\bf pipe complex}.  We will show
that $\Delta_{v,w}$ is the Stanley--Reisner complex of ${\rm in}_\prec
I_{v,w}$.

The {\bf canonical labeling} of $D(v)$ is obtained by filling the $t$ boxes in
row $i$ with the labels $i,i+1, \ldots,i+t-1$ (with $i$ being the label of the
leftmost box). We will interchangeably refer to a box in $D(v)$ with label
$\ell$ and a generator $u_{\ell}$ of the {\bf NilHecke algebra} ${\mathcal
  A}_{n}$.  The NilHecke algebra is the $\mathbb{C}$-algebra with generators
$u_1,\ldots,u_{n-1}$ and relations:
\begin{eqnarray}
\label{eqn:nilhecke}
u_i u_j& = & u_j u_i, \ |i-j|\geq 2;\\ \nonumber
u_i u_{i+1} u_i & = & u_{i+1} u_i u_{i+1},  \ \ 1\leq i\leq n-2; \\ \nonumber
u_i^2 & = & -u_i.  \nonumber
\end{eqnarray}
The NilHecke algebra $\mathcal{A}_n$ has
a basis $\{u_w\}$ indexed by permutations $w\in S_n$.  For the simple
transposition
\[s_i:=(i\leftrightarrow i+1),\]
let
\[u_{s_i}:=u_i,\]
and if
\[w=s_{i_1}\cdots s_{i_\ell}\]
is any reduced transposition for $w$, then let
\[u_w:=u_{i_1}\cdots u_{i_\ell}.\]
Multiplication in this basis by a generator
$u_i$ is as follows.
\[u_w u_i=\left\{\begin{array}{cc}
-u_w & \mbox{if  $ws_i<w$;}\\
u_{ws_i} & \mbox{otherwise.}
\end{array}\right.
\]

A {\bf pipe dream on $D(v)$} is a configuration ${\mathcal P}$ of crosses $+$
in a subset of boxes of $D(v)$.  Let ${\rm Pipes}(v)$ be the set of all such
pipe dreams. Associated to ${\mathcal P}$ is the {\bf Demazure product}
$\prod{\mathcal P}$ of generators $u_{\ell}$ obtained by reading the crosses
in rows, from left to right, and then from top to bottom.  Fix $w\in S_n$.
Following~\cite{Buch.Rimanyi}, a pipe dream ${\mathcal P}$ on $D(v)$ is
furthermore a {\bf pipe dream for $w$} if
\[\prod{\mathcal P}=\pm u_w\]
Such a pipe dream is {\bf reduced} if
\[\#{\mathcal
  P}=\ell(w).\]

Let
\[{\rm Pipes}(v,w)\subseteq {\rm Pipes}(v)\]
be the
collection of all pipe dreams on $D(v)$ for $w$, and let
\[{\rm
  RedPipes}(v,w)\subseteq {\rm Pipes}(v,w)\]
  be the subset of reduced
ones.  Finally, note that if ${\mathcal P}$ is the pipe dream with a $+$ in
every box of $D(v)$ then
\[\prod {\mathcal P}= u_{w_0v}.\]
This reformulates the well
known fact that the canonical labeling of $D(v)$ encodes, via the reading order
we use, a reduced word of $w_0 v$; see for example \cite[Remark~2.1.9]{Manivel}.

\begin{example}
\label{exa:31524}
Let $v=31524$ (in one line notation). Then the diagram $D(v)$ and the
associated canonical labeling is given below:
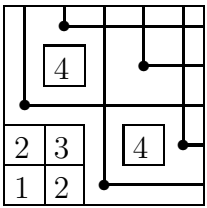
\[
\begin{picture}(150,75)
\put(37.5,0){\makebox[0pt][l]{\framebox(75,75)}}
\put(37.5,30){\line(1,0){30}}
\put(67.5,30){\line(0,-1){30}}
\put(52.5,30){\line(0,-1){30}}
\put(37.5,15){\line(1,0){30}}
\put(82.5,15){\makebox[0pt][l]{\framebox(15,15)}}
\put(52.5,45){\makebox[0pt][l]{\framebox(15,15)}}
\thicklines
\put(75,7.5){\circle*{4}}
\put(75,7.5){\line(1,0){37.5}}
\put(75,7.5){\line(0,1){67.5}}
\put(45,37.5){\circle*{4}}
\put(45,37.5){\line(1,0){67.5}}
\put(45,37.5){\line(0,1){37.5}}
\put(60,67.5){\circle*{4}}
\put(60,67.5){\line(1,0){52.5}}
\put(60,67.5){\line(0,1){7.5}}
\put(90,52.5){\circle*{4}}
\put(90,52.5){\line(1,0){22.5}}
\put(90,52.5){\line(0,1){22.5}}
\put(105,22.5){\circle*{4}}
\put(105,22.5){\line(1,0){7.5}}
\put(105,22.5){\line(0,1){52.5}}
\put(41,2.5){$1$}
\put(56,2.5){$2$}
\put(41,17.5){$2$}
\put(56,17.5){$3$}
\put(86,17.5){$4$}
\put(56,47.5){$4$}
\end{picture}
\]
The reader can check that the canonical labeling gives a reduced word
$s_4 s_2 s_3 s_4 s_1 s_2$ for $w_0 v$.

Let $w=13254=s_2 s_4 = s_4 s_2$. Then ${\rm RedPipes}(v,w)$ consists of:
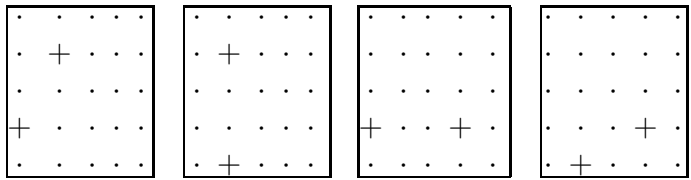
\[
\begin{picture}(240,60)
\put(-2,-4){\makebox[0pt][l]{\framebox(55,64)}}
\put(0,25){$\begin{matrix}
\!\cdot\! &\!\cdot\!&\!\cdot\!&\!\cdot\!&\!\cdot\!\\
\!\cdot\! & \!+\! &\!\cdot\!&\!\cdot\!&\!\cdot\!\\
\!\cdot\! &\!\cdot\!&\!\cdot\!&\!\cdot\!&\!\cdot\!\\
\!+\! &\!\cdot\!&\!\cdot\!&\!\cdot\!&\!\cdot\!\\
\!\cdot\! &\!\cdot\!&\!\cdot\!&\!\cdot\!&\!\cdot\!
\end{matrix}$}
\put(65,-4){\makebox[0pt][l]{\framebox(55,64)}}
\put(70,25){$\begin{matrix}
\!\cdot\! &\!\cdot\!&\!\cdot\!&\!\cdot\!&\!\cdot\!\\
\!\cdot\! & \!+\! &\!\cdot\!&\!\cdot\!&\!\cdot\!\\
\!\cdot\! &\!\cdot\!&\!\cdot\!&\!\cdot\!&\!\cdot\!\\
\!\cdot\! &\!\cdot\!&\!\cdot\!&\!\cdot\!&\!\cdot\!\\
\!\cdot\! &\!+\!&\!\cdot\!&\!\cdot\!&\!\cdot\!\\
\end{matrix}$}
\put(131,-4){\makebox[0pt][l]{\framebox(57,64)}}
\put(133,25){$
\begin{matrix}
\!\cdot\! &\!\cdot\!&\!\cdot\!&\!\cdot\!&\!\cdot\!\\
\!\cdot\! & \!\cdot\! &\!\cdot\!&\!\cdot\!&\!\cdot\!\\
\!\cdot\! &\!\cdot\!&\!\cdot\!&\!\cdot\!&\!\cdot\!\\
\!+\! &\!\cdot\!&\!\cdot\!&\!+\!&\!\cdot\!\\
\!\cdot\! &\!\cdot\!&\!\cdot\!&\!\cdot\!&\!\cdot\!\\
\end{matrix}$}
\put(200,-4){\makebox[0pt][l]{\framebox(55,64)}}
\put(203,25){$\begin{matrix}
\!\cdot\! &\!\cdot\!&\!\cdot\!&\!\cdot\!&\!\cdot\!\\
\!\cdot\! & \!\cdot\! &\!\cdot\!&\!\cdot\!&\!\cdot\!\\
\!\cdot\! &\!\cdot\!&\!\cdot\!&\!\cdot\!&\!\cdot\!\\
\!\cdot\! &\!\cdot\!&\!\cdot\!&\!+\!&\!\cdot\!\\
\!\cdot\! &\!+\!&\!\cdot\!&\!\cdot\!&\!\cdot\!\\
\end{matrix}$}
\end{picture}
\]
whereas ${\rm Pipes}(v,w)$ consists of the above pipe dreams together with:
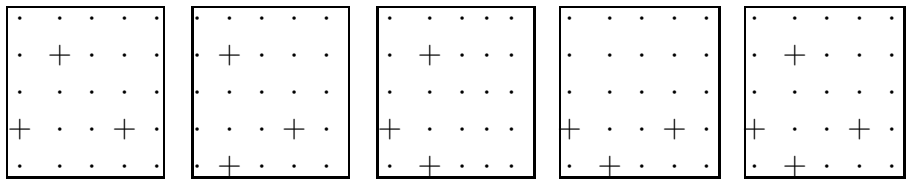
\[
\begin{picture}(320,60)
\put(-2,-4){\makebox[0pt][l]{\framebox(59,64)}}
\put(0,25){$\begin{matrix}
\!\cdot\! &\!\cdot\!&\!\cdot\!&\!\cdot\!&\!\cdot\!\\
\!\cdot\! & \!+\! &\!\cdot\!&\!\cdot\!&\!\cdot\!\\
\!\cdot\! &\!\cdot\!&\!\cdot\!&\!\cdot\!&\!\cdot\!\\
\!+\! &\!\cdot\!&\!\cdot\!&\!+\!&\!\cdot\!\\
\!\cdot\! &\!\cdot\!&\!\cdot\!&\!\cdot\!&\!\cdot\!\\
\end{matrix}$}
\put(68,-4){\makebox[0pt][l]{\framebox(59,64)}}
\put(70,25){$
\begin{matrix}
\!\cdot\! &\!\cdot\!&\!\cdot\!&\!\cdot\!&\!\cdot\!\\
\!\cdot\! & \!+\! &\!\cdot\!&\!\cdot\!&\!\cdot\!\\
\!\cdot\! &\!\cdot\!&\!\cdot\!&\!\cdot\!&\!\cdot\!\\
\!\cdot\! &\!\cdot\!&\!\cdot\!&\!+\!&\!\cdot\!\\
\!\cdot\! &\!+\!&\!\cdot\!&\!\cdot\!&\!\cdot\!\\
\end{matrix}$}
\put(138,-4){\makebox[0pt][l]{\framebox(59,64)}}
\put(140,25){$
\begin{matrix}
\!\cdot\! &\!\cdot\!&\!\cdot\!&\!\cdot\!&\!\cdot\!\\
\!\cdot\! & \!+\! &\!\cdot\!&\!\cdot\!&\!\cdot\!\\
\!\cdot\! &\!\cdot\!&\!\cdot\!&\!\cdot\!&\!\cdot\!\\
\!+\! &\!\cdot\!&\!\cdot\!&\!\cdot\!&\!\cdot\!\\
\!\cdot\! &\!+\!&\!\cdot\!&\!\cdot\!&\!\cdot\!\\
\end{matrix}$}
\put(207,-4){\makebox[0pt][l]{\framebox(60,64)}}
\put(208,25){$
\begin{matrix}
\!\cdot\! &\!\cdot\!&\!\cdot\!&\!\cdot\!&\!\cdot\!\\
\!\cdot\! & \!\cdot\! &\!\cdot\!&\!\cdot\!&\!\cdot\!\\
\!\cdot\! &\!\cdot\!&\!\cdot\!&\!\cdot\!&\!\cdot\!\\
\!+\! &\!\cdot\!&\!\cdot\!&\!+\!&\!\cdot\!\\
\!\cdot\! &\!+\!&\!\cdot\!&\!\cdot\!&\!\cdot\!\\
\end{matrix}$}
\put(277,-4){\makebox[0pt][l]{\framebox(60,64)}}
\put(278,25){$
\begin{matrix}
\!\cdot\! &\!\cdot\!&\!\cdot\!&\!\cdot\!&\!\cdot\!\\
\!\cdot\! & \!+\! &\!\cdot\!&\!\cdot\!&\!\cdot\!\\
\!\cdot\! &\!\cdot\!&\!\cdot\!&\!\cdot\!&\!\cdot\!\\
\!+\! &\!\cdot\!&\!\cdot\!&\!+\!&\!\cdot\!\\
\!\cdot\! &\!+\!&\!\cdot\!&\!\cdot\!&\!\cdot\!\\
\end{matrix}$}
\end{picture}
\]
Lastly, ${\rm Pipes}(v)$ consists of the above pipe dreams together with
all remaining $2^6-9=23$ pipe dreams whose Demazure product does not give $w$.
\qed \end{example}

We are now ready to
define the pipe complex $\Delta_{v,w}$.  The complex $\Delta_{v,w}$ is
a complex on the vertex set $D(v)$ such that $F$ is a face if $D(v)\setminus
F$ contains a (reduced) pipe dream for $w_0w$ ({\em not} $w$).  Following
\cite{Knutson.Miller:subword} it is
convenient for us to label the faces of $\Delta_{v,w}$ by pipe dreams, so that
a face $F$ will be labeled by the pipe dream $\mathcal{P}$ with crosses
everywhere except at the vertices of $F$.  This means that vertices are
labeled by pipe dreams with precisely $\ell(w_0v)-1$ $\cross$'s, and the
empty face is labeled by the pipe dream with a $\cross$ in every square of
$D(v)$.

Pipe dreams on $D(v)$ were introduced in \cite{Buch.Rimanyi} on purely
combinatorial grounds. Our main result about $\Delta_{v,w}$ is the following,
which gives a geometric rationale for these pipe dreams.  Moreover, we will
use this result in the next section to geometrically explain the
specialization formula of~\cite{Buch.Rimanyi} for Grothendieck polynomials (as
well as a generalization).

\begin{theorem}
\label{thm:prime}
For $v,w\in S_n$, $\Delta_{v,w}$ is the Stanley--Reisner complex of
${\rm in}_{\prec}I_{v,w}$.  Moreover,
the following prime decomposition holds:
\begin{equation}
\label{eqn:prime}
{\rm in}_{\prec} I_{v,w}= \bigcap_{{\mathcal P}\in
{\rm RedPipes}(v,w_0w)}\langle
z_{ij}\ | \ (i,j)\in {\mathcal P}\rangle.
\end{equation}
\end{theorem}

The second statement follows from the first by the general fact that
elements of the prime decomposition of a Stanley--Reisner ideal are
given by the facets of the Stanley--Reisner complex.  See for
example~\cite[Theorem 1.7]{Miller.Sturmfels}.

The proof, given in parallel with the proof of
Theorem~\ref{thm:main1}, is in Section~6.

\subsection{Pipe complexes are subword complexes}

We now state some facts about pipe complexes which will be needed in our
proof of Theorem~\ref{thm:prime}.  To do so, we need to recall some definitions
from~\cite{Knutson.Miller:subword}. Let
\[Q=(i_1,\ldots,i_{\ell})\]
be a sequence from the alphabet $\{1,2,\ldots,n-1\}$ such that if
$s_i=(i\leftrightarrow i+1)$ is the simple reflection in $S_n$ then
$s_{i_1}\cdots s_{i_{\ell}}$ is a reduced word for $v\in S_n$.  In
particular, $\ell=\ell(v)$.

Fix $w\in S_n$ and a word $Q$ in the simple generators
$\{s_1,\ldots,s_{n-1}\}$.  In \cite{Knutson.Miller:subword}, the {\bf subword
  complex} $\Delta(Q,w)$ is defined to be the simplicial complex where the
vertex set is the positions $1,2,\ldots,\ell$ of $Q$ with $F\subseteq
\{1,\ldots,\ell\}$ defined to be a facet if and only if $\prod Q\setminus F$,
which is defined as the product of simple roots corresponding to
the indices in the complement of $F$, is a reduced word for $w$.  More generally, faces
\[F\in \Delta(Q,w)\]
correspond to subwords of $Q$ such that the Demazure product
\[\prod Q\setminus F\geq w\]
in Bruhat order (here we ignore the sign coming from $s_i^2=-s_i$ in the definition of the Demazure product when making comparisons in Bruhat order).

The following is immediate from results of
\cite{Knutson.Miller:subword} and our definitions.

\begin{proposition}
\label{prop:subwordtranslated}
$\Delta_{v,w}$ is the subword complex $\Delta(Q,w_0 w)$ where $Q$ is the
canonical labeling of $D(v)$ in our reading order and therefore a reduced word
for $w_0 v$.  Hence $\Delta_{v,w}$ is shellable and homeomorphic to a ball or
a sphere. Moreover, the facets of $\Delta_{v,w}$ are labeled by ${\mathcal
  P}\in {\rm RedPipes}(v,w_0w)$ and interior faces labeled by ${\mathcal P}\in
{\rm Pipes}(v,w_0w)$.
\end{proposition}
\begin{proof}
From the definition of $\Delta_{v,w}$ and $\Delta(Q,w_0w)$ the first claim
amounts to the well-known fact that the canonical filling encodes a reduced
word for $w_0v$. The remainder are general properties of any subword
complex; see Theorems~2.5 and~3.7 in~\cite{Knutson.Miller:subword}.
\end{proof}

\begin{example}
\label{exa:31542previous}
Let $v=31542$. Then $D(v)$ and the canonical
labeling are given by:
\begin{figure}[h]
\begin{picture}(150,75)
\put(37.5,0){\makebox[0pt][l]{\framebox(75,75)}}
\put(37.5,30){\line(1,0){30}}
\put(67.5,30){\line(0,-1){30}}
\put(37.5,15){\line(1,0){30}}
\put(52.5,0){\line(0,1){30}}
\thicklines
\put(90,22.5){\circle*{4}}
\put(90,22.5){\line(1,0){22.5}}
\put(90,22.5){\line(0,1){52.5}}
\put(45,37.5){\circle*{4}}
\put(45,37.5){\line(1,0){67.5}}
\put(45,37.5){\line(0,1){37.5}}
\put(60,67.5){\circle*{4}}
\put(60,67.5){\line(1,0){52.5}}
\put(60,67.5){\line(0,1){7.5}}
\put(75,7.5){\circle*{4}}
\put(75,7.5){\line(1,0){37.5}}
\put(75,7.5){\line(0,1){67.5}}
\put(105,52.5){\circle*{4}}
\put(105,52.5){\line(1,0){7.5}}
\put(105,52.5){\line(0,1){22.5}}

\thinlines
\put(52.5,45){\line(1,0){15}}
\put(67.5,45){\line(0,1){15}}
\put(67.5,60){\line(-1,0){15}}
\put(52.5,45){\line(0,1){15}}
\put(41,2.5){$1$}
\put(41,17.5){$2$}
\put(56,2.5){$2$}
\put(56,17.5){$3$}
\put(56,47.5){$4$}
\end{picture}
\end{figure}
Let $w=53142$ and hence $w_0 w=13524=s_4 s_2 s_3$.
Then the pipe complex
$\Delta_{v,w}=\Delta_{31542,53142}$ is the
one dimensional ball of Figure~\ref{fig:31542.53142}.
\begin{figure}[h]
\begin{picture}(270,170)
\thicklines
\put(50,80){\circle*{4}}
\put(50,80){\line(1,0){160}}
\put(210,80){\circle*{4}}

\thinlines
\put(13,0){\makebox[0pt][l]{\framebox(55,64)}}
\thicklines
\put(15,30){$\begin{matrix}
\!\cdot\! &\!\cdot\!&\!\cdot\!&\!\cdot\!&\!\cdot\!\\
\!\cdot\! & \!+\! &\!\cdot\!&\!\cdot\!&\!\cdot\!\\
\!\cdot\! &\!\cdot\!&\!\cdot\!&\!\cdot\!&\!\cdot\!\\
\!+\! &\!+\!&\!\cdot\!&\!\cdot\!&\!\cdot\!\\
\!+\! &\!\cdot\!&\!\cdot\!&\!\cdot\!&\!\cdot\!\\
\end{matrix}$}

\thinlines
\put(188,0){\makebox[0pt][l]{\framebox(58,64)}}
\thicklines
\put(190,30){$\begin{matrix}
\!\cdot\! &\!\cdot\!&\!\cdot\!&\!\cdot\!&\!\cdot\!\\
\!\cdot\! & \!+\! &\!\cdot\!&\!\cdot\!&\!\cdot\!\\
\!\cdot\! &\!\cdot\!&\!\cdot\!&\!\cdot\!&\!\cdot\!\\
\!+\! &\!+\!&\!\cdot\!&\!\cdot\!&\!\cdot\!\\
\!\cdot\! &\!+\!&\!\cdot\!&\!\cdot\!&\!\cdot\!\\
\end{matrix}$}

\thinlines
\put(98,90){\makebox[0pt][l]{\framebox(60,64)}}
\thicklines
\put(100,120){$\begin{matrix}
\!\cdot\! &\!\cdot\!&\!\cdot\!&\!\cdot\!&\!\cdot\!\\
\!\cdot\! & \!+\! &\!\cdot\!&\!\cdot\!&\!\cdot\!\\
\!\cdot\! &\!\cdot\!&\!\cdot\!&\!\cdot\!&\!\cdot\!\\
\!+\! &\!+\!&\!\cdot\!&\!\cdot\!&\!\cdot\!\\
\!\cdot\! &\!\cdot\!&\!\cdot\!&\!\cdot\!&\!\cdot\!\\
\end{matrix}$}
\thinlines
\end{picture}
\caption{\label{fig:31542.53142} $\Delta_{31542,53142}$ is a one-dimensional ball}
\end{figure}
\qed \end{example}

\begin{example}\label{exa:31452}
Let $v=31452$ and $w=53142$. So $w_0 w = 13524=s_4 s_2 s_3$ as above. We have:
\[
\begin{picture}(150,75)
\put(37.5,0){\makebox[0pt][l]{\framebox(75,75)}}
\put(37.5,30){\line(1,0){30}}
\put(67.5,30){\line(0,-1){15}}
\put(52.5,30){\line(0,-1){30}}
\put(37.5,15){\line(1,0){45}}
\put(67.5,0){\line(0,1){15}}
\put(82.5,0){\line(0,1){15}}
\thicklines
\put(90,7.5){\circle*{4}}
\put(90,7.5){\line(1,0){22.5}}
\put(90,7.5){\line(0,1){67.5}}
\put(45,37.5){\circle*{4}}
\put(45,37.5){\line(1,0){67.5}}
\put(45,37.5){\line(0,1){37.5}}
\put(60,67.5){\circle*{4}}
\put(60,67.5){\line(1,0){52.5}}
\put(60,67.5){\line(0,1){7.5}}
\put(75,22.5){\circle*{4}}
\put(75,22.5){\line(1,0){37.5}}
\put(75,22.5){\line(0,1){52.5}}
\put(105,52.5){\circle*{4}}
\put(105,52.5){\line(1,0){7.5}}
\put(105,52.5){\line(0,1){22.5}}

\thinlines
\put(52.5,45){\line(1,0){15}}
\put(67.5,45){\line(0,1){15}}
\put(67.5,60){\line(-1,0){15}}
\put(52.5,45){\line(0,1){15}}
\put(41,2.5){$1$}
\put(41,17.5){$2$}
\put(56,2.5){$2$}
\put(56,17.5){$3$}
\put(71,2.5){$3$}
\put(56,47.5){$4$}
\end{picture}
\]

Therefore, $\Delta_{v,w}=\Delta_{31452,53142}$ is the complex given by
Figure~\ref{fig:31452.53142}.
\qed \end{example}

\begin{figure}[h]
\begin{picture}(550,330)

\put(32,262){\circle*{4}}
\put(224,262){\circle*{4}}
\put(128,70){\circle*{4}}
\put(320,70){\circle*{4}}
\put(416,262){\circle*{4}}

\thicklines
\put(128,70){\line(1,0){192}}
\put(128,70){\line(-1,2){96}}
\put(128,70){\line(1,2){40}}
\put(224,262){\line(-1,-2){22}}
\put(224,262){\line(1,-2){22}}

\put(320,70){\line(-1,2){40}}
\put(320,70){\line(1,2){96}}
\put(32,263){\line(1,0){384}}

\thinlines
\put(83,190){\makebox[0pt][l]{\framebox(55,64)}}
\thicklines
\put(85,218){$\begin{matrix}
\!\cdot\! &\!\cdot\!&\!\cdot\!&\!\cdot\!&\!\cdot\!\\
\!\cdot\! & \!+\! &\!\cdot\!&\!\cdot\!&\!\cdot\!\\
\!\cdot\! &\!\cdot\!&\!\cdot\!&\!\cdot\!&\!\cdot\!\\
\!+\! &\!+\!&\!\cdot\!&\!\cdot\!&\!\cdot\!\\
\!\cdot\! &\!\cdot\!&\!\cdot\!&\!\cdot\!&\!\cdot\!\\
\end{matrix}$}

\thinlines
\put(98,-7){\makebox[0pt][l]{\framebox(58,64)}}
\thicklines
\put(100,22){$\begin{matrix}
\!\cdot\! &\!\cdot\!&\!\cdot\!&\!\cdot\!&\!\cdot\!\\
\!\cdot\! & \!+\! &\!\cdot\!&\!\cdot\!&\!\cdot\!\\
\!\cdot\! &\!\cdot\!&\!\cdot\!&\!\cdot\!&\!\cdot\!\\
\!+\! &\!+\!&\!\cdot\!&\!\cdot\!&\!\cdot\!\\
\!+\! &\!\cdot\!&\!+\!&\!\cdot\!&\!\cdot\!\\
\end{matrix}$}

\thinlines
\put(298,-7){\makebox[0pt][l]{\framebox(59,64)}}
\thicklines
\put(300,22){$\begin{matrix}
\!\cdot\! &\!\cdot\!&\!\cdot\!&\!\cdot\!&\!\cdot\!\\
\!\cdot\! & \!+\! &\!\cdot\!&\!\cdot\!&\!\cdot\!\\
\!\cdot\! &\!\cdot\!&\!\cdot\!&\!\cdot\!&\!\cdot\!\\
\!+\! &\!\cdot\!&\!\cdot\!&\!\cdot\!&\!\cdot\!\\
\!+\! &\!+\!&\!+\!&\!\cdot\!&\!\cdot\!\\
\end{matrix}$}

\thinlines
\put(312,189){\makebox[0pt][l]{\framebox(55,64)}}
\thicklines
\put(315,218){$\begin{matrix}
\!\cdot\! &\!\cdot\!&\!\cdot\!&\!\cdot\!&\!\cdot\!\\
\!\cdot\! & \!+\! &\!\cdot\!&\!\cdot\!&\!\cdot\!\\
\!\cdot\! &\!\cdot\!&\!\cdot\!&\!\cdot\!&\!\cdot\!\\
\!\cdot\! &\!\cdot\!&\!\cdot\!&\!\cdot\!&\!\cdot\!\\
\!\cdot\! &\!+\!&\!+\!&\!\cdot\!&\!\cdot\!\\
\end{matrix}$}

\thinlines
\put(198,78){\makebox[0pt][l]{\framebox(59,64)}}
\thicklines
\put(200,107){$\begin{matrix}
\!\cdot\! &\!\cdot\!&\!\cdot\!&\!\cdot\!&\!\cdot\!\\
\!\cdot\! & \!+\! &\!\cdot\!&\!\cdot\!&\!\cdot\!\\
\!\cdot\! &\!\cdot\!&\!\cdot\!&\!\cdot\!&\!\cdot\!\\
\!+\! &\!\cdot\!&\!\cdot\!&\!\cdot\!&\!\cdot\!\\
\!\cdot\! &\!\cdot\!&\!+\!&\!\cdot\!&\!\cdot\!\\
\end{matrix}$}

\thinlines
\put(153,151){\makebox[0pt][l]{\framebox(59,64)}}
\thicklines
\put(155,180){$\begin{matrix}
\!\cdot\! &\!\cdot\!&\!\cdot\!&\!\cdot\!&\!\cdot\!\\
\!\cdot\! & \!+\! &\!\cdot\!&\!\cdot\!&\!\cdot\!\\
\!\cdot\! &\!\cdot\!&\!\cdot\!&\!\cdot\!&\!\cdot\!\\
\!+\! &\!+\!&\!\cdot\!&\!\cdot\!&\!\cdot\!\\
\!\cdot\! &\!\cdot\!&\!+\!&\!\cdot\!&\!\cdot\!\\
\end{matrix}$}

\thinlines
\put(238,151){\makebox[0pt][l]{\framebox(59,64)}}
\thicklines
\put(240,180){$\begin{matrix}
\!\cdot\! &\!\cdot\!&\!\cdot\!&\!\cdot\!&\!\cdot\!\\
\!\cdot\! & \!+\! &\!\cdot\!&\!\cdot\!&\!\cdot\!\\
\!\cdot\! &\!\cdot\!&\!\cdot\!&\!\cdot\!&\!\cdot\!\\
\!+\! &\!\cdot\!&\!\cdot\!&\!\cdot\!&\!\cdot\!\\
\!\cdot\! &\!+\!&\!+\!&\!\cdot\!&\!\cdot\!\\
\end{matrix}$}

\thinlines
\put(10,124){\makebox[0pt][l]{\framebox(55,64)}}
\thicklines
\put(12,153){$\begin{matrix}
\!\cdot\! &\!\cdot\!&\!\cdot\!&\!\cdot\!&\!\cdot\!\\
\!\cdot\! & \!+\! &\!\cdot\!&\!\cdot\!&\!\cdot\!\\
\!\cdot\! &\!\cdot\!&\!\cdot\!&\!\cdot\!&\!\cdot\!\\
\!+\! &\!+\!&\!\cdot\!&\!\cdot\!&\!\cdot\!\\
\!+\! &\!\cdot\!&\!\cdot\!&\!\cdot\!&\!\cdot\!\\
\end{matrix}$}

\thinlines
\put(198,-6){\makebox[0pt][l]{\framebox(60,64)}}
\thicklines
\put(200,23){$\begin{matrix}
\!\cdot\! &\!\cdot\!&\!\cdot\!&\!\cdot\!&\!\cdot\!\\
\!\cdot\! & \!+\! &\!\cdot\!&\!\cdot\!&\!\cdot\!\\
\!\cdot\! &\!\cdot\!&\!\cdot\!&\!\cdot\!&\!\cdot\!\\
\!+\! &\!\cdot\!&\!\cdot\!&\!\cdot\!&\!\cdot\!\\
\!+\! &\!\cdot\!&\!+\!&\!\cdot\!&\!\cdot\!\\
\end{matrix}$}

\thinlines
\put(385,124){\makebox[0pt][l]{\framebox(60,64)}}
\thicklines
\put(387,153){$\begin{matrix}
\!\cdot\! &\!\cdot\!&\!\cdot\!&\!\cdot\!&\!\cdot\!\\
\!\cdot\! & \!+\! &\!\cdot\!&\!\cdot\!&\!\cdot\!\\
\!\cdot\! &\!\cdot\!&\!\cdot\!&\!\cdot\!&\!\cdot\!\\
\!\cdot\! &\!\cdot\!&\!\cdot\!&\!\cdot\!&\!\cdot\!\\
\!+\! &\!+\!&\!+\!&\!\cdot\!&\!\cdot\!\\
\end{matrix}$}

\thinlines
\put(103,271){\makebox[0pt][l]{\framebox(56,64)}}
\thicklines
\put(105,300){$\begin{matrix}
\!\cdot\! &\!\cdot\!&\!\cdot\!&\!\cdot\!&\!\cdot\!\\
\!\cdot\! & \!+\! &\!\cdot\!&\!\cdot\!&\!\cdot\!\\
\!\cdot\! &\!\cdot\!&\!\cdot\!&\!\cdot\!&\!\cdot\!\\
\!+\! &\!+\!&\!\cdot\!&\!\cdot\!&\!\cdot\!\\
\!\cdot\! &\!+\!&\!\cdot\!&\!\cdot\!&\!\cdot\!\\
\end{matrix}$}

\thinlines
\put(198,271){\makebox[0pt][l]{\framebox(60,64)}}
\thicklines
\put(200,300){$\begin{matrix}
\!\cdot\! &\!\cdot\!&\!\cdot\!&\!\cdot\!&\!\cdot\!\\
\!\cdot\! & \!+\! &\!\cdot\!&\!\cdot\!&\!\cdot\!\\
\!\cdot\! &\!\cdot\!&\!\cdot\!&\!\cdot\!&\!\cdot\!\\
\!+\! &\!+\!&\!\cdot\!&\!\cdot\!&\!\cdot\!\\
\!\cdot\! &\!+\!&\!+\!&\!\cdot\!&\!\cdot\!\\
\end{matrix}$}

\thinlines
\put(301,271){\makebox[0pt][l]{\framebox(58,64)}}
\thicklines
\put(305,300){$\begin{matrix}
\!\cdot\! &\!\cdot\!&\!\cdot\!&\!\cdot\!&\!\cdot\!\\
\!\cdot\! & \!+\! &\!\cdot\!&\!\cdot\!&\!\cdot\!\\
\!\cdot\! &\!\cdot\!&\!\cdot\!&\!\cdot\!&\!\cdot\!\\
\!\cdot\! &\!+\!&\!\cdot\!&\!\cdot\!&\!\cdot\!\\
\!\cdot\! &\!+\!&\!+\!&\!\cdot\!&\!\cdot\!\\
\end{matrix}$}

\thinlines
\put(10,271){\makebox[0pt][l]{\framebox(57,64)}}
\thicklines

\put(12,300){$\begin{matrix}
\!\cdot\! &\!\cdot\!&\!\cdot\!&\!\cdot\!&\!\cdot\!\\
\!\cdot\! & \!+\! &\!\cdot\!&\!\cdot\!&\!\cdot\!\\
\!\cdot\! &\!\cdot\!&\!\cdot\!&\!\cdot\!&\!\cdot\!\\
\!+\! &\!+\!&\!\cdot\!&\!\cdot\!&\!\cdot\!\\
\!+\! &\!+\!&\!\cdot\!&\!\cdot\!&\!\cdot\!\\
\end{matrix}$}

\thinlines
\put(385,271){\makebox[0pt][l]{\framebox(59,64)}}
\thicklines
\put(387,300){$\begin{matrix}
\!\cdot\! &\!\cdot\!&\!\cdot\!&\!\cdot\!&\!\cdot\!\\
\!\cdot\! & \!+\! &\!\cdot\!&\!\cdot\!&\!\cdot\!\\
\!\cdot\! &\!\cdot\!&\!\cdot\!&\!\cdot\!&\!\cdot\!\\
\!\cdot\! &\!+\!&\!\cdot\!&\!\cdot\!&\!\cdot\!\\
\!+\! &\!+\!&\!+\!&\!\cdot\!&\!\cdot\!\\
\end{matrix}$}
\end{picture}
\caption{\label{fig:31452.53142} $\Delta_{31452,53142}$ is a two dimensional ball}
\end{figure}
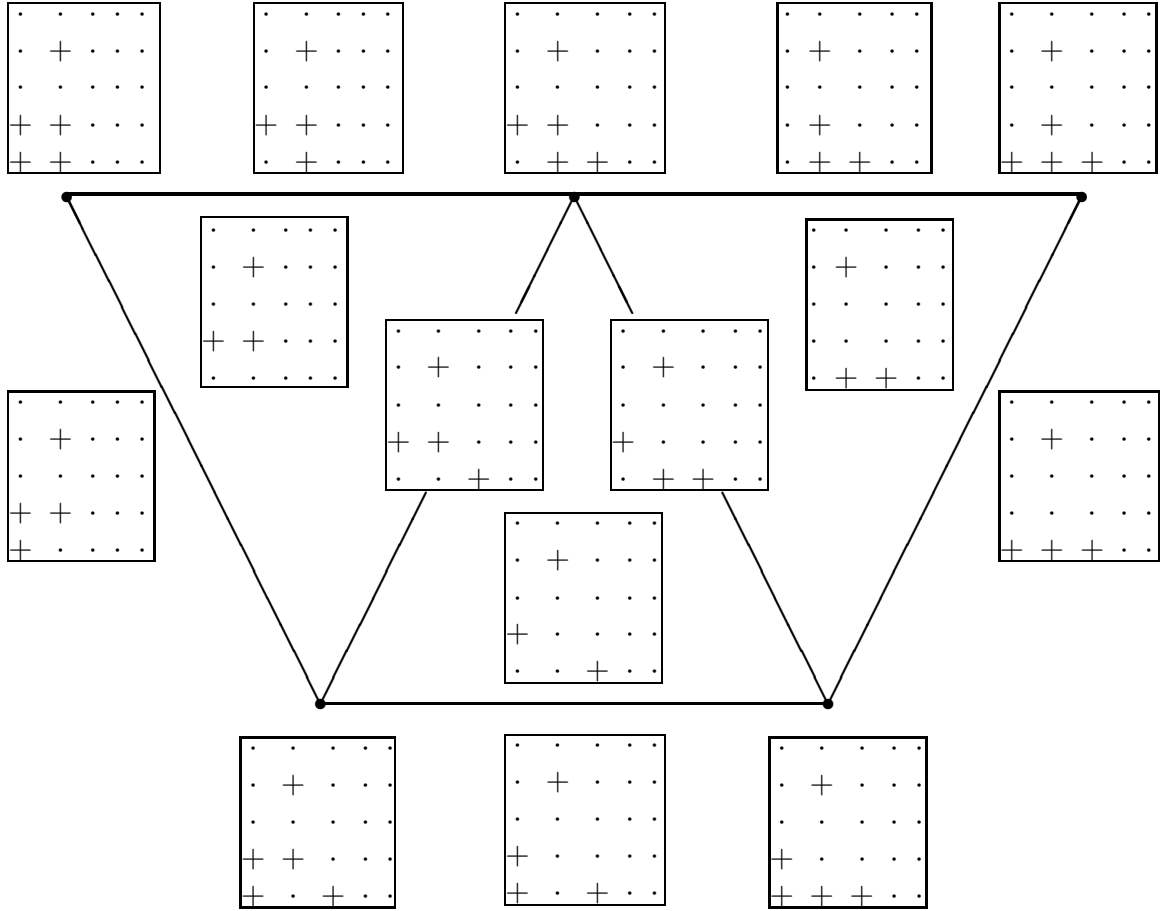

\subsection{The interlacing strands of pipe dreams}

Theorem~\ref{thm:prime} explains the geometric naturality of drawing pipe
dreams as we do rather than considering them only abstractly as subwords of a
fixed reduced word. This is closer to the point of view of
\cite{Knutson.Miller} or \cite{KMY} than of \cite{Knutson.Miller:subword} or
\cite{Knutson:patches}.

The proposition below describes a further combinatorial property of pipe dreams in
${\rm Pipes}(v,w)$, also not
seen at the subword level, that furthermore justifies the nomenclature in
relation to graphical objects of \cite{Fomin.Kirillov}, so named {\bf RC
  graphs} in \cite{Billey.Bergeron} and renamed {\bf pipe dreams} in
\cite{Knutson.Miller}. In those earlier papers, one tiles each square on an
$n\times n$ grid by crosses $\cross$ and elbows $\ \ \ \rj \ \ \ $, resulting
in a collection of strands which visibly encode a
permutation. We explain how to
similarly introduce strands into our pipe dreams:
Given the diagram $D(v)$ of $v$, let ${\tt flatten}(D(v))$ denote the {\bf
flattening} of $D(v)$, which is defined by compressing the squares in each
column southward past all non-diagram squares. This fixes an obvious bijection
between $D(v)$ and ${\tt flatten}(D(v))$. If ${\mathcal P}\in {\rm Pipes}(v,w)$,
we construct a new pipe dream
${\hat {\mathcal P}}$ by placing a cross $+$ in each square of ${\tt
flatten}(D(v))$ if and only if a cross $+$ appears in the corresponding square
of $D(v)$. Define ${\overline {\mathcal P}}$ by
placing an elbow $\ \ \ \ \rj\ \ \ \ $
in all other $1\times 1$ squares of the $n\times
n$ grid.

\begin{proposition}
\label{prop:strand}
If ${\mathcal P}\in {\rm RedPipes}(v,w)$, then ${\overline {\mathcal P}}$
consists of pipes such that the strand that starts in column $i$ ends in row
$w(i)$ (as counted from the bottom). In addition, no two strands cross more
than once, and crosses $+$ occur in the strict lower triangular part.
\end{proposition}
\begin{proof}
Consider the configuration ${\overline {\mathcal P}}$ on the $n\times n$
grid. Adorn said grid with the canonical labeling, so $1,2,\ldots,n$
will line the bottom-most row, then $2,3,\ldots,n-1$ will line the
next row, etc. Now note that ${\tt flatten}$ sends each $+$ in
${\mathcal P}$ to a $+$ in ${\overline {\mathcal P}}$ such that the
associated labeling under the canonical labelings of $D(v)$ and $n\times n$
is the same. Moreover, the reading words of both pipe dreams is the same.
Finally, notice that there are at most $n-i$ boxes of $D(v)$ in column $i$.
Hence ${\overline {\mathcal P}}$ must have all of its $+$'s in the strict
lower triangular part of $n\times n$. Therefore,
it follows from the discussion found in Section~5 (and
specifically in Example~5.1)
of \cite{Knutson.Miller:subword} that
${\overline {\mathcal P}}$ is a pipe dream for $w$ in the sense of \cite{Fomin.Kirillov}, and the proposition follows.
\end{proof}

\begin{example}
\label{exa:crosseswithpipes}
Below are the facets ${\mathcal P}\in {\rm RedPipes}(31452,13524)$ of
the pipe complex from Figure~\ref{fig:31452.53142}, and their corresponding
flattenings ${\overline {\mathcal P}}$.
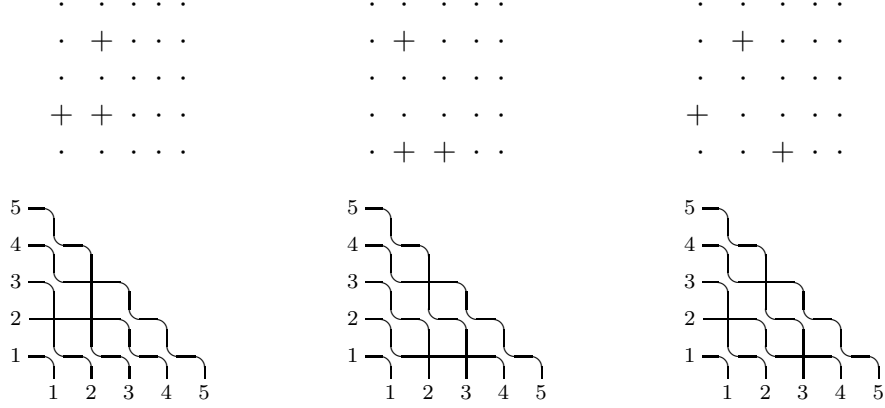
\[
\begin{matrix}
\!\cdot\! &\!\cdot\!&\!\cdot\!&\!\cdot\!&\!\cdot\!\\
\!\cdot\! & \!+\! &\!\cdot\!&\!\cdot\!&\!\cdot\!\\
\!\cdot\! &\!\cdot\!&\!\cdot\!&\!\cdot\!&\!\cdot\!\\
\!+\! &\!+\!&\!\cdot\!&\!\cdot\!&\!\cdot\!\\
\!\cdot\! &\!\cdot\!&\!\cdot\!&\!\cdot\!&\!\cdot\!\\
\end{matrix} \ \ \ \ \ \ \ \ \ \ \ \ \ \ \ \ \ \ \ \ \ \ \ \
\begin{matrix}
\!\cdot\! &\!\cdot\!&\!\cdot\!&\!\cdot\!&\!\cdot\!\\
\!\cdot\! & \!+\! &\!\cdot\!&\!\cdot\!&\!\cdot\!\\
\!\cdot\! &\!\cdot\!&\!\cdot\!&\!\cdot\!&\!\cdot\!\\
\!\cdot\! &\!\cdot\!&\!\cdot\!&\!\cdot\!&\!\cdot\!\\
\!\cdot\! &\!+\!&\!+\!&\!\cdot\!&\!\cdot\!\\
\end{matrix} \ \  \ \ \ \ \ \ \ \ \ \ \ \ \ \ \ \ \ \ \ \ \ \
\begin{matrix}
\!\cdot\! &\!\cdot\!&\!\cdot\!&\!\cdot\!&\!\cdot\!\\
\!\cdot\! & \!+\! &\!\cdot\!&\!\cdot\!&\!\cdot\!\\
\!\cdot\! &\!\cdot\!&\!\cdot\!&\!\cdot\!&\!\cdot\!\\
\!+ &\!\cdot\!&\!\cdot\!&\!\cdot\!&\!\cdot\!\\
\!\cdot\! &\!\cdot\!&\!+\!&\!\cdot\!&\!\cdot\!\\
\end{matrix}\]
\[
    \begin{array}{cccccccccc}
    \petit5 &   \ej   &        &       &        &         \\
    \petit4 &   \rj   &   \ej  &       &        &         \\
    \petit3 &   \rj   &   \+   &   \ej &        &         \\
    \petit2 &   \+    &   \+   &   \rj &  \ej   &         \\
    \petit1 &   \rj   &   \rj  &    \rj &  \rj   &  \ej    \\
            &\perm1{}&\perm2{}&\perm3{}&\perm4{}&\perm5{} \\
    \end{array} \ \ \ \ \ \ \ \ \ \ \ \ \ \
    \begin{array}{cccccccccc}
    \petit5 &   \ej   &        &       &        &         \\
    \petit4 &   \rj   &   \ej  &       &        &         \\
    \petit3 &   \rj   &   \+   &   \ej &        &         \\
    \petit2 &   \rj    &   \rj  &   \rj &  \ej   &         \\
    \petit1 &   \rj   &   \+  &    \+ &  \rj   &  \ej    \\
            &\perm1{}&\perm2{}&\perm3{}&\perm4{}&\perm5{} \\
    \end{array} \ \ \ \ \ \ \ \ \ \ \ \ \ \
    \begin{array}{cccccccccc}
    \petit5 &   \ej   &        &       &        &         \\
    \petit4 &   \rj   &   \ej  &       &        &         \\
    \petit3 &   \rj   &   \+   &   \ej &        &         \\
    \petit2 &   \+    &   \rj  &   \rj &  \ej   &         \\
    \petit1 &   \rj   &   \rj  &    \+ &  \rj   &  \ej    \\
            &\perm1{}&\perm2{}&\perm3{}&\perm4{}&\perm5{} \\
    \end{array}
  \]
\qed \end{example}
In the case
${\mathcal P}\in {\rm RedPipes}(w_0\star w_0, {\hat w})$,
flattening $D(w_0\star w_0)$ does nothing:
\[{\tt flatten}(D(w_0\star w_0))=D(w_0\star w_0),\]
and ${\hat {\mathcal P}}={\mathcal P}$. Then these ${\mathcal P}$ are
precisely all the pipe dreams for $w$ (after rotation and reflection to match
conventions) in the sense of \cite{Fomin.Kirillov}.

\section{(Un)specializing Grothendieck polynomials}

\subsection{Four cases of torus actions and their weights}
\label{subsect:actions}

We describe some torus actions on ${\rm Fun}[\Omega_v^{\circ}]\cong
{\mathbb C}[{\bf z}^{(v)}]$ we use in this paper.

First and more importantly, the action of $T\cong ({\mathbb C}^{\star})^n$ on
${\rm Flags}({\mathbb C}^n)$ induces what we will call the {\bf usual action}.
This action is the left action of diagonal matrices on $B$-cosets of $G$
written in our coordinates.  The action rescales rows independently and
rescales columns dependently, as upon rescaling a row one must rescale a
corresponding column to ensure there is a $1$ in position $(n-v(j)+1,j)$ (as
read with our usual upside-down matrix coordinates).  Adopting the usual
convention that the homomorphism picking out the $i$-th diagonal entry is the
weight $t_i$ and writing weights additively, this action gives the matrix
entry at $(i,j)$ the weight $t_{n-i+1}-t_{v(j)}$.  The variable $z_{ij}$ is
the coordinate function on this matrix entry and therefore (the torus action
on the variable) has weight
\[{\rm wt}(z_{ij})=t_{v(j)}-t_{n-i+1}.\]
The weight of $z_{ij}$ is always a positive root, and hence a
positive integer linear combination of the positive simple roots
$t_i-t_{i+1}$.  This action descends to the coordinate ring of ${\mathcal
  N}_{v,w}$.

Second, in the special case of the usual action
where
$v=w_0 \star w_0$
(as defined after
Corollary~\ref{cor:Knutson.Miller}), this usual action can be thought of as
independently rescaling the bottom $n$ rows and the leftmost $n$ columns,
while dependently scaling the remaining rows and columns to fix the matrix
outside of the southwest $n\times n$ block.  It is convenient in this case to
relabel the weights by
\[x_j=t_{n-j+1}=t_{v(j)}, \mbox{ \ and $y_i=t_{2n+1-i}$.}\]
This gives each variable $z_{ij}$ the weight
\[{\rm wt}(z_{ij})=x_j-y_i.\]
In this case, as explained by
the proof of Corollary~\ref{cor:Knutson.Miller}, the Kazhdan--Lusztig variety
${\mathcal N}_{w_0 \star w_0,{\hat w}}$ is the matrix Schubert variety ${\overline X}_{w}$
defined in Section~2.3.  We will see
that the multidegree and $K$-polynomial of the coordinate ring of ${\mathcal
  N}_{w_0 \star w_0,{\hat w}}$ with respect to this grading and weight
labeling are respectively the double Schubert and double Grothendieck
polynomials of \cite{Lascoux.Schutzenberger1, Lascoux.Schutzenberger2}.

Third, there is the action of ${\mathbb C}^{\star}$ that equally rescales
each variable in ${\mathbb C}[{\bf z}^{(v)}]$ with the same weight $t$;
this is the {\bf dilation action}.  This action fixes ${\mathcal N}_{v,w}$
if $I_{v,w}$ is homgeneous under the standard grading that assigns each
variable $z_{ij}$
\[{\rm wt}(z_{ij})=1.\]
In this case we say that $I_{v,w}$ is {\bf
  standardly homogeneous}.  As was pointed out to us in a private communication
by A.~Knutson, this automatically happens if there exists a
coweight $\lambda$ for which $\langle t_{v(j)}-t_{n-i+1}, \lambda\rangle=1$ for
all $i,j$ where $z_{ij}$ is a variable in ${\mathbb C}[{\bf z}^{(v)}]$, in
which case we say that $w_0v$ is {\bf $\lambda$-cominuscule}.  Note that this
condition does not depend on $w$.  If we take into account $w$, there are
other cases for which $\mathcal{N}_{v,w}$ is fixed by the dilation action, but
we know of no useful characterization; see Section~5.

Fourth and finally, there is the {\bf rescaling action} of $({\mathbb
  C}^{\star})^{\ell(w_0v)}$ that independently rescales each variable $z_{ij}$
with weight
\[{\rm wt}(z_{ij})=t_{ij}.\]
This action preserves only unions of
coordinate subspaces (and other monomial subschemes in our coordinates).

\subsection{Variously graded $K$-polynomials and multidegrees}
\label{subsection:various}
We now use some notions from combinatorial commutative algebra which
can be found in the textbook \cite{Miller.Sturmfels}.

Consider a polynomial ring
\[R=\mathbb{C}[z_1,\ldots,z_m]\]
with a grading such
that $z_i$ has some degree $\mathbf{a}_i\in\mathbb{Z}^N$.  A finitely graded $R$-module
\[M=\bigoplus_{{\bf v}\in {\mathbb Z}^N} M_{{\bf v}}\]
over $R$ has a free resolution
\[E_{\bullet}:0\leftarrow M \leftarrow E_0\leftarrow E_1\leftarrow\cdots\leftarrow E_L\leftarrow 0\]
where
\[E_j=\bigoplus_{i=1}^{\beta_j}R(-\mathbf{d}_{ij})\]
is graded with the $j$-th summand of $E_i$ generated in degree
$\mathbf{d}_{ij}\in {\mathbb Z}^N$.

Then the (${\mathbb Z}^N$-graded) {\bf K-polynomial} of $M$ is
\[{\mathcal K}(M,{\bf t})=\sum_j(-1)^j\sum_i {\bf t}^{{\bf d}_{ij}}.\]
In any case where $R$ is {\bf positively graded},
meaning that the $\mathbf{a}_i$
generate a pointed cone in $\mathbb{Z}^N$, ${\mathcal K}(M,{\bf t})$ is the
numerator of the ${\mathbb Z}^N$-graded Hilbert series:
\[{\rm Hilb}(M,{\bf t})=\frac{{\mathcal K}(M,{\bf t})}
{\prod_{i}(1-\mathbf{t}^{\mathbf{a}_i})}.\]
The {\bf multidegree} ${\mathcal C}(M, {\bf t})$
is by definition the sum of the lowest degree terms of
${\mathcal K}(M,{\bf 1-t})$.  (This means we substitute $1-t_k$ for $t_k$ for
all $k$, $1<k<N$.)

In Section~4.3, the geometric underpinnings of
(un)specializing Grothendieck and Schubert polynomials are explained
in terms of inclusions of tori. We therefore present a discussion of the
necessary background now.

Suppose our grading comes from a group action of $(\mathbb{C}^*)^N$ on
$R$; this means that the grading group $\mathbb{Z}^N$ is
identified with the
weight lattice of $(\mathbb{C}^*)^N$, and $\mathbf{a}_i$ is the weight
(written additively) of the action on $z_i$.  In this case, a quotient ring
$R/I$ is a homogeneous $R$-module under our grading if and only if the affine
variety (or scheme) $V(I)$ is fixed by the $(\mathbb{C}^*)^N$ action.
Furthermore ${\mathcal K}(R/I,{\bf t})$ is now the equivariant $K$-theory
class $[{\mathcal O}_{V(I)}]\in K_T^0(\mathbb{C}^m)$, whereas
${\mathcal C}(R/I,{\bf
  t})$ is the equivariant cohomology class $[V(I)]\in H_T(\mathbb{C}^m)$.
Note that, while we wrote our weights additively in describing the degree
$\mathbf{a}_i$ given to the variable $z_i$, in the $K$-polynomial the weights
are ``exponentiated'' and written as $\mathbf{t}^{\mathbf{a}_i}$.  Since
$\mathbb{C}^m$ is contractible to a point, $K_T^0(\mathbb{C}^m)$ can be
identified with the $K_T$-ring of a point, which is the representation ring of
$T=(\mathbb{C}^*)^N$.  This ring is isomorphic to the Laurent polynomial ring
in $k$ variables, and weights are multiplicative in this ring.  On the other
hand, essentially due to \cite[Prop. 8.49]{Miller.Sturmfels}, weights are
additive in the cohomology ring.

Suppose moreover that we have tori
\[T_1=(\mathbb{C}^*)^M \mbox{ \ and
$T_2=(\mathbb{C}^*)^N$}\]
acting on $R$, with a map of tori
\[\rho: T_1\rightarrow T_2\]
which is compatible with this action.  Here, compatibility means that
$t\cdot f=\rho(t)\cdot f$ for all $t\in T_1$ and $f\in S$.  The map of
tori induces a map $\rho^*$ from weights of $T_2$ to weights of $T_1$.
This in turn induces a map
\[K_{T_2}^0(\mathbb{C}^m)\rightarrow K_{T_1}^0(\mathbb{C}^m).\]
On the level of
$K$-polynomials, one can obtain the $K$-polynomial
$\mathcal{K}_{T_1}(M,\mathbf{s})$ with respect to the grading from $T_1$ from
the $K$-polynomial $\mathcal{K}_{T_2}(M,\mathbf{t})$ by substituting
$\mathbf{s}^{\rho^*(e_i)}$ for each $t_i$.  (Here $e_i$ is the vector with $1$
in the $i$-th coordinate and $0$ elsewhere.)  On the other hand, weights are
additive in cohomology, so one obtains the multidegree
$\mathcal{C}_{T_1}(M,\mathbf{s})$ from $\mathcal{C}_{T_2}(M,\mathbf{t})$ by
substituting $\langle \rho^*(e_i),\mathbf{s}\rangle$ for $\mathbf{t}_i$.

In Section~4.3 below, $R$ is the ring ${\mathbb C}[{\bf z}^{(v)}]$, considered
as the coordinate ring of the affine space $\Omega_{v}^{\circ}\cong
{\mathbb C}^{\ell(w_0 v)}$, and the torus
action will be one of those specified in the previous section.  Note that there
is an embedding of the (torus for) the usual action into the rescaling action,
and, in the $\lambda$-cominuscule case, an embedding (by the
coweight $\lambda$) of the dilation action into the usual action.

\subsection{Unspecialized Grothendieck polynomials as $K$-polynomials}

Recall the NilHecke algebra ${\mathcal A}_n$ defined by (\ref{eqn:nilhecke}) in
Section~2.
Consider the following generating series
\[{\widetilde \Groth}_{v}=\prod_{(i,j)\in D(v)}(1+u_{{\tt label}(i,j)}(1-t_{ij}))
\in {\mathcal A}_{n}[\{t_{ij}\}],\] where ${\tt label}(i,j)$ is the label in
box $(i,j)$ (in the $i$-th row from the bottom and $j$-th column from the left
in accordance with our convention) of the canonical labeling of $D(v)$ and
where the product is taken from left to right along rows and from top to
bottom (in accordance with the reading of the canonical labeling).  Similar
generating series were considered in \cite{Fomin.Kirillov, Buch.Rimanyi}.

Now define the {\bf unspecialized Grothendieck polynomial} by
\[\Groth_{v,w}(t_{11},t_{12},\ldots\ )
:=\mbox{coefficient of $u_{w_0 w}$ in \ } \Groth_{v}.\]
It is clear from the construction that
\begin{equation}
\label{eqn:pipeformulaG}
\Groth_{v,w}(t_{11},t_{12},\ldots\ )
=\sum_{{\mathcal P}\in {\rm Pipes}(v,w_0w)}(-1)^{\#{\mathcal P}-\ell(w_0 w)}
({\bf 1-t}^{\mathcal P}),
\end{equation}
where
\[{\bf 1-t}^{\mathcal P}=\prod_{(i,j) \mbox{\ contains a $+$}} (1-t_{ij}).\]
We furthermore define the {\bf unspecialized Schubert polynomial} by
\begin{equation}
\label{eqn:pipeformulaS}
\Schub_{v,w}=\sum_{{\mathcal P}\in
{\rm RedPipes}(v,w_0w)}{\bf t}^{\mathcal P}.
\end{equation}

The following result interprets $\Groth_{v,w}({\bf t})$ and
$\Schub_{v,w}({\bf t})$ in terms of the pipe complex $\Delta_{v,w}$:

\begin{proposition}
\label{prop:unspecializedequalsK}
If $R={\mathbb C}[{\bf z}^{(v)}]$ then
$\Groth_{v,w}({\bf t})$ and $\Schub_{v,w}({\bf t})$ are respectively the
$K$-polynomial and multidegree of $R/{\rm in}_\prec I_{v,w}$ under the
rescaling action.  Equivalently,
they are the $K$-polynomial and multidegree of $R/K_{v,w}$, the
Stanley--Reisner ring of $\Delta_{v,w}$.  Furthermore, $\Schub_{v,w}$ is the
lowest degree term of $\Groth_{v,w}(\mathbf{1-t})$.  (This is the result of
substituting $1-t_{ij}$ for $t_{ij}$ for all $i$ and $j$.)
\end{proposition}
\begin{proof}
From Proposition~\ref{prop:subwordtranslated}, we know $\Delta_{v,w}$ is a
ball or sphere.  It is established in
\cite[Theorem 4.1]{Knutson.Miller:subword}
that whenever $\Delta$ is a ball or
sphere and $R$ is its Stanley--Reisner ring, then
$$\mathcal{K}(R,\mathbf{t})=\sum_{F}(-1)^{\dim(\Delta)-\dim(F)}\prod_{i\not\in
  F} (1-t_i),$$ where the sum is over all the {\em internal} faces $F$ of
$\Delta$.  (Their statement is for the case of the subword complex although
their proof, a short derivation from Hochster's formula~\cite{Hochster},
applies more generally.)  Now, $\Groth_{v,w}({\bf t})$ is by
construction exactly this sum for $\Delta_{v,w}$.  (Here our labeling of a
face $F$ by the pipe dream with crosses everywhere except the vertices of $F$
makes our statement much cleaner.)  This proves our statement for $R/K_{v,w}$.
Since the facets of $\Delta_{v,w}$ are precisely the faces labeled by reduced
pipe dreams, the multidegree statement for $R/K_{v,w}$ follows
\cite[Prop. 8.49, Thm. 8.53]{Miller.Sturmfels}.

Theorem~\ref{thm:prime} says $\Delta_{v,w}$ is the Stanley--Reisner complex of
$R/{\rm in}_\prec I_{v,w}$, so the statements for
$R/{\rm in}_{\prec} I_{v,w}$ follow.  (Alternatively, we can deduce this result directly from
Theorem~\ref{thm:subwordKpoly}.)

The stated relationship between $\Schub_{v,w}$ and $\Groth_{v,w}$ is the
standard relationship between $K$-polynomials and multidegrees;
see \cite[Section~8.5]{Miller.Sturmfels}.
\end{proof}

\begin{example}
\label{exa:4.2}
Continuing Example~\ref{exa:31452}, we have
\begin{multline}\nonumber
{\widetilde \Groth_{31452}}=(1+u_4(1-t_{42}))(1+u_2(1-t_{21}))(1+u_3(1-t_{22}))
(1+u_1(1-t_{11}))\times\\ \nonumber
(1+u_2(1-t_{12}))(1+u_3(1-t_{13})).
\end{multline}
Expanding then collecting all terms with coefficient
\[u_{w_0w}=u_{13524}=u_4 u_2 u_3\equiv u_2 u_4 u_3\]
gives
an alternating sum over all the internal faces of
$\Delta_{31452,53142}$, namely
\begin{multline}\nonumber
\Groth_{31452,53142}({\bf t})=(1-t_{42})(1-t_{21})(1-t_{22})+
(1-t_{42})(1-t_{21})(1-t_{13})
+(1-t_{42})(1-t_{12})(1-t_{13})\\ \nonumber
-(1-t_{42})(1-t_{21})(1-t_{22})(1-t_{13})
-(1-t_{42})(1-t_{21})(1-t_{12})(1-t_{13}).
\end{multline}
If we calculate the unspecialized Schubert polynomial, we get
\[\Schub_{31452,53142}({\bf t})=t_{42}t_{21}t_{22}+t_{42}t_{21}t_{13}
+t_{42}t_{12}t_{13}.\]
\qed \end{example}

Rather than give a standard definition of the double Grothendieck and Schubert
polynomials of \cite{Lascoux.Schutzenberger1, Lascoux.Schutzenberger2}, we
prefer from our viewpoint to define them via the unspecialized versions,
proving the equivalence by assuming the formula of \cite{Fomin.Kirillov}. In
fact, the final claim of the definition--theorem below recovers
\cite[Theorem~A]{Knutson.Miller}, which states that the double Grothendieck
polynomial $\Groth_{w}({\bf x}, {\bf y})$ is the $K$-polynomial of a matrix
Schubert variety ${\overline X}_{w}$ (which is isomorphic to the
Kazhdan--Lusztig variety ${\mathcal N}_{w_0 \star w_0,{\hat w}}$ where $w_0
\star w_0$ and ${\hat w}$ are defined as in the proof of
Corollary~\ref{cor:Knutson.Miller}).

\begin{defthm}
\label{defthm:defineSchub}
The double Grothendieck polynomial and double Schubert polynomial of
\cite{Lascoux.Schutzenberger1, Lascoux.Schutzenberger2} are given
by
\begin{equation}
\label{eqn:defequalityG}
\Groth_{w}({\bf x},{\bf y})
=\Groth_{w_0 \star w_0,{\hat w}}(t_{ij}\mapsto x_j/y_i), \mbox{ \ and \ }
\end{equation}
\begin{equation}
\label{eqn:defequalityS}
\Schub_{w}({\bf x},{\bf y})
=\Schub_{w_0 \star w_0,{\hat w}}(t_{ij}\mapsto x_j-y_i).
\end{equation}
In particular, these give respectively the $K$-polynomial and multidegree of
${\mathcal N}_{w_0 \star w_0,{\hat w}}$ under our special case of the usual action.
\end{defthm}

\begin{example}
\label{exa:Schub13524}
Let us compute $\Schub_{13524}(\xx,\yy)$. Here $w=13524$, so
${\hat w}={\underline{10}} 8 6 9754321\in S_{10}$. To compute
$\Schub_{w_0\star w_0, {\hat w}}({\bf t})$, we consider all pipe dreams
in the $5\times 5$ box given by $D(w_0\star w_0)$ whose associated
product is
$w_0^{(10)}{\hat w}=135246789{\underline{10}}=w\times 1_{5}$. The reader
can check that there are six such pipe dreams, and summing up their
weights under $t_{ij}\mapsto x_i-y_j$ gives, by Definition--Theorem~\ref{defthm:defineSchub}:
\begin{multline}\nonumber
\Schub_{13524}(\xx,\yy)=(x_2-y_3)(x_1-y_2)(x_2-y_2)+(x_2-y_3)(x_1-y_2)(x_3-y_1)\\\nonumber
+(x_2-y_3)(x_2-y_1)(x_3-y_1)+(x_1-y_4)(x_1-y_2)(x_3-y_1)+(x_1-y_4)(x_1-y_2)(x_2-y_2)\\ \nonumber
+(x_1-y_4)(x_2-y_1)(x_3-y_1)\nonumber
\end{multline}
\qed \end{example}

\noindent
\emph{Proof of Definition--Theorem~\ref{defthm:defineSchub}:} Under our
conventions, the Schubert polynomial formula \cite[Theorem~2.3]{Fomin.Kirillov} states
$$\Schub_{w}(\xx,\yy) = \sum_{\mathcal{P}}
\prod_{(i,j)\in\mathcal{P}} x_j-y_i,$$ where the
sum is over reduced pipe dreams for $w$
fitting inside an $n\times n$ square. Moreover, in
\cite[Theorem~2.3]{Fomin.Kirillov} one also has (under our conventions)
$$\Groth_{w}(\xx,\yy) = \sum_{\mathcal{P}}
(-1)^{\#\mathcal{P}-\ell(w)}\prod_{(i,j)\in\mathcal{P}} 1-x_j/y_i,$$ with the
sum being over all pipe dreams for $w$ on $n\times n$.  It is straightforward
to check that ${\rm Pipes}(w_0 \star w_0,{\hat w})$ is the same set of pipe
dreams.  (Previous authors write their pipe dreams transposed and turned
upside down from ours because they used conventions natural for the study of
Schubert polynomials rather than conventions natural for the study of Schubert
varieties.) Thus (\ref{eqn:pipeformulaG}), after the substitution
\[t_{ij}\mapsto x_j/y_i\]
is
precisely the above formula for $\Groth_{w}({\bf x},{\bf y})$.
Similarly, (\ref{eqn:pipeformulaS}) is the known formula for $\Schub_{w}(\xx,
\yy)$ after the substitution
\[t_{ij}\mapsto x_j-y_i.\]
Hence
(\ref{eqn:defequalityG}) and (\ref{eqn:defequalityS}) hold.

Note that the substitution $t_{ij}\mapsto x_j/y_i$ is precisely the map on
$K$-polynomials induced by the inclusion of the usual action (with relabeled
weights) into the rescaling action, and the substituion $t_{ij}\mapsto
x_j-y_i$ is the equivalent map for multidegrees.
Thus the claim of the final sentence follows from
Theorem~\ref{thm:main1}, Theorem~\ref{thm:prime},
and Proposition~\ref{prop:unspecializedequalsK} combined.\qed

While Definition--Theorem 4.3 exploits the usual action on ${\mathcal
  N}_{w_0 \star w_0,{\hat w}}$ for the double Schubert and Grothendieck
polynomials, the usual action on arbitrary ${\mathcal N}_{v,w}$ can be used to
geometrically explain the specialization formula
\cite{Buch.Rimanyi} for
double Grothendieck polynomials.  (We emphasize that the following result
and its proof hold even if we define Grothendieck polynomials as traditionally
done \cite{Lascoux.Schutzenberger2}.)


\begin{theorem}
\label{thm:specialization}
We have the equalities:
\begin{multline}
\label{eqn:tripleequalityG}
\Groth_{w_0w}(t_{v(1)},\ldots,t_{v(n)};t_{n},t_{n-1}\ldots,t_{1})=
{\mathcal K}({\mathcal N}_{v,w},t_{ij}\mapsto t_{v(j)}/t_{n-i+1})\\
=\Groth_{v,w}(t_{ij} \mapsto t_{v(j)}/t_{n-i+1}),
\end{multline}
and
\begin{multline}
\label{eqn:tripleequalityS}
\Schub_{w_0w}(t_{v(1)},\ldots,t_{v(n)};t_{n},t_{n-1},\ldots,t_{1})=
{\mathcal C}({\mathcal N}_{v,w},t_{ij}\mapsto t_{v(j)}-t_{n-i+1})\\
=\Schub_{v,w}(t_{ij}\mapsto t_{v(j)}-t_{n-i+1}).
\end{multline}
\end{theorem}

The equality of the first and third polynomials in each of (\ref{eqn:tripleequalityG}) and
(\ref{eqn:tripleequalityS}) was obtained by \cite{Buch.Rimanyi}, who furthermore
ask for a geometric explanation. We respond to that question by showing that both are
in fact equal to an equivariant ($K$-theory) class of ${\mathcal N}_{v,w}$. Our proof utilizes
our Gr\"{o}bner basis result, Theorem~\ref{thm:main1}.

\noindent
\emph{Proof of Theorem~\ref{thm:specialization}:}
The injection $\{e_v\}\hookrightarrow {\rm Flags}(\mathbb{C}^n)$ induces a map
\begin{eqnarray}\nonumber
K_T({\rm Flags}(\mathbb{C}^n)) & \rightarrow & K_T(e_v)\\ \nonumber
[\mathcal{O}_{X_w}]_T & \mapsto & [\mathcal{O}_{X_w}]_T\mid_{e_v}.
\end{eqnarray}
Here $T$ is the torus $(\mathbb{C}^*)^n$ and the torus action is the usual
action from Section~4.1.

Since $v\Omega_{id}^{\circ}$ $T$-equivariantly contracts to $e_v$, and the
isomorphism of Equation~\ref{eqn:KL} is $T$-equivariant, we can identify the
class
\[[\mathcal{O}_{\mathcal{N}_{v,w}}]_T\in K_T(\Omega^\circ_v)\]
with
\[[\mathcal{O}_{X_w}]_T\mid_{e_v}\in K_T(e_v).\]

Choosing the usual coordinates for the weight space of $T$, the class in
$K_T(\Omega^\circ_v)$ of any subscheme (or, in general $\mathcal{O}$-module)
over $\Omega^\circ_v$ is given by its $K$-polynomial.  Therefore, we can
make the identifications
\[[\mathcal{O}_{\mathcal{N}_{v,w}}]_T=[\mathcal{O}_{X_w}]_T\mid_{e_v}=
\mathcal{K}(\mathcal{N}_{v,w},t_{ij}\mapsto t_{v(j)}/t_{n-i+1}),\]
by the
  discussion of Section~4.1.

On the other hand, it is a folklore theorem that
\[[\mathcal{O}_{X_w}]_T\mid_{e_v}=\Groth_{w_0w}(x_j\mapsto t_{v(j)}, y_i\mapsto
t_{w_0(i)}).\]
While we could not find an explicit proof for this statement in the
literature, R.~Goldin \cite{Goldin} gave a proof for the equivalent statement
for equivariant cohomology; her proof can be seen to extend to equivariant
$K$-theory, with the appropriate modifications.

Alternatively, one can also construct a similar proof, substituting
homological algebra for geometry, as sketched below.  First we impose the
torus action on $M_n$ so that the variables in ${\mathbf z}$ have the weights
they would have in ${\mathbf z}^{(v)}$; this means giving the variable
$z_{ij}$ the weight $t_{v(j)}/t_{w_0(i)}$ rather than the weight $x_j/y_i$.
Now we consider the restriction map
$$
K_T(\mathbb{C}[\mathbf{z}]) \rightarrow K_T(\mathbb{C}[\mathbf{z}^{(v)}]).
$$
Since it is equivariantly contractible to the identity map from a point to a
point, it is the identity map on $K$-polynomials.  Furthermore, since this is
a map of affine schemes, left-derived pullback is simply $\mathrm{Tor}$, so
for any $\mathbb{C}[\mathbf{z}]$-module $Y$, the class $[Y]\in K_T(\mathbb{C}[\mathbf{z}])$ is mapped to
$$\sum_i (-1)^i [{\rm Tor}_i(Y,{\mathbb C}[\mathbf{z}^{(v)}])].$$
Equivalently, if one wants to use only algebraic arguments, this can be seen
from the calculation of $\mathrm{Tor}$ using a free resolution of $Y$
and the calculation of $K$-polynomials from a free resolution (or indeed any
exact sequence).

Now
$\operatorname{Tor}_0(\mathbb{C}[\mathbf{z}]/I_{w_0w},\mathbb{C}[\mathbf{z}^{(v)}])$
is simply the coordinate ring $\mathcal{O}_{\mathcal{N}_{v,w}}$.
Furthermore, the ideal defining $\mathbb{C}[\mathbf{z}^{(v)}]$
(defined in Section~\ref{subsect:eqns}) is generated by elements of
$\mathbb{C}[\mathbf{z}]$ which are a regular sequence on the
coordinate ring $\mathbb{C}[\mathbf{z}]/I_{w_0w}$ of the matrix
Schubert variety $\overline{X}_{w_0w}$.  (It can be easily seen that
they are part of a system of parameters, since the ideal is generated
by $\binom{n+1}{2}+\ell(v)$ elements (as $\binom{n}{2}+\ell(v)$
entries are set to $0$ and $n$ entries to $1$) and the codimension of
$\mathcal{N}_{v,w}$ in $\overline{X}_{w_0w}$ is also
$\binom{n+1}{2}-\ell(v)$.  Since $\overline{X}_{w_0w}$ is known to be
Cohen--Macaulay~\cite{Fulton:Duke92}, any part of a system of
parameters is also a regular sequence.)  Therefore, by standard facts
(see for example \cite[Prop. 1.6.9, Thm. 1.6.17b]{Bruns.Herzog}) in
the theory of regular sequences,
$\operatorname{Tor}_i(\mathbb{C}[\mathbf{z}]/I_w,\mathbb{C}[\mathbf{z}^{(v)}])=0$
for $i>0$.

Therefore, the $K_T$-class for $\overline{X}_w$ (with respect to the
action which restricts to the usual action on $\mathbb{C}[\mathbf{z}^{(v)}]$)
restricts to the $K_T$ class for $\mathcal{N}_{v,w}$.  Since this restriction
map is the identity on $K$-polynomials, we must have that
\[[\mathcal{O}_{{\mathcal N}_{v,w}}]_T(=[\mathcal{O}_{X_w}]_T\mid_{e_v})=\Groth_{w_0w}(x_j\mapsto t_{v(j)}, y_i\mapsto
t_{w_0(i)}),\]
as desired.

Thus we obtain the first equality of (\ref{eqn:tripleequalityG}). Since
$[{\mathcal O}_{{\mathcal N}_{v,w}}]_T$, the $K$-polynomial of
${\mathcal N}_{v,w}$, is preserved under Gr\"{o}bner degeneration, the
second equality of (\ref{eqn:tripleequalityG})
follows from
Theorem~\ref{thm:main1}, Theorem~\ref{thm:prime}
and Proposition~\ref{prop:unspecializedequalsK} combined.

The proof of (\ref{eqn:tripleequalityS}) is similar (although in this case we
can use Goldin's result).
\qed

\begin{example}
For simplicity, let us only illustrate the Schubert polynomial assertions
(\ref{eqn:tripleequalityS}) of Theorem~\ref{thm:specialization}. We continue
Example~\ref{exa:4.2} where $v=31452$ and $w=53142$ (and hence $w_0 w=13524$).
Now, from Example~\ref{exa:Schub13524}, we have that
\begin{multline}\nonumber
\Schub_{w_0w}(t_{v(1)},\ldots,t_{v(n)}; t_n,\ldots,t_1)=\Schub_{13524}(t_3,t_1,t_4,t_5,t_2;
t_5,t_4,t_3,t_2,t_1)=\\ \nonumber
(t_1-t_3)(t_3-t_4)(t_1-t_4)+(t_1-t_3)(t_3-t_4)(t_4-t_5)+(t_1-t_3)(t_1-t_5)(t_4-t_5)\\ \nonumber
+(t_3-t_2)(t_3-t_4)(t_4-t_5)+(t_3-t_2)(t_3-t_4)(t_1-t_4)+(t_3-t_2)(t_1-t_5)(t_4-t_5)
\end{multline}
On the other hand, from Example~\ref{exa:4.2}, we see that
\begin{multline}\nonumber
\Schub_{v,w}(t_{ij}\mapsto t_{v(j)}-t_{n-i+1})=\Schub_{31452,53142}(t_{ij}\mapsto t_{v(j)}-t_{n-i+1})=\\
(t_1-t_2)(t_3-t_4)(t_1-t_4)+(t_1-t_2)(t_3-t_4)(t_4-t_5)+(t_1-t_2)(t_1-t_5)(t_4-t_5).\\
\end{multline}
Now, Theorem~\ref{thm:specialization} asserts that these two polynomials in
the $t_i$'s are equal, which the reader can check by direct
computation. However, this equality is not obvious \emph{a priori}.
\qed \end{example}

\subsection{A rationale for unspecializing}
The standard definition for Grothendieck and Schubert polynomials is in terms
of (isobaric) divided difference operators. Our rationale for presenting this
nonstandard (and highly ahistorical) definition through ``unspecialization''
is as follows. What \cite{Knutson.Miller} taught us is that the Grothendieck
and Schubert polynomials, being presented as polynomials in the $x_i, y_j$
variables is already ``biased'' as a equivariant ($K$-theory) class of a
matrix Schubert variety for the special torus action.  However, from this
point of view, the specialization of these polynomials examined in
Theorem~\ref{thm:specialization} appears geometrically unnatural and even
combinatorially mysterious.  Our approach seeks to emphasize that rather than
viewing the latter as a specialization of the former, one should think that
the two are specializations of two different though related unspecialized
Grothendieck polynomials.  Moreover, these specializations are geometrically
natural, since they arise as explained in Section 4.2 from a restriction from
the larger rescaling torus action to the smaller usual torus action.

\section{Multiplicities of Schubert varieties}

We take this opportunity to relate our work to the problem of finding positive
formulas for multiplicities of Schubert varieties.

The {\bf multiplicity} of a point $p$ in a scheme $X$, denoted
${\rm mult}_{p}(X)$ is defined as the degree of the projective tangent cone
$Proj({\rm gr}_{\mathfrak m_p} {\mathcal O}_{X_p})$
as a subvariety of the projective tangent space
$Proj(Sym^* {\mathfrak m_p}/{\mathfrak m}_p^2)$, where
$({\mathcal O}_{X_p}, {\mathfrak m_p})$ is the local ring
associated to $p\in X$. Equivalently, if the Hilbert--Samuel polynomial of
${\mathcal O}_{X_p}$ is $a_d x^d+a_{d-1}x^{d-1}+\ldots +a_0$ then
\[{\rm mult}_{p}(X)=d!a_d.\]

The following open problem has been of interest:
\begin{problem}
\label{problem:mult_problem}
Give an explicit, nonrecursive, positive combinatorial rule to compute
${\rm mult}_{e_v}(X_w)$, for each $(v,w)\in S_n\times S_n$.
\end{problem}
This problem remains open for (most cases of) generalized flag varieties
$G/P$ and even for the case of ${\rm Flags}({\mathbb C}^n)$, which is our
present focus.  However, for minuscule $G/P$, a recursive formula was given by
Lakshmibai and Weyman \cite{Lakshmibai.Weyman}.  For the special case of the
Grassmannian $Gr(k,{\mathbb C}^n)$ of $k$-dimensional planes in ${\mathbb
  C}^n$, several closed formulas have been given~\cite{Rosenthal.Zelevinsky,
  Krattenthaler, Kreiman.Lakshmibai}, and similar formulas were recently given
for the symplectic Grassmannian~\cite{Ghorpade.Raghavan} and the orthogonal
Grassmannian~\cite{Raghavan.Upadhyay} (both in the case of maximal isotropic
subspaces).

\subsection{Homogeneity and parabolic moving}

Let us now describe two facts and a conjecture, which allow us to positively
compute multiplicities in many cases. The first gives a combinatorial
rule whenever the Kazhdan--Lusztig ideal is homogeneous. The proof was
suggested to us by A.~Knutson in a private communication of his combinatorial
rule for multiplicity problems $(v,w)$ which holds in any type whenever $w_0v$
is $\lambda$-cominuscule (for some weight $\lambda$).  In particular $w_0 v$
is $\lambda$-cominuscule in type $A$ if it is $321$-avoiding.  See
\cite[Proposition~2.1]{Stembridge} and Section~\ref{subsect:actions} for
details.

\begin{fact}[Homogeneity]
\label{thm:knutsonargument}
Suppose $v,w\in S_n$ and $I_{v,w}$ is homogeneous with respect to the standard
grading that assigns $\deg(z_{ij})=1$.  Then ${\rm mult}_{e_v}(X_w)$ equals
the number of facets of $\Delta_{v,w}$, or equivalently the number of reduced
pipe dreams for $w$ on $D(v)$.
\end{fact}
\begin{proof}
The value of ${\rm mult}_{e_v}(X_w)$ equals
the degree of the initial ideal of $I_{v,w}$ with respect to any
term order $\prec'$ that always picks out a lowest degree term as its leading
term. However, if $I_{v,w}$ is already homogeneous with respect to the standard
grading, then one can use $\prec':=\prec$, and thus the result follows from
Theorems~\ref{thm:main1} and~\ref{thm:prime}.
\end{proof}

Ideally, one would have a simple combinatorial characterization for when
$I_{v,w}$ is homogeneous with respect to the standard grading
(see \cite[Problem~5.5]{WYII}).  At present, we do not know how to solve even
the presumably simpler problem of determining when the defining (or essential)
minors are homogeneous.

For the purposes of computing multiplicity in general, we would need, as
stated in the proof above, a Grobner basis
under any term order that picks out a lowest degree term.
The defining determinants are not a Gr\"{o}bner basis in
general for any such term orders we have tried. However, as we explain below,
it suffices to solve a subset of these problems.


To see this, let us now recall another well-known trick.
Let
\[{\mathcal T}=\{s_i=(i\leftrightarrow i+1) \  | \ s_iw<w\}.\]
These are
known as the {\bf left descents} of $w$.  Similarly let
\[{\mathcal T}^\prime=\{s_i=(i\leftrightarrow i+1) \ | \ ws_i<w\}\]
be the set of {\bf right descents} of $w$.
In general, given a Schubert variety $X_w$, the parabolic subgroup
$P_{\mathcal T}\subset G$ generated by $B$ and the transpositions in ${\mathcal T}$ acts on it by left
multiplication. In particular, if $s_iw<w$, then this action induces an
isomorphism of a local neighbourhood of $e_v$ in $X_{w}$ with a local
neighbourhood of $e_{s_i v}$ in $X_w$, thus preserving all local properties at
these points. Since local properties are preserved under inverse (as
$\mathcal{N}_{v,w}$ is isomorphic to $\mathcal{N}_{v^{-1},w^{-1}}$), we also
have that, if $ws_i<w$, then a similar statement holds for $e_v$ and
$e_{vs_i}$ in $X_w$. Thus, one can compute invariants of $e_v$ in $X_w$
from $e_{v'}$ in $X_w$ whenever $v$ and $v'$ are in the same double coset
$S_{\mathcal T}vS_{{\mathcal T}^\prime}$ in $S_n$.  Here $S_{\mathcal T}$ and
$S_{{\mathcal T}^\prime}$ denote respectively the subgroups of $S_n$ generated
by the simple transpositions in ${\mathcal T}$ and ${\mathcal T}^\prime$.

Applying the above trick to computing multiplicity, we have:

\begin{fact}[Parabolic moving]
\label{fact:B}
If $s_iw<w$, then
${\rm mult}_{e_v}(X_w)={\rm mult}_{e_{s_iv}}(X_w)$. Similarly, if
$ws_i<w$, then ${\rm
  mult}_{e_{vs_i}}(X_w)={\rm mult}_{e_v}(X_w)$.
\end{fact}

Combining Facts~\ref{thm:knutsonargument} and~\ref{fact:B},
one can hope to positively compute
${\rm mult}_{e_v}(X_w)$ by using parabolic moving to instead calculate
${\rm mult}_{e_{v'}}(X_w)$ where
\[v'\in S_{{\mathcal T}}vS_{{\mathcal T}'}\]
and $I_{v',w}$ is standardly homogeneous.

Actually, we expect conjecturally that there is a particular ``good''
$v'\in S_{{\mathcal T}}vS_{{\mathcal T}'}$ to use.
Define
\[v'=v_{\rm max}\]
to be {\bf parabolically maximal} if it is maximal (in Bruhat order) in its
double coset $S_{\mathcal T}vS_{{\mathcal T}^{\prime}}$.  Each double coset
has a unique maximal element.  Combinatorially, if $v'$ is parabolically
maximal (for some $v$) if its left and right descent sets contain those of
$w$.  Moreover given $v$, we can find $v_{\rm max}$ by first rearranging in
decreasing order the entries of $v$ with numbers corresponding to segments of
consecutive generators of ${\mathcal T}$, then rearranging
the resulting permutation so that entries in positions corresponding to
segments of consecutive generators of ${\mathcal T}^\prime$ are in decreasing
order.

\begin{example}
Let $v=316298475\leq w=896354721$. Then the left descents and right descents
of $w$ are given respectively by
\[{\mathcal T}=\{s_1,s_2,s_4,s_5,s_7\} \mbox{\ \  and \ \ }
{\mathcal T}'=\{s_2,s_3,s_5,s_7,s_8\}.\] In order to obtain $v_{\rm max}$, the
elements $s_1$ and $s_2$ of ${\mathcal T}$ indicate that one should rearrange
the labels $1,2,3$ in $v$ in decreasing order, whereas the elements $s_4$ and
$s_5$ of $\mathcal{T}$ indicate that one should then rearrange the labels
$4,5,6$ in $v$ in decreasing order, and $s_7$ indicates that $8$ should be put
before $7$.  Doing this, one obtains $v\mapsto 326198574$.  Now, similarly,
the elements $s_{2}$ and $s_3$ of ${\mathcal T}'$ tell us to rearrange the
positions $2,3,4$ of $326198574$, and so on. This process then terminates with
$v_{\rm max}=362198754$.
\qed \end{example}

Our discussion above shows that to solve all multiplicity problems, it
suffices to solve the parabolically maximal ones.  The following asserts that
it suffices to check the homogeneity of $I_{v_{\rm max},w}$ if one wishes to
know if Facts 5.1 and 5.2 suffice to compute the multiplicity of $e_v$ on
$X_w$.

\begin{conjecture}[Parabolic maximality]
\label{conj:C}
Suppose the Kazhdan--Lusztig ideal $I_{v,w}$ is standardly homogeneous.  If
$ws_i<w$ but $vs_i>v$, then $I_{vs_i,w}$ is standardly homogeneous, and
similarly, if $s_i w<w$ but $s_iv>v$, then $I_{s_i v,w}$ is standardly
homogeneous. Therefore, $I_{v_{\rm max},w}$ is standardly homogeneous.
\end{conjecture}

The final sentence of Conjecture~\ref{conj:C} clearly follows by induction
using the second sentence. It seems plausible that one can deduce the second
sentence using similar analysis as in our proof of
Proposition~\ref{prop:mainprop1} in Section~6.2, although we do not pursue
this here.


\subsection{Computational results and Monte Carlo simulation}
Together Facts~\ref{thm:knutsonargument}, \ref{fact:B} and Conjecture
\ref{conj:C} provide a useful means to solve multiplicity problems.  We can
use the symbolic algebra software {\tt Macaulay~2} to computationally check
whether $I_{v,w}$ is standardly homogeneous, by first applying {\tt trim} to
the set of defining (or essential) minors of $I_{v,w}$ and then using the
function {\tt isHomogeneous}.  Testing on the set
\[\Gamma_n:=\{(v,w)\in S_n\times S_n\ | \ v<w \mbox{\ in Bruhat order}\},\]
we found by exhaustive search that, for $n=5$, $74\%$ of all problems fall to
Fact~\ref{thm:knutsonargument} alone.  However, not suprisingly, the success
percentage falls off quickly.

On the other hand, if we also use Fact~\ref{fact:B} and consistently replace
$(v,w)\in \Gamma_n$ with $(v_{\rm max},w)\in \Gamma_n$, the success
percentage increases rather substantially. By exhaustive search, all problems
for $n\leq 4$ are solved this way while $98.5\%$ of the $3871$ problems are
solved for $n=5$. Monte Carlo simulation estimates are summarized in
the following table:

\begin{table}[h]
\begin{tabular}{|l||l|l|l|l|l|}
\hline
$n$ & 6 & 7 & 8 & 9 & 10\\ \hline
success $\%$ & 94 & 86 & 73 & 62 & 46\\ \hline
\end{tabular}
\caption{Estimates of success percentage with $2,000$ Monte Carlo trials, using
Facts~\ref{thm:knutsonargument} and~\ref{fact:B}.}
\end{table}

We found it encouraging that such simple tricks allow one to cover
such a large fraction of all multiplicity problems for even up to
$n=10$.  Furthermore, G.~Warrington has discovered a similar
phenomenon in his investigations of leading coefficients of
Kazhdan--Lusztig polynomials~\cite{Warrington:KLExper}.

Let us collect a few more computationally determined facts:

\begin{fact}
\begin{itemize}
\item[(I)] $I_{v_{{\rm max}},w}$ need not be standardly homogeneous,
one example is
\[I_{13425, 34512}=\langle z_{11},z_{12},z_{21},z_{13}z_{22}z_{31}-z_{14}z_{41}\rangle\]
(Therefore, Conjecture~\ref{conj:C}, even if true,
would not solve all multiplicity problems).
\item[(II)] $I_{v,w}$ may be standardly homogeneous even if $w_0v$ is not
  $321$-avoiding (and therefore not $\lambda$-cominuscule for any $\lambda$);
  for example
\[I_{45213,54231}=\langle z_{33}\rangle.\]
\item[(III)] $I_{v,w}$ might not be standardly homogeneous even if
$I_{v_{\rm max},w}$ is.  (Hence the converse to Conjecture~\ref{conj:C}
is false.)  For example,
\[I_{v,w}=I_{31524,43512}=\langle z_{11}, z_{12}, z_{24}z_{42}-z_{22}\rangle\]
while
\[I_{v_{\rm max},w}=I_{41532,43512}=\langle z_{11},z_{12},z_{42}\rangle.\]
\end{itemize}
\end{fact}

It is worthwhile to mention that,
as with all checks on $\Gamma_n$, the computational demands are large
for $n\geq 6$. On the other hand, it is not difficult to
Bernoulli sample a pair
$(v,w)\in \Gamma_n$ uniformly at random.
One can independently and uniformly pick two permutations $\sigma,\rho\in S_n$,
and reject until either $\sigma\leq \rho$ or $\rho\leq \sigma$. In the first
case, one returns $(u,v)=(\sigma,\rho)$ while in the latter case one returns
$(u,v)=(\rho,\sigma)$.

In our experience, this approach allows one to practically estimate
success probabilities for $n$ beyond the reach
of exhaustive search, in the sense that the true bottleneck in computation comes from
the Gr\"{o}bner basis computations. In particular, Conjecture~\ref{conj:C} is endorsed
up to $n\leq 10$ using this method.

We also remark that rigorous analysis of the likelihood of picking a Bruhat
comparable pair from $S_n\times S_n$ was performed recently by
A.~Hammett and B.~Pittel
\cite{Pittel}. They bound this
probability by
\[c_1(0.708^n)\leq {\mathbb P}[u\leq v \mbox{\ or \ } v\leq u | (u,v)\in S_n\times S_n]\leq c_2\frac{1}{n^2},\]
for universal constants $c_1,c_2>0$. Moreover they conjecture this probability is about $n^{-3/2}$. If their
conjecture is true, then one would expect the Bernoulli
sampling algorithm to
terminate quickly, in about $O(n^{3/2})$ trials, which
agrees with our experience.

The above computations support the idea that Monte Carlo simulation is a
useful resource when studying algebraic combinatorics and computational
commutative algebra such as that present in \cite{WYII}. In that work, one
needs to sample elements of $\Gamma_n$ satisfying ``interval pattern avoidance
conditions''.  This motivates the need for more sophisticated (and efficient)
sampling algorithms (via methods such as Markov Chain Monte Carlo or
importance sampling); further discussion may appear elsewhere.

\subsection{Formulas in the Grassmannian case}

Now suppose $w\in S_n$ is co-Grassmannian, meaning that it has a unique ascent
$w(k)<w(k+1)$ (or, equivalently, that there is at most one simple
transposition $s_k$ with the property that $ws_k>w$).  Let us consider the
multiplicity problem in only this case.  In general, if $v\leq w$, it is not
true that $v$ is also {\bf co-Grassmannian}.  However, observe that we can
always replace $v$ by $v_{\rm max}$, which is then (by the discussion of
Section~5.1), co-Grassmannian, with its unique ascent also in position $k$.
This reduces the problem to computing multiplicities on Grassmannians, a
problem previously considered in \cite{Lakshmibai.Weyman,
  Rosenthal.Zelevinsky, Krattenthaler}. In summary these results provide
determinantal and tableau based formulas for the multiplicity.  Our goal here
is to provide a (mildly) more general, simpler, formula (being valid for all
$v\leq w$ and not only co-Grassmannian $v\leq w$), together with a new
conceptual explanation for the appearance of these formulas, using the results
of this paper.

Let us therefore assume unless otherwise stated that $v\leq w$ are both
co-Grassmannian with the same ascent position $k$.  The co-Grassmannian
assumption on $v$ allows us to easily check that the defining generators of
$I_{v,w}$ are homogeneous with respect to the standard grading.  In fact,
homogeneity also follows from $w_0v$ being $\lambda$-cominuscule.

Consequently, Fact~\ref{thm:knutsonargument} guarantees that the multiplicity
of $e_{v}$ in $X_{w}$ is the number of reduced pipe dreams on $D(v)$ for
$w_0w$.  Note that $w_0w$ is a {\bf Grassmannian permutation}, meaning one
with a unique descent, in this case at position $k$.  Moreover, under the
present assumptions one can establish a bijection between reduced pipe dreams
for $w_0w$ in $D(v)$ and ``flagged'' semistandard tableaux, as we explain
below.

First we need some standard facts about co-Grassmannian
permutations.  To each such permutation $w$ with its unique ascent at position
$k$, we associate a partition
\[\lambda(w)=(\lambda_1\geq\lambda_2\geq\ldots\geq\lambda_k\geq 0)
\mbox{\ \  by setting \ \ }
\lambda_{k-i+1}=n-w(i)+1-i.\]

Hence, for example, if $w=975386421$ then $\lambda(w)=(4\geq 2\geq 1\geq0)$.
If $v$ and $w$ are two co-Grassmannian permutations both with ascents at $k$,
then $\lambda(v)\supseteq \lambda(w)$ if and only if $v\leq w$.  Note under
these conventions $\left|\lambda\right|=\sum_{i} \lambda_i$ is the codimension
of $X_w$ in $\Flags({\mathbb C}^n)$, which $\binom{n}{2}-\ell(w)$ or
equivalently the number of non-inversions in $w$.

The boxes of ${\tt flatten}(D(v))$ form the shape $\lambda(v)$ rotated $180$
degrees and conjugated (transposed).  Now consider the flattened pipe dreams
inside ${\tt flatten}(D(v))$, as in Section~3.3. The co-Grassmannian
assumption on $w$ (and hence Grassmannian assumption on $w_0w$) implies that
we can produce each reduced pipe dream for $w_0 w$ by the following procedure,
which we describe in terms of pipe dreams drawn on ${\tt flatten}(D(v))$
rather than on $D(v)$.  One can easily recover pipe dreams in ${\rm
  RedPipes}(v,w_0w)$ by unflattening.

Start with the {\bf
  starting pipe dream}, which is the unique pipe dream whose $+$'s form the
shape $\lambda(w)$ rotated 180 degrees, conjugated, and placed in the lower
right hand corner of ${\tt flatten}(D(v))$. Then, locally, one can make the
transformation
\begin{equation}
\label{eqn:transforms}
\begin{matrix}
\cdot & \cdot\\
\cdot & +
\end{matrix} \ \mapsto \
\begin{matrix}
+ & \cdot\\
\cdot & \cdot
\end{matrix}
\end{equation}
where each $2\times 2$ configuration describes a subsquare of ${\tt
  flatten}(D(v))$, and the ``$\cdot$'' refers to a square of ${\tt
  flatten}(D(v))$ without a $+$.  Such a transformation will produce another
reduced pipe dream for $w_0w$, and one can generate any reduced pipe dream for
$w_0w$ which fits inside ${\tt flatten}(D(v))$ by some sequence of such
transformations from the starting pipe dream.

We now associate a semistandard Young tableau of shape $\lambda(w)$ to each
reduced pipe dream for $w_0w$ on ${\tt flatten}(D(v))$.  We associate to the
starting pipe dream the super-semistandard {\bf starting tableau} of shape
$\lambda(w)'$, defined to have label $m$ in every box of row $m$.  More
generally, each $+$ in the starting tableau is in obvious bijection with a box
of $\lambda(w)$, and following the local transformation (\ref{eqn:transforms})
allows one to coherently associate each $+$ of any pipe dream to a box of
$\lambda(w)$, namely the box associated to the $+$ in the starting tableau
that it came from.  Now if a $+$ is in the $i$-th column of ${\tt
  flatten}(D(v))$, counting from the \emph{right} (and starting with the first
column having a box), then we put an $i$ in the corresponding box of
$\lambda(w)$.  The resulting tableau can be seen (by induction) to be
semi-standard.  These conclusions essentially follow from the analysis of
\cite{KMY} together with Section~3.3.

For example, if $\lambda(v)=(5,4,4,2)$ and $\lambda(w)=(4,2,1,0)$ we have that the starting pipe dream (after
rotating $180$-degrees) and the starting tableau are (after flattening, rotating and conjugating):
\begin{equation}
\label{eqn:startingpipe}
\tableau{{+}&{+}&{+}&{+}&{ \ }\\{+}&{+}&{\ }&{\ }\\{+}&{\ }&{\ }&{\ }\\{\ }&{\ }}\leftrightarrow
\tableau{{1}&{1}&{1}&{1}\\{2}&{2}\\{3}},
\end{equation}
and all others are obtained by the local moves (\ref{eqn:transforms}),
rotated, which look like
\begin{equation}
\label{eqn:transformsconj}
\begin{matrix}
+ & \cdot\\
\cdot & \cdot
\end{matrix} \ \mapsto \
\begin{matrix}
\cdot & \cdot\\
\cdot & +
\end{matrix}
\end{equation}

Not every semistandard tableau of shape $\lambda(w)$ can be obtained this way.
The maximum entry of row $m$ of such a tableau $T$ is bounded from above by
how far south the rightmost $+$ in the $m$-th row of the starting pipe
dream can travel and remain inside $\lambda(v)$.  Thus in the above example,
the possible semistandard tableau are of shape $\lambda(w)=(4,2,1)$ such that
the entries in the first, second and third rows respectively are bounded by
$1$, $3$, and $4$. In general, this is given by $b_m$, which is the row at
the bottom of the largest square that contains the right most box of
$\lambda(w)_m$ as its northwest corner and is contained inside
$\lambda(v)$.  That is, for $1\leq m\leq k$ define
\[b_m=\max\{1\leq i\leq k|\lambda(v)_{i}\geq \lambda(w)_{m}+i-m\}.\]
Clearly the sequence ${\bf b}=(b_1,b_2,\ldots,b_k)$ is weakly increasing.

\begin{example}
Let $v=743198652\leq w=975286431$. Then
$\lambda(v)=(5,4,4,2)\supseteq \lambda(w)=(4,2,1,0)$.
Therefore ${\bf b}=(1,3,4)$.
The diagram $D(v)$ and its canonical labeling are depicted below:
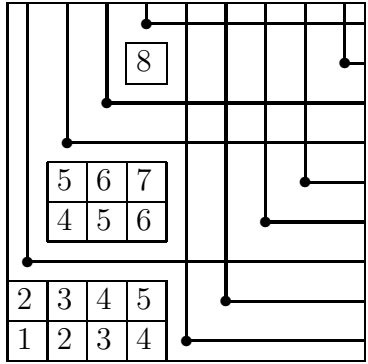
\[
\begin{picture}(200,140)
\put(37.5,0){\makebox[0pt][l]{\framebox(135,135)}}
\put(37.5,30){\line(1,0){60}}
\put(97.5,30){\line(0,-1){30}}
\put(37.5,15){\line(1,0){60}}
\put(52.5,30){\line(0,-1){30}}
\put(67.5,30){\line(0,-1){30}}
\put(82.5,30){\line(0,-1){30}}
\put(52.5,45){\line(1,0){45}}
\put(52.5,45){\line(0,1){30}}
\put(52.5,75){\line(1,0){45}}
\put(97.5,75){\line(0,-1){30}}
\put(82.5,75){\line(0,-1){30}}
\put(67.5,75){\line(0,-1){30}}
\put(52.5,60){\line(1,0){45}}
\put(82.5,105){\framebox(15,15)}
\thicklines
\put(45,37.5){\circle*{4}}
\put(45,37.5){\line(1,0){127.5}}
\put(45,37.5){\line(0,1){97.5}}
\put(60,82.5){\circle*{4}}
\put(60,82.5){\line(1,0){112.5}}
\put(60,82.5){\line(0,1){52.5}}
\put(75,97.5){\circle*{4}}
\put(75,97.5){\line(1,0){97.5}}
\put(75,97.5){\line(0,1){37.5}}

\put(90,127.5){\circle*{4}}
\put(90,127.5){\line(1,0){82.5}}
\put(90,127.5){\line(0,1){7.5}}

\put(105,7.5){\circle*{4}}
\put(105,7.5){\line(1,0){67.5}}
\put(105,7.5){\line(0,1){127.5}}

\put(120,22.5){\circle*{4}}
\put(120,22.5){\line(1,0){52.5}}
\put(120,22.5){\line(0,1){112.5}}

\put(135,52.5){\circle*{4}}
\put(135,52.5){\line(1,0){37.5}}
\put(135,52.5){\line(0,1){82.5}}

\put(150,67.5){\circle*{4}}
\put(150,67.5){\line(1,0){22.5}}
\put(150,67.5){\line(0,1){67.5}}

\put(165,112.5){\circle*{4}}
\put(165,112.5){\line(1,0){7.5}}
\put(165,112.5){\line(0,1){22.5}}

\put(41,4.5){$1$}
\put(56,4.5){$2$}
\put(71,4.5){$3$}
\put(86,4.5){$4$}

\put(41,19.5){$2$}
\put(56,19.5){$3$}
\put(71,19.5){$4$}
\put(86,19.5){$5$}

\put(56,49.5){$4$}
\put(71,49.5){$5$}
\put(86,49.5){$6$}

\put(56,64.5){$5$}
\put(71,64.5){$6$}
\put(86,64.5){$7$}

\put(86,109.5){$8$}
\end{picture}
\]
After flattening $D(v)$ we obtain the shape $(4,4,3,3,1)$ as read from bottom to top. This is the conjugate shape of
$\lambda(v)$. The starting pipe dream is given by
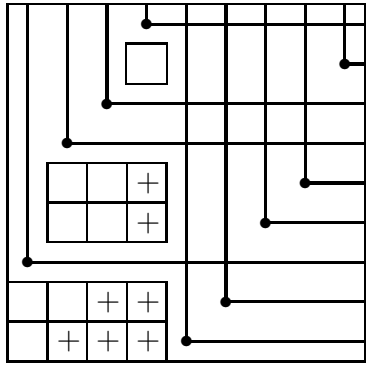
\[
\begin{picture}(200,140)
\put(37.5,0){\makebox[0pt][l]{\framebox(135,135)}}
\put(37.5,30){\line(1,0){60}}
\put(97.5,30){\line(0,-1){30}}
\put(37.5,15){\line(1,0){60}}
\put(52.5,30){\line(0,-1){30}}
\put(67.5,30){\line(0,-1){30}}
\put(82.5,30){\line(0,-1){30}}
\put(52.5,45){\line(1,0){45}}
\put(52.5,45){\line(0,1){30}}
\put(52.5,75){\line(1,0){45}}
\put(97.5,75){\line(0,-1){30}}
\put(82.5,75){\line(0,-1){30}}
\put(67.5,75){\line(0,-1){30}}
\put(52.5,60){\line(1,0){45}}
\put(82.5,105){\framebox(15,15)}
\thicklines
\put(45,37.5){\circle*{4}}
\put(45,37.5){\line(1,0){127.5}}
\put(45,37.5){\line(0,1){97.5}}
\put(60,82.5){\circle*{4}}
\put(60,82.5){\line(1,0){112.5}}
\put(60,82.5){\line(0,1){52.5}}
\put(75,97.5){\circle*{4}}
\put(75,97.5){\line(1,0){97.5}}
\put(75,97.5){\line(0,1){37.5}}

\put(90,127.5){\circle*{4}}
\put(90,127.5){\line(1,0){82.5}}
\put(90,127.5){\line(0,1){7.5}}

\put(105,7.5){\circle*{4}}
\put(105,7.5){\line(1,0){67.5}}
\put(105,7.5){\line(0,1){127.5}}

\put(120,22.5){\circle*{4}}
\put(120,22.5){\line(1,0){52.5}}
\put(120,22.5){\line(0,1){112.5}}

\put(135,52.5){\circle*{4}}
\put(135,52.5){\line(1,0){37.5}}
\put(135,52.5){\line(0,1){82.5}}

\put(150,67.5){\circle*{4}}
\put(150,67.5){\line(1,0){22.5}}
\put(150,67.5){\line(0,1){67.5}}

\put(165,112.5){\circle*{4}}
\put(165,112.5){\line(1,0){7.5}}
\put(165,112.5){\line(0,1){22.5}}

\put(41,4.5){ \ }
\put(56,4.5){$+$}
\put(71,4.5){$+$}
\put(86,4.5){$+$}

\put(41,19.5){$\ $}
\put(56,19.5){$  $}
\put(71,19.5){$+ $}
\put(86,19.5){$+$}

\put(56,49.5){$\ $}
\put(71,49.5){$\ $}
\put(86,49.5){$+$}

\put(56,64.5){$\ $}
\put(71,64.5){$\ $}
\put(86,64.5){$ +$}

\put(86,109.5){$\ $}
\end{picture}
\]
The reader can check that the associated reduced word for this starting pipe dream is
$s_7 s_6 s_4 s_5 s_2 s_3 s_4=w_0w=135824679$.

After conjugating and rotating, the pipe dreams are precisely those which can
be obtained via a sequence of local moves from (\ref{eqn:transformsconj}) the
starting pipe dream depicted in (\ref{eqn:startingpipe}). These are then
in natural bijection with the semistandard tableaux with row bounds ${\bf b}=(1,3,4)$:
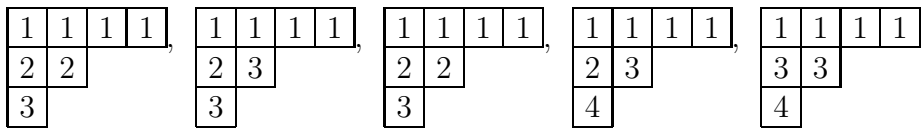
\[\tableau{{1}&{1}&{1}&{1}\\{2}&{2}\\{3}},\ \
\tableau{{1}&{1}&{1}&{1}\\{2}&{3}\\{3}},\  \
\tableau{{1}&{1}&{1}&{1}\\{2}&{2}\\{3}},\ \
\tableau{{1}&{1}&{1}&{1}\\{2}&{3}\\{4}},\ \
\tableau{{1}&{1}&{1}&{1}\\{3}&{3}\\{4}}.
\]
Hence ${\rm mult}_{v}(X_w)=5$.
\qed \end{example}

The weight generating series
\[\sum_{T}{\bf x}^{{\rm wt}(T)}\]
where the sum runs over all semistandard tableaux of shape $\lambda$ with row
entries {\bf flagged} (bounded) by a vector ${\bf b}$ is called the {\bf
  flagged Schur polynomial}. A standard Gessel--Viennot type
argument shows that
\[\sum_{T}{\bf x}^{{\rm wt}(T)} = \det(h_{\lambda_i-i+j}(x_1,\ldots x_{b_i})),\]
where $h_d(x_1,\ldots,x_b)$ is the complete homogeneous symmetric function of
degree $d$ in
the variables $x_1,\ldots,x_b$.  See
\cite[Cor~2.6.3]{Manivel} for details.

By setting $x_1=x_2=\ldots = 1$ into this formula, we obtain a formula for
the multiplicity as a determinant of a matrix with binomial coefficient
entries.

The above discussion therefore proves the following theorem:
\begin{theorem}
Let $w\in S_n$ be a co-Grassmannian permutation with unique ascent at position
$k$. Then if $v\leq w$, we have that $v_{\rm max}$ is co-Grassmannian with
unique ascent at position $k$, and ${\rm mult}_{e_v}(X_w)$ equals the number
of semistandard flagged Young tableau of shape $\lambda=\lambda(w)$ flagged by
the vector ${\bf b}$ given by
\[b_m=\max_{i}\{\lambda(v_{\rm max})_{i}\geq \lambda(w)_{m}+i-m\}.\]
In addition, we have the determinantal formula
\[{\rm mult}_{e_v}(X_w)=\det\left({b_i+\lambda_i-i+j-1\choose \lambda_i-i+j}\right)_{1\leq i,j\leq \ell(\lambda)},\]
where $\ell(\lambda)$ is the number of nonzero parts of $\lambda$.
\end{theorem}

\begin{example}
Continuing the previous example, we have
\[{\rm mult}_{e_v}(X_w)=\left(
\begin{matrix}
b_1+\lambda_1-1 \choose \lambda_1 & b_1 +\lambda_1 \choose \lambda_1+1 & b_1+\lambda_1+1 \choose \lambda_1+2 \\
b_2+\lambda_2-2 \choose \lambda_2-1 & b_2+\lambda_2-1 \choose \lambda_2 & b_2+\lambda_2 \choose \lambda_2+1\\
b_3+\lambda_3-3 \choose \lambda_3-2 & b_3+\lambda_3-2\choose \lambda_3-1 & b_3+\lambda_3-1 \choose \lambda_3
\end{matrix}
\right)
=
\left(
\begin{matrix}
4 \choose 4 & 5 \choose 5 & 6 \choose 6\\
3 \choose 1 & 4 \choose 2 & 5 \choose 3\\
2 \choose -1 & 3 \choose 0 & 4 \choose 1
\end{matrix}
\right)=5,
\]
in agreement with our previous computation.
\qed \end{example}

Although our formula is also a determinant of a matrix of binomial
coefficients, it is different from the ones given in \cite{Lakshmibai.Weyman}
and \cite{Rosenthal.Zelevinsky}.  Presumably it would not be difficult to show
the formulas are equivalent through a succession of determinantal identities.
Our tableau rule is also different than the one given by \cite{Krattenthaler}
to explain the positivity and equivalence of these two determinantal
expressions; his rule instead counts semistandard tableaux of an irregular
shape satisfying certain column and row bounds.  The intermediate pipe dream
arguments we use, as we have said, are closely related to \cite{KMY}. However,
they also appear in later work, specifically of V.~Kreiman \cite{Kreiman} and
T.~Ikeda and H.~Naruse \cite{Ikeda.Naruse}, for reasons similar to ours.  The
rather trivial distinction is that these authors focus on $T$-fixed points on
the Grassmannian itself rather than on Schubert varieties in the flag variety
indexed by co-Grassmannian permutations as we do here.

Our proof gives a Gr\"{o}bner
geometry explanation of the appearance of tableaux: the multiplicity is the
degree of the Kazhdan--Lusztig ideal $I_{v,w}$, which can via
Theorems~\ref{thm:main1} and~\ref{thm:prime} be thought of geometrically as
the number of components of ${\rm in}_{\prec}I_{v,w}$ and combinatorially as
the number of pipe dreams.  These pipe dreams are in this case transparently
in bijection with flagged semistandard tableaux.

Perhaps notable is the appearance of \emph{flagged} tableaux in our formulas
for the multiplicity. Flagged tableaux and flagged Schur functions most often
appear in the combinatorics of co-vexillary permutations and their Schubert
polynomials. While co-Grassmannian permutations are co-vexillary, our proof
does not extend in general to cases where $w$ or $v$ is
co-vexillary. Algebraically, this amounts to the fact that the Kazhdan--Lusztig
ideal is no longer homogeneous with respect to the standard grading (even if
one replaces $v$ by $v_{\rm max}$).  Instead, a more refined degeneration
argument is needed; see \cite{LiYong} (which uses results of this paper).

The argument we use should work more generally to give formulas for
$X_w\subseteq G/B$ when $w_0w$ is $\lambda$-cominuscule.  In particular, one
should be able to obtain determinantal formulas for multiplicities of Schubert
varieties of $G/P$ where $P$ is a co-minuscule maximal parabolic.

\section{Proof of Theorems~\ref{thm:main1} and~\ref{thm:prime}}

Let $R={\mathbb C}[{\bf z}^{(v)}]$ and suppose
\[I_{v,w}:=\langle m_1,\ldots,m_N\rangle\subseteq R\]
where $m_1,\ldots,m_N$ (for some $N$)
are the essential minors, as defined in Section 2.2.
Let
\begin{equation}
\label{eqn:trivial_containment}
J_{v,w}:=\langle {\rm LT}_{\prec}(m_1),\ldots,{\rm LT}_{\prec}(m_N)\rangle\subseteq
{\rm in}_{\prec}I_{v,w}
\end{equation}
be the ideal generated by the leading terms of the essential minors, with
respect to the term order $\prec$ defined in Section~1.3. By definition,
the containment is
an equality if and only if $\{m_1,\ldots,m_N\}$ is a {\bf Gr\"{o}bner basis}
with respect to $\prec$.

%
%
%

The key technical step for our proof of Theorems~\ref{thm:main1}
and~\ref{thm:prime} is the following.

\begin{theorem}
\label{thm:nonfacepipedreamargument}
If a pipe dream ${\mathcal P}$ does not label a face of $\Delta_{v,w}$, then
the corresponding monomial ${\bf z}^{\mathcal P}$ in $R$ is divisible by one
of the leading terms ${\rm LT}_{\prec}(m_1),\ldots, {\rm LT}_{\prec}(m_N)$ of
an essential minor of $J_{v,w}$.
\end{theorem}

The converse also holds, and follows from Theorem~\ref{thm:main1}, but we will
not need this.

We also need the following theorem.

\begin{theorem}
\label{thm:subwordKpoly}
Let $R/K_{v,w}$ be the Stanley--Reisner ring of $\Delta_{v,w}$.  Then
\[\mathcal{K}(R/K_{v,w},\mathbf{t})=\mathcal{K}(R/I_{v,w},\mathbf{t}),\]
where the $K$-polynomials are calculated relative to the grading given by the usual
action defined in Section~4.1.
\end{theorem}

Delaying the proof of Theorems~\ref{thm:nonfacepipedreamargument}
and~\ref{thm:subwordKpoly}, we are now ready to give the:

\noindent
\emph{Conclusion of the proof of Theorem~\ref{thm:main1}:}
Let
\[K_{v,w}= \mbox{the Stanley--Reisner ideal of the pipe complex
$\Delta_{v,w}$}.\]
Theorem~\ref{thm:nonfacepipedreamargument} implies that if ${\mathcal P}$
is a nonface of $\Delta_{v,w}$ then ${\mathcal P}$ is in $J_{v,w}$. Hence
\begin{equation}
\label{eqn:KinsideJ}
K_{v,w}\subseteq J_{v,w}\subseteq {\rm in}_{\prec}I_{v,w},
\end{equation}
where the latter containment reiterates (\ref{eqn:trivial_containment}).

So by (\ref{eqn:KinsideJ}) we have surjections
\[R/K_{v,w} \twoheadrightarrow R/J_{v,w}\twoheadrightarrow
R/{\rm in}_{\prec}I_{v,w}.\] Theorem~\ref{thm:subwordKpoly} states that
\[\mathcal{K}(R/K_{v,w},\mathbf{t})=\mathcal{K}(R/I_{v,w},\mathbf{t}).\]
Hence
the above containments are actually equalities, and
\[K_{v,w}=J_{v,w}={\rm in}_{\prec} I_{v,w}.\]
Thus the essential minors of $I_{v,w}$ are Gr\"{o}bner with respect to
$\prec$.  Moreover $\Delta_{v,w}$ is the Stanley--Reisner complex of the
initial ideal.

The above argument only proves the theorem when $\Bbbk={\mathbb C}$, since we
have used facts about Schubert varieties that are proved only for that case.
(See, specifically, the proof of Theorem~\ref{thm:subwordKpoly}.) However, the
general case follows since all coefficients of terms in the essential minors
are $\pm 1$.  \qed

\noindent
\emph{Conclusion of the proof of Theorem~\ref{thm:prime}:}
In the proof of Theorem~\ref{thm:main1} we saw that
$\Delta_{v,w}$ is the
Stanley--Reisner complex of the initial ideal ${\rm in}_{\prec}I_{v,w}$.
By Proposition~\ref{prop:subwordtranslated}, $\Delta_{v,w}$ is
homeomorphic to a ball or sphere. In particular it is equidimensional with
the stated facets. Hence the prime decomposition claim follows . \qed

In the remainder of this section, we prove
Theorems~\ref{thm:nonfacepipedreamargument} and~\ref{thm:subwordKpoly}.

\subsection{Vertex decompositions of simplical complexes}

Given a simplicial complex $\Delta$ and a vertex $V\in\Delta$,
the {\bf deletion} of
$V$ is the set of the faces of $\Delta$ that do not contain $V$:
\[\del_{V}(\Delta)=\{F\in \Delta:V\not\in F\}.\]
The {\bf link} of $V$ consists of the faces in the deletion which remain faces
of $\Delta$ after adding $V$:
\[\link_{V}(\Delta)=\{F\in \Delta:V\not\in F, \ F\cup\{V\}\in \Delta\},\]
and the {\bf star} of $V$ is:
\[{\rm star}_{V}(\Delta)={\rm cone}_{V}(\link_{V}(\Delta)).\]
Moreover, one has the decomposition
\[\Delta=\del_{V}(\Delta)\cup {\rm star}_{V}(\Delta).\]
A good case is when $\del_{V}(\Delta)$ is homeomorphic to a dimension $d$-ball $B^d$
and $\link_{V}(\Delta)\cong B^{d-1}$ is on the spherical surface of
${\rm star}_{V}(\Delta)$,
whence we can deduce $\Delta\cong B^{d}$. This gives an inductive way of proving ballness
(or similarly sphereness) of $\Delta$ by a good ordering of the vertices of $\Delta$, and
additionally implies a shelling of $\Delta$.

The above type of reasoning was introduced in \cite{Billera.Provan} and
exploited in \cite{Knutson.Miller} to deduce that all subword complexes, and
therefore the complexes $\Delta_{v,w}$, are balls or spheres. We will use this
inductive framework to prove Theorem 6.1. In \cite{KMY}, A.~Knutson, E.~Miller
and the second author developed a theory of ``geometric vertex
decompositions'' by which one can inductively deduce Gr\"{o}bnerness of a
generating set of an ideal $I$ by the Gr\"{o}bnerness of a related generating
set for a partial Gr\"{o}bner degeneration $I'$ of $I$. However, we will not
use this theory, but rather base our induction on the following observation,
whose proof is immediate from the definitions:

\begin{lemma}
\label{lemma:tautological}
Let $S$ be a set of vertices on any simplicial complex
$\Delta$ and $V$ a vertex of $\Delta$.

The set $S$ is a nonface of $\Delta$ if and only if either:
\begin{enumerate}
\item The vertex $V\in S$, and $S\setminus\{V\}$ is a nonface of
$\link_{V}(\Delta)$.
\item The vertex $V\not\in S$, and $S$ is a nonface of the
$\del_{V}(\Delta)$.
\end{enumerate}
\end{lemma}

Let $z_{\rm last}$ be the largest letter under the term order $\prec$ (which
is the rightmost, then southmost variable appearing in $Z^{(v)}$). For the
remainder of this paper, let $V$ be the associated vertex of $\Delta_{v,w}$,
which is the vertex labeled by the pipe dream having a cross $\cross$ in
every position of $D(v)$ \emph{except} at the position of $z_{\rm last}$.

\begin{theorem}[See~\cite{Knutson.Miller:subword}, proof of Theorem~2.5]
\label{thm:vertex_dec}
Let $i$ be the last ascent of $v$ and $s_i$ the corresponding adjacent
transposition.
\begin{itemize}
\item[(A)]
If $i$ is an descent of $w$, then $V$ is a cone point of
$\Delta_{v,w}$, and
\[\link_{V}(\Delta_{v,w})=\del_{V}(\Delta_{v,w})\cong \Delta_{vs_i,w}.\]
\item[(B)]
If $i$ is an ascent of $w$, then
\[\link_{V}(\Delta_{v,w})\cong \Delta_{vs_i,w} \mbox{ \ \ and \ \ }
\del_{V}(\Delta_{v,w})\cong \Delta_{vs_i,ws_i}.\]
\end{itemize}
\end{theorem}

In \cite{Knutson.Miller:subword}, the authors find a vertex decomposition of
any subword complex at the vertex associated to the last letter of $Q$. In
view of Proposition~\ref{prop:subwordtranslated}, we can deduce
Theorem~\ref{thm:vertex_dec} from their results (or proof). That said, to
exploit the vertex decomposition in our proof of
Theorem~\ref{thm:nonfacepipedreamargument}, we need to have handy the specific
homeomorphisms we use in terms of the pipe complex.  (See claims inside the
proof below.)  Explaining them amounts to a proof of
Theorem~\ref{thm:vertex_dec} anyway.

Our proofs of Theorems~\ref{thm:nonfacepipedreamargument}
and~\ref{thm:vertex_dec} use the following straightforward fact:
\begin{lemma}
\label{lemma:diagramchasing}
Let $v\in S_n$ and $i$ be the last ascent of $v$. Then the placement of
boxes of $D(v)$ and $D(vs_i)$ agree in all columns except $i$ and $i+1$.
Moreover, to obtain $D(vs_i)$ from $D(v)$, move all the boxes of $D(v)$
in column $i$ strictly above row $v(i+1)$ one box to the right, and also
delete the box (that must appear) in position $(v(i+1),i)$ of $D(v)$.
\end{lemma}

Also, we need:

\begin{lemma}
\label{lemma:lastisi}
Let $i$ be the last ascent of $v$. The rightmost and southmost box of $D(v)$
(which is the position of $z_{\rm last}$) is in column $i$, and, moreover, the
canonical labeling fills that box with $i$.
\end{lemma}
\begin{proof}
Note that the label of a box under the canonical labeling is the ``Manhattan
distance'' of that box from the southwest corner minus the number of dots
southwest of the box.  From this description, the canonical labeling is the
same whether it is defined by filling boxes as one reads along the rows or the
columns. The Lemma follows from the latter description.
\end{proof}

\begin{example}
Let $v=142653$. Then the last ascent occurs at $i=3$. In
Figure~\ref{fig:142653} we draw the diagrams for $D(v)$ and $D(vs_i)$,
as an illustration of Lemma~\ref{lemma:diagramchasing} and~\ref{lemma:lastisi}.

\begin{figure}[h]
\begin{picture}(300,90)
\put(37.5,0){\makebox[0pt][l]{\framebox(90,90)}}
\put(37.5,30){\line(1,0){45}}
\put(67.5,30){\line(0,-1){30}}
\put(37.5,15){\line(1,0){45}}
\put(82.5,30){\line(0,-1){30}}
\put(52.5,0){\line(0,1){75}}
\put(52.5,75){\line(-1,0){15}}
\put(52.5,60){\line(-1,0){15}}
\put(52.5,45){\line(-1,0){15}}
\thicklines
\put(45,82.5){\circle*{4}}
\put(45,82.5){\line(1,0){82.5}}
\put(45,82.5){\line(0,1){7.5}}
\put(105,22.5){\circle*{4}}
\put(105,22.5){\line(1,0){22.5}}
\put(105,22.5){\line(0,1){67.5}}
\put(60,37.5){\circle*{4}}
\put(60,37.5){\line(1,0){67.5}}
\put(60,37.5){\line(0,1){52.5}}
\put(75,67.5){\circle*{4}}
\put(75,67.5){\line(1,0){52.5}}
\put(75,67.5){\line(0,1){22.5}}
\put(90,7.5){\circle*{4}}
\put(90,7.5){\line(1,0){37.5}}
\put(90,7.5){\line(0,1){82.5}}
\put(120,52.5){\circle*{4}}
\put(120,52.5){\line(1,0){7.5}}
\put(120,52.5){\line(0,1){37.5}}

\thinlines
\put(67.5,45){\line(1,0){15}}
\put(82.5,45){\line(0,1){15}}
\put(82.5,60){\line(-1,0){15}}
\put(67.5,45){\line(0,1){15}}
\put(41,2.5){$1$}
\put(41,17.5){$2$}
\put(41,32.5){$3$}
\put(41,47.5){$4$}
\put(41,62.5){$5$}
\put(56,2.5){$2$}
\put(56,17.5){$3$}
\put(71,2.5){$3$}
\put(71,17.5){$4$}
\put(71,47.5){$5$}


\put(197.5,0){\makebox[0pt][l]{\framebox(90,90)}}
\put(197.5,30){\line(1,0){30}}
\put(227.5,30){\line(0,-1){30}}
\put(197.5,15){\line(1,0){30}}
\put(242.5,30){\line(0,-1){15}}
\put(242.5,30){\line(1,0){15}}
\put(257.5,30){\line(0,-1){15}}
\put(257.5,15){\line(-1,0){15}}
\put(212.5,0){\line(0,1){75}}
\put(212.5,75){\line(-1,0){15}}
\put(212.5,60){\line(-1,0){15}}
\put(212.5,45){\line(-1,0){15}}
\thicklines
\put(205,82.5){\circle*{4}}
\put(205,82.5){\line(1,0){82.5}}
\put(205,82.5){\line(0,1){7.5}}
\put(265,22.5){\circle*{4}}
\put(265,22.5){\line(1,0){22.5}}
\put(265,22.5){\line(0,1){67.5}}
\put(220,37.5){\circle*{4}}
\put(220,37.5){\line(1,0){67.5}}
\put(220,37.5){\line(0,1){52.5}}
\put(235,7.5){\circle*{4}}
\put(235,7.5){\line(1,0){52.5}}
\put(235,7.5){\line(0,1){82.5}}
\put(250,67.5){\circle*{4}}
\put(250,67.5){\line(1,0){37.5}}
\put(250,67.5){\line(0,1){22.5}}
\put(280,52.5){\circle*{4}}
\put(280,52.5){\line(1,0){7.5}}
\put(280,52.5){\line(0,1){37.5}}

\thinlines
\put(242.5,45){\line(1,0){15}}
\put(257.5,45){\line(0,1){15}}
\put(257.5,60){\line(-1,0){15}}
\put(242.5,45){\line(0,1){15}}
\put(201,2.5){$1$}
\put(201,17.5){$2$}
\put(201,32.5){$3$}
\put(201,47.5){$4$}
\put(201,62.5){$5$}
\put(216,2.5){$2$}
\put(216,17.5){$3$}
\put(246,17.5){$4$}
\put(246,47.5){$5$}

\put(-5,45){$D(v)=$}
\put(144,45){$D(vs_i)=$}
\end{picture}
\caption{\label{fig:142653} $D(v)$ versus $D(vs_i)$}
\end{figure}
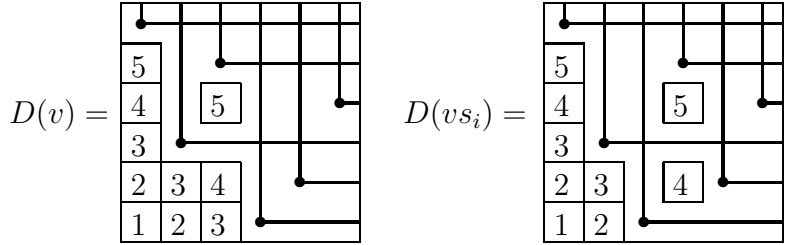

\qed \end{example}

\noindent
\emph{Proof of Theorem~\ref{thm:vertex_dec}:}
We first prove (A). Since the facets of
$\Delta_{v,w}$ are pipe dreams ${\mathcal P}$ such that
$\prod{\mathcal P}$ is a reduced word for $w_0w$, the assertion that $V$ is
a cone point of $\Delta_{v,w}$ amounts to the following:

\begin{claim}
\label{claim:everyface}
No reduced pipe dream for $w_0w$ puts a
$+$ at $z_{\rm last}$.
\end{claim}
\noindent
\emph{Proof of Claim~\ref{claim:everyface}:}
Combine  Lemma~\ref{lemma:lastisi}
with the hypothesis that $i$ is an ascent of $w_0w$.
\qed

Whenever one vertex decomposes at a cone point, the link is automatically
equal to (rather than merely being a subset of) the deletion. Thus we
complete (A) by proving:

\begin{claim}
\label{claim:firsthomeo}
The homeomorphism between
\[\del_{V}(\Delta_{v,w})=\link_{V}(\Delta_{v,w}) \mbox{\ \  and
$\Delta_{vs_i,w}$}\]
is obtained as follows: given a pipe dream
${\rm Pipe}({\mathcal P})$ of a face ${\mathcal P}\in
\del_{V}(\Delta_{v,w})$, construct a pipe dream
${\rm Pipe}({\widetilde {\mathcal P}})$ of a face ${\widetilde {\mathcal P}}$
of $\Delta_{vs_i,w}$ by first deleting the $+$ in the position of
$z_{\rm last}$ and moving each remaining $+$ in column $i$ of
${\rm Pipe}({\mathcal P})$ one step to the right into column $i+1$.
\end{claim}
\noindent
\emph{Proof of Claim~\ref{claim:firsthomeo}:}
By Lemma~\ref{lemma:diagramchasing}, it follows that
${\widetilde {\mathcal P}}$ is a pipe dream for $D(vs_i)$.
Since
\[{\mathcal P}\in \del_{V}(\Delta_{v,w}),\]
${\rm Pipe}({\mathcal P})$
has a $+$ at the position of $z_{\rm last}$. Now by Lemma~\ref{lemma:lastisi},
$\prod {\mathcal P}$ has an $s_i$ at the right end of the Demazure product, and this
product is by assumption equal to $w_0 w$. But $w_0 w$ has an ascent at position
$i$, so it follows that the same Demazure product with $s_i$ removed still gives
$w_0 w$. This latter product is the same as $\prod {\widetilde {\mathcal P}}$,
so
\[{\widetilde {\mathcal P}} \in \Delta_{vs_i,w}.\]

It is also clear that the map ${\mathcal P}\mapsto {\widetilde {\mathcal P}}$
is injective and reversible and preserves face containment.  Thus the
conclusion follows.
\qed

Now we prove (B).
Let us first analyze $\link_{V}(\Delta_{v,w})$, which by definition
consists of all faces ${\mathcal P}\in \Delta_{v,w}$ that
\begin{itemize}
\item[(a)] do not contain $V$ (being in $\del_{V}(\Delta_{v,w})$)
\item[(b)] but satisfy ${\mathcal P}\cup V\in \Delta_{v,w}$.
\end{itemize}

Translating, (a) says that ${\rm Pipe}({\mathcal P})$ uses a $+$ in position
$z_{\rm last}$, whereas (b) says that removing that $+$ still gives
a face of $\Delta_{v,w}$. In view of this, we have:

\begin{claim}
\label{claim:secondhomeo}
The homeomorphism between
\[\link_{V}(\Delta_{v,w}) \mbox{\  and
$\Delta_{vs_i,w}$}\]
is obtained with as similar map as in Claim~\ref{claim:firsthomeo}: given a pipe dream
${\rm Pipe}({\mathcal P})$ of a face ${\mathcal P}\in
\link_{V}(\Delta_{v,w})$,
we construct a pipe dream
${\rm Pipe}({\widetilde {\mathcal P}})$ of a face ${\widetilde {\mathcal P}}$
of $\Delta_{vs_i,w}$ by first deleting the $+$ in the position of
$z_{\rm last}$ and moving each remaining $+$ in column $i$ of
${\rm Pipe}({\mathcal P})$ one step to the right, into column $i+1$. The same map describes
a homeomorphism between $\del_{V}(\Delta_{v,w})$ and $\Delta_{vs_i,ws_i}$.
\end{claim}
\begin{proof}
This proof is similar to that for Claim~\ref{claim:firsthomeo}.
The key point of the link claim is that removing the $+$ in position
$z_{\rm last}$ does not change the Demazure product. In the deletion
claim, this removal of a $+$ does change the Demazure product of $w$, but since
$w$ has an ascent at $i$, the resulting Demazure product is $ws_i$ instead.
\end{proof}
The proof of Theorem~\ref{thm:vertex_dec} follows.\qed

\begin{example}
Continuing Example~\ref{exa:31452}, we have that $i=3$ is an ascent of
$w$. The vertex $V$ is the top leftmost vertex of Figure~\ref{fig:31452.53142}
and the link is a $1$-dimensional ball isomorphic to $\Delta_{vs_i,w}$, which
is Example~\ref{exa:31542previous}.  This agrees with
Theorem~\ref{thm:vertex_dec} and Claim~\ref{claim:secondhomeo}.
\qed \end{example}

\subsection{The Kostant--Kumar recursion and the proof of
Theorem~\ref{thm:subwordKpoly}}

S.~Kumar shows~\cite[Theorem 2.2]{Kumar} that the $K$-polynomials (as
defined in Section~\ref{subsection:various})
    \[{\mathcal K}(R/I_{v,w},{\bf t})\]
and hence the equivariant $K$-theory classes
\[[{\mathcal O}_{{\mathcal N}_{v,w}}]_T\in K_T(\Omega_{v}^{\circ}),\]
satisfy the following recursion, originally defined and shown to have
a unique solution by B.~Kostant and S.~Kumar~\cite[Proposition
  2.4]{Kostant.Kumar}.  (They use significantly different language and
notation; our version previously appeared in~\cite[Theorem
  1]{Knutson:patches}.)

\begin{theorem}
Let $v,w\in S_n$.
\begin{itemize}
\item If $v\not\leq w$, then
\[\mathcal{K}(R/I_{v,w},\mathbf{t})=0.\]
\item If $v=w_0$, then $w=w_0$ (or we are in the previous case).  Then
  \[\mathcal{K}(R/I_{v,w},\mathbf{t})=1.\]
\item Otherwise, let $i$ be a right ascent of $v$, so $vs_i>v$.  Then
\begin{enumerate}
\item If $i$ is a descent of $w$, so $ws_i<w$, then
  $$\mathcal{K}(R/I_{v,w},\mathbf{t})=\mathcal{K}(R/I_{vs_i,w},\mathbf{t}).$$
\item If $i$ is an ascent of $w$, so $ws_i>w$, then
\begin{multline}\nonumber
\mathcal{K}(R/I_{v,w},\mathbf{t}) =
\mathcal{K}(R/I_{vs_i,w},\mathbf{t})
+(1-t_{v(i)}/t_{v(i+1)})\mathcal{K}(R/I_{vs_i,ws_i},\mathbf{t})\\ \nonumber
-(1-t_{v(i)}/t_{v(i+1)})\mathcal{K}(R/I_{vs_i,w},\mathbf{t}).
\end{multline}
\end{enumerate}
\end{itemize}
\end{theorem}

Our strategy to prove Theorem~\ref{thm:subwordKpoly} is to show that the
$K$-polynomials under the usual action for the Stanley--Reisner rings
$R/K_{v,w}$ associated to the pipe complexes $\Delta_{v,w}$ satisfy the same
recursion.  Our proof for this fact parallels that of \cite{Knutson:patches}
for subword complexes; we could in fact refer to \cite[Corollary
  2]{Knutson:patches} by showing that the grading on $R$ given by the usual
action matches the grading given in the cited Corollary. However, that
matching demands about as much analysis as the direct argument we give below.

If $v\not\leq w$, then $\Delta_{v,w}$ is the empty complex (since
$w_0v\not\geq w_0w$, so no subwords of a reduced word for $w_0v$ can be a
reduced word for $w_0w$).  Thus, in this case,
\[\mathcal{K}(R/K_{v,w},\mathbf{t})=0.\]

If $v=w=w_0$, then $\Delta_{v,w}$ is the simplicial complex (on zero vertices)
whose only face is the empty face.  Therefore, in this case
\[\mathcal{K}(R/K_{v,w},\mathbf{t})=1.\]

Otherwise, we rely on the vertex decomposition of $\Delta_{v,w}$ given by
Theorem~\ref{thm:vertex_dec}.  First note that the homeomorphisms of
Claims~\ref{claim:firsthomeo} and~\ref{claim:secondhomeo} are
weight preserving by the following argument.  Boxes in $D(v)$ which are not in
column $i$ or $i+1$ remain in the same place, and for $k\neq i,i+1$, the
variable $z_{jk}$ has weight
\[t_{v(k)}-t_{n-j+1}=t_{vs_i(k)}-t_{n-j+1}\]
in both
$\mathbb{C}[\mathbf{z}^{(v)}]$ and $\mathbb{C}[\mathbf{z}^{(vs_i)}]$.  A box
in the column $i$ and some row $j$ of $D(v)$ corresponds to the box in row
$j$ and column $i+1$ in $D(vs_i)$ (unless $n-j+1=v(i+1)$, in which case the
box is deleted).  The weight of $z_{j,i}$ in $\mathbb{C}[\mathbf{z}^{(v)}]$ is
\[t_{v(i)}-t_{n-j+1},\]
which is equal to
\[t_{vs_i(i+1)}-t_{n-j+1},\]
the weight of $z_{j,i+1}$ in $\mathbb{C}[\mathbf{z}^{(vs_i)}]$.

Furthermore, the variable $z_{\mathrm{last}}$ corresponding to $V$ is
$z_{n-v(i+1)+1,i}$.  Therefore, $z_{\mathrm{last}}$ has weight
$t_{v(i)}-t_{v(i+1)}$.

A face of $\Delta_{v,w}$ either contains $V$ or not.  Therefore, it is either
a face of $\del_V(\Delta_{v,w})$, or the union of a face of
$\link_V(\Delta_{v,w})$ with $V$.  Let $\rho(a)$ denote the weight of the
  variable associated to a vertex $a$.  Now \cite[Theorem
    1.13]{Miller.Sturmfels} (with the appropriate substitution to account for
  our use of the usual action rather than the rescaling action) states
  that $$\mathcal{K}(R/K_{v,w},\mathbf{t})=\sum_{\sigma\in\Delta_{v,w}}
  \left(\prod_{a\in\sigma} \mathbf{t}^{\rho(a)} \cdot \prod_{a\not\in\sigma}
  (1-\mathbf{t}^{\rho(a)}) \right).$$

When $i$ is a descent of $w$, Theorem 5.4 (A) asserts
\[\link_V(\Delta_{v,w})=\del_V(\Delta_{v,w})=\Delta_{vs_i,w}.\]
Therefore
\begin{eqnarray}\nonumber
\mathcal{K}(R/K_{v,w},\mathbf{t}) & = &
t_{v(i)}/t_{v(i+1)}\mathcal{K}(R/K_{vs_i,w},\mathbf{t})+
(1-t_{v(i)}/t_{v(i+1)})\mathcal{K}(R/K_{vs_i,w},\mathbf{t})\\\nonumber
& = & \mathcal{K}(R/K_{vs_i,w},\mathbf{t}).\nonumber
\end{eqnarray}
When $i$ is an ascent of $w$, Theorem 5.4 (B) states that
\[\link_V(\Delta_{v,w})=\Delta_{vs_i,w} \mbox{\ \ and \ \
$\del_V(\Delta_{v,w})=\Delta_{vs_i,ws_i}$.}\]
Therefore,
\begin{eqnarray}\nonumber
\mathcal{K}(R/K_{v,w},\mathbf{t}) &
= & t_{v(i)}/t_{v(i+1)}\mathcal{K}(R/K_{vs_i,w},\mathbf{t})
+ (1-t_{v(i)}/t_{v(i+1)})\mathcal{K}(R/K_{vs_i,ws_i},\mathbf{t})\\\nonumber
& = & \mathcal{K}(R/K_{vs_i,w},\mathbf{t})
+ (1-t_{v(i)}/t_{v(i+1)})\mathcal{K}(R/K_{vs_i,ws_i},\mathbf{t})\\ \nonumber
& & \ \ \ \ \ \ \ \ \ \ \ \ \ \ \ \ \ \ \ \ \ \ \ \ \ \ \ \ \ \ \
- (1-t_{v(i)}/t_{v(i+1)})\mathcal{K}(R/K_{vs_i,w},\mathbf{t}).\nonumber
\end{eqnarray}

Therefore, the $K$-polynomials for $R/K_{v,w}$ satisfy the same recurrence
relations as the $K$-polynomials for $R/I_{v,w}$, and hence they are
equal.\qed

\subsection{Proof of Theorem~\ref{thm:nonfacepipedreamargument}}
The following is immediate from
Lemma~\ref{lemma:diagramchasing}:
\begin{deflemma}
Given $f\in {\mathbb C}[{\bf z}^{(vs_i)}]$
let $f^{\circ}$ be obtained
by the substitution $z_{j,i+1}\mapsto z_{j,i}$. Then
$f^\circ\in {\mathbb C}[{\bf z}^{(v)}]$.
\end{deflemma}

In what follows, let $\prec$ refer to the lexicographic term order we use on
$Z^{(vs_i)}$ and $\prec_{\circ}$ be the term order on $Z^{(v)}$; see
Section~2.3 for a definition of this term order.

The following is clear:

\begin{lemma}
\label{lemma:circlead}
If ${\mathcal L}$ is the leading term of $f\in {\mathbb C}[{\bf z}^{(vs_i)}]$
with respect to
$\prec$, then ${\mathcal L}^{\circ}$ is the leading term of
$f^{\circ} \in {\mathbb C}[{\bf z}^{(v)}]$ with respect to
$\prec_{\circ}$.
\end{lemma}

The main technical point of this paper is below:

\begin{proposition}
\label{prop:mainprop1}
Let $i$ be the last ascent of $v\in S_n$. Suppose
\begin{itemize}
\item[(I)] $i$ is a descent of $w\in S_n$ and
${\mathcal L}$ is a leading term of an essential minor of
$I_{vs_i,w}$; or
\item[(II)] $i$ is an ascent of $w\in S_n$ and ${\mathcal L}$
is the leading term of an essential minor of $I_{vs_i,w}$; or
\item[(III)] $i$ is an ascent of $w\in S_n$ and ${\mathcal L}$ is the leading term of
an essential minor of $I_{vs_i,ws_i}$.
\end{itemize}
Then in cases (I) and (III), ${\mathcal L}^{\circ}$ is divisible by the
leading term ${\mathcal L}'$ with respect to $\prec_{\circ}$
of an essential minor
of $I_{v,w}$. In case (II), the same holds for
${\mathcal L}^{\circ}z_{\rm last}$.
\end{proposition}

\noindent\emph{Proof of Proposition~\ref{prop:mainprop1}:}
The basic idea of the proof is as follows: given an essential determinant $D$
in $I_{vs_i,w}$ (or $I_{vs_i,ws_i}$) which uses the submatrix $M$ of
$Z^{(vs_i)}$ and has leading term ${\mathcal L}$, we consider a determinant
$D'$ of $I_{v,w}$ that uses the submatrix $M'$ of $Z^{(v)}$ with the same
columns and rows as $M$ except that if column $i$ is used in $M$, we use
column $i+1$ in $M'$, and vice versa. In view of
Lemma~\ref{lemma:diagramchasing}, \emph{usually} this works to give an
essential determinant
\[D'=\det M'\]
whose leading term $\mathcal{L}'$ has the desired properties.  However, this
sometimes fails, and our analysis below accounts for this.  Each of (I), (II)
and (III) is handled in four subcases, depending on which of the $i$-th and
$i+1$-th columns of $Z^{(vs_i)}$ $M$ uses.

\medskip
\noindent
\underline{ Case I.1 ($M$ uses neither the $i$-th nor $i+1$-th column):}
Let $M'$ be the submatrix that uses the
same rows and columns as $M$. By Lemma~\ref{lemma:diagramchasing},
$Z^{(vs_i)}$ and $Z^{(v)}$ do not differ in these columns, so
we have that
\[D^{\circ}=D=D'.\]
Hence by
Lemma~\ref{lemma:circlead},
\[{\mathcal L}^{\circ}={\mathcal L}={\mathcal L}',\]
so in particular
${\mathcal L}'$ divides ${\mathcal L}^{\circ}$.

\medskip
\noindent
\underline{ Case I.2 ($M$ uses both the $i$-th and $i+1$-th column):} We may assume $D\neq 0$,
hence:
\begin{claim}
\label{claim:theonlynonzero}
The only nonzero entry of $M$ in the $i$-th column
comes from row $vs_i(i)=v(i+1)$ of $Z^{(vs_i)}$, and that entry is a $1$.
\end{claim}
\begin{proof}
Follows from Lemma~\ref{lemma:diagramchasing} and the hypothesis that
$i$ is the last ascent of $v$.
\end{proof}
First construct a submatrix $M''$ of $Z^{(vs_i)}$ by changing
$M$ by replacing the $0$ at position $(vs_i(i),i+1)$
by $z_{n-v(i+1)+1,i+1}$ and switching the $i$-th and $i+1$-th columns.
Now, we can compute $D$ by cofactor expansion along the $i$-th column, so,
by Claim~\ref{claim:theonlynonzero}, changing any entry in
row $vs_i(i)$ of $M$ other than the $1$ in column $i$
does not change the determinant.
Hence
\[D''=\pm D.\]
Let $M'$ be the submatrix of $Z^{(v)}$ that uses the same rows and columns
as $M$. By Lemma~\ref{lemma:diagramchasing},
\[(D'')^{\circ}=\pm D^{\circ}\]
equals the essential determinant $D'$.
Thus by Lemma~\ref{lemma:circlead},
\[{\mathcal L}'=\pm {\mathcal L}^{\circ},\]
and therefore ${\mathcal L}'$ divides ${\mathcal L}^{\circ}$.

\medskip
\noindent
\underline{ Case I.3 ($M$ uses the $i$-th column but not the $i+1$-th column):}
$M$ appears as a
submatrix $M'$ (with determinant $D'$ and leading term ${\mathcal L}'$)
of $Z^{(v)}$ using the same rows and columns except that the
$i$-th column is replaced by the $i+1$-th. Therefore, by
Lemma~\ref{lemma:diagramchasing},
\[D=D'=D^{\circ}\]
and
\[{\mathcal L}={\mathcal L}'={\mathcal L}^{\circ},\]
so ${\mathcal L}'$ divides
${\mathcal L}^{\circ}$. Thus we are done
in this case provided $D'$ is an essential minor of $I_{v,w}$, which follows
from:
\begin{claim}
\label{claim:essentialnochange}
If $i$ is a descent of $w$ then there are no boxes of the essential set
of $D(w)$ in column $i$.
\end{claim}
\begin{proof}
Any box of $D(w)$ in column $i$ must have a box of $D(w)$ to its immediate right.
\end{proof}

\medskip
\noindent
\underline{ Case I.4 ($M$ uses the $i+1$-th column but not the $i$-th column):}
If $M$ does not use row \linebreak
$vs_i(i)=v(i+1)$ then let $M'$ be the submatrix
of $Z^{(v)}$ that uses the same rows and columns of $M$ except that we use
column $i$ instead of column $i+1$. Then by Lemma~\ref{lemma:diagramchasing}
it follows
\[D^{\circ}=D'.\]
Hence, by Lemma~\ref{lemma:circlead},
\[{\mathcal L}^{\circ}={\mathcal L}'.\]
Moreover, $D'$ is still essential
since $M'$ uses columns weakly to the left of those of $M$ (and uses the
same rows).

On the other hand, if
$M$ uses row $vs_i(i)=v(i+1)$, let $j$ be the column of
the entry in the leading term of $D$ used from row $v(i+1)$.
Construct a submatrix $M''$ of $Z^{(vs_i)}$ by replacing column $j$
with column $i$ (with associated determinant $D''$ and leading term ${\mathcal L}''$
with respect to $\prec$).
Note that column $i$ is to the right of column $j$ and
consists only of $0$'s except for a $1$ in row $v(i+1)$.
Now, since by assumption $M$ uses column $i+1$,
$D''$ is essential for $w$.
It is easy to see from the fact $\prec$ is a lexicographic order that
\[{\mathcal L}''=\pm{\mathcal L}/z_{n-v(i+1)+1,j}.\]
Now let $M'$ be the submatrix of $Z^{(v)}$ using the same rows
and columns as $M''$. Then noting that $M''$ uses both columns $i$ and $i+1$, we can repeat the argument of
Case I.2 to see that
\[{\mathcal L'}=\pm({\mathcal L}'')^{\circ}=\pm({\mathcal L}/z_{n-v(i+1)+1,j})^{\circ})=\pm{\mathcal L}^{\circ}/z_{n-v(i+1)+1,j}.\]
Hence, ${\mathcal L'}$ divides ${\mathcal L}^{\circ}$ as desired.

\medskip
\noindent
\underline{ Case II.1 ($M$ uses neither the $i$-th nor $i+1$-th column):} This is proved exactly as in Case I.1.

\medskip
\noindent
\underline{ Case II.2 ($M$ uses both the $i$-th and $i+1$-th columns):} This is proved exactly as in Case I.2.

\medskip
\noindent
\underline{ Case II.3 ($M$ uses only the $i$-th column but not the
$i+1$-th column):} Construct a submatrix $M'$ of $Z^{(v)}$
by taking $M$ and replacing the $i$-th column of $Z^{(vs_i)}$
with the $i$-th column of $Z^{(v)}$ and leaving all other columns of $M$
unchanged. Note that
\[{\mathcal L}^{\circ}={\mathcal L}.\]
However, ${\mathcal L}'$ may not divide ${\mathcal L}^{\circ}$.
Instead we wish to prove
\[{\mathcal L}'=\pm {\mathcal L}^{\circ}z_{\rm last}\]
(and hence ${\mathcal L}'$ divides ${\mathcal L}^{\circ}z_{\rm last}$).

We assert that the position in each column of $M'$ that contributes to
${\mathcal L}'$ is the same as for ${\mathcal L}$ (except that we use $z_{\rm
  last}$ in column $i$ rather than $1$, respectively). This is
straightforward.  Clearly $D'$ is essential for $I_{v,w}$ since it uses the
same rows and columns as $D$.

\medskip
\noindent
\underline{Case II.4 ($M$ uses the $i+1$-th column and not the $i$-th column):} This is proved exactly as in Case I.4.

In the analysis of (III), the main new issue is that we must show that, given an essential minor $D$ for
$ws_i$, the newly constructed minor $D'$ is essential for $w$ instead.

\medskip
\noindent
\underline{Case III.1 ($M$ uses neither the $i$-th nor $i+1$-th column):}
The argument given in Case I.1 constructs a
 determinant $D'$ such that ${\mathcal L}'$ divides ${\mathcal L}^{\circ}$.
It remains to show that $D'$ is essential. The only places where the rank
matrices $R^{ws_i}$ and $R^{w}$ differ
are in column $i$ and rows $t$ for which
\[ws_i(i+1)< t\leq ws_i(i).\]
Note moreover that
no boxes of $D(ws_i)$ lie in this region. Now, let $d$ be the essential set
box of $ws_i$ causing $D$ to be essential.
If $d$ is not in column $i+1$, or $d$ is in column $i+1$ and strictly south
of row $ws_i(i)$, it easily follows that $D'$ is also
essential for $w$. Otherwise $d$ must be in
position
\[(ws_i(i+1)+1,i+1).\]
Then $d$ is no longer even a box of $D(w)$,
but the box $d^{\star}$ in position
\[(ws_i(i+1)+1,i)\] (to the
immediate left of $d$) is in $\Ess(w)$. In addition,
\[R^{w}_{d^{\star}}=R^{ws_i}_{d}-1,\]
and the columns of $D$
are weakly to the left of column $i$.

Consider the positions of the variables in $M$ contributing to the leading
term ${\mathcal L}$ of $D$. Let $M''$ be the submatrix defined by any
$R^{w}_{d^{\star}}=R^{ws_i}_{d}-1$ of these positions.  (For definiteness, we
can take all but the rightmost position, that of variable ${\bf z}_{\rm
  rightmost}$.) Then $M''$ has leading term ${\mathcal L}''$ given by the product
of the aforementioned variables we picked out, so
\[{\mathcal L}={\mathcal L}'' {\bf z}_{\rm rightmost}.\]
(Otherwise we would contradict the fact that ${\mathcal L}$ is a leading term of $D$.) Now let $M'$ be the submatrix of
$Z^{(vs_i)}$ using the same rows and columns as $M''$.
Then ${\mathcal L}'={\mathcal L}''=({\mathcal L}'')^{\circ}$,
and it follows that ${\mathcal L}'={\mathcal L}^{\circ}z_{\rm rightmost}$,
and hence ${\mathcal L}'$ divides ${\mathcal L}^{\circ}$.

\medskip
\noindent
\underline{{Case III.2 ($M$ uses both the $i$-th and $i+1$-th column):}}
Repeat the construction of Case I.2 to obtain a minor $D''$. Note that since
$M$ uses column $i+1$, the essential set box $d$ of $ws_i$ causing $D$ to be
essential is weakly to the right of column $i+1$.  If $d$ is strictly to the
right, then let $D'$ be the determinant of $Z^{(v)}$ that uses the same rows
and columns as $M$. This case follows as in Case I.2. Otherwise, if $d$ lies
in column $i+1$ and is at
\[(ws_i(i+1)+1,i+1),\]
the box $d^{\star}$ to its immediate left
is an essential box
for $w$, of rank one less, as in Case III.1.
Now let $M'$ be the one smaller
minor that uses the same rows and columns as $M$, except that
it excludes row $v(i+1)$ and column $i+1$. Then $D'$ is
essential for $w$ due to box $d^{\star}$.
This case then follows.

\medskip
\noindent
\underline{Case III.3 ($M$ uses column $i$ but not column $i+1$):}
To construct $D'$ we use the same construction as in Case I.3, and
its essentialness follows as in Case III.2.

\medskip
\noindent
\underline{Case III.4 ($M$ uses column $i+1$ but not column $i$):} Use the same construction as in Case I.4. There are
two cases to prove essentialness, paralleling the two subcases of Case I.4. If $M$ does not use row $vs_i(i)$ then we apply
the essential box argument of Case III.1. In the other subcase, we argue essentialness as in Case III.2.
\qed

\medskip
\noindent
\emph{Conclusion of the proof of Theorem~\ref{thm:nonfacepipedreamargument}:}
We induct on $\ell(w_0v)\geq 0$. The base case of $\ell(w_0v)=0$, which is
where $v=w_0$, is trivial since ${\bf z}^{(v)}=\emptyset$. For the induction
step, assume that $\ell(w_0 v)\geq 1$; hence ${\bf z}^{(v)}\neq\emptyset$, and
in particular, we have a last variable $z_{\rm last}$ and associated vertex
$V\in\Delta_{v,w}$, as defined above.

Let ${\mathcal P}$ be a nonface of $\Delta_{v,w}$; we must show ${\bf z}^{\mathcal P}$ is
divisible by the leading term of a defining minor of $I_{v,w}$.

Suppose $V\in \mathcal{P}$. Then (1) of Lemma~\ref{lemma:tautological} asserts $\mathcal{P}\setminus V$
is a nonface of $\link_V(\Delta)$. There are then two cases, depending on whether
$i$ is a descent or ascent of $w$. If $i$ is a descent of $w$, then part (A) of
Theorem~\ref{thm:vertex_dec} says
\[\link_V(\Delta)\cong \Delta_{vs_i,w}.\] Under the
relabeling map of Claim~\ref{claim:firsthomeo}, ${\widetilde {\mathcal P}\setminus V}$
is a nonface of $\Delta_{vs_i,w}$ and hence by induction ${\bf z}^{\widetilde{\mathcal P}\setminus V}$
is divisible by the leading term of a defining minor of $I_{vs_i,w}$. The conclusion then follows from
part (I) of Proposition~\ref{prop:mainprop1}.

If $V\in\mathcal{P}$ and $i$ is an ascent of $w$, then
$\mathcal{P}\setminus V$ is a nonface of $\link_V(\Delta)\cong
\Delta_{vs_i,w}$.  By induction ${\bf z}^{\widetilde{\mathcal
    P}\setminus V}$ is divisible by a leading term of a defining minor
of $I_{vs_i,w}$.  Since ${\bf z}^\mathcal{P}={\bf
  z}^{\widetilde{\mathcal P}\setminus V}z_{\rm last}$, the conclusion
follows from part (II) of Proposition~\ref{prop:mainprop1}.

If $V\not\in\mathcal{P}$, then $\mathcal{P}\setminus V$ is a nonface
of $\del_V(\Delta_{v,w})$.  Depending on whether $i$ is a descent or
ascent of $w$, $\del_V(\Delta_{v,w})=\Delta_{vs_i,w}$ or
$\del_v(\Delta_{v,w})=\Delta_{vs_i,ws_i}$, and part (I) or part (III) of
Proposition~\ref{prop:mainprop1} completes the proof.\qed

\section*{Acknowledgements}

We thank Allen Knutson for a number of inspiring communications and
suggestions and in particular for sharing his observation about
homogeneity and multiplicities of Schubert varieties. We also thank
Rebecca Goldin, Li Li, Hal Schenck and Frank Sottile for helpful
conversations and an anonymous referee for helpful comments and
suggestions.  This paper was partially completed during the NSF VIGRE
supported meeting on ``Combinatorial Algebraic Geometry of Flag
Varieties'' at the University of Iowa; we thank the organizers Megumi
Harada and Julianna Tymoczko.  AY is partially supported by NSF grants
DMS-0601010 and DMS-0901331.  AW thanks St. Olaf College, where he was
employed during most of the time this work was completed, for its
support.  We made extensive use of {\tt
  Macaulay~2} in our investigations.


\begin{thebibliography}{KnuMilYon09}
\bibitem[BerBil93]{Billey.Bergeron}
N.~Bergeron and S.~Billey, \emph{RC-graphs and Schubert polynomials}, Experiment.~Math.~{\bf 2}(1993), no. 4, 257--269.
\bibitem[BilPro79]{Billera.Provan} L.~J.~Billera and J.~S.~Provan, \emph{A decomposition
property for simplicial complexes and its relation to diameters and shellings},
Second International Conference on Combinatorial Mathematics (New York, 1978), New York
Acad.~Sci., New York, 1979, pp.~82--85.
\bibitem[Bil99]{Billey} S.~Billey, \emph{Kostant polynomials and the cohomology
ring for $G/B$}, Duke Math.~J.~{\bf 96}(1999), no.~1, 205--224.
\bibitem[BruHer98]{Bruns.Herzog} W.~Bruns and J.~Herzog, \emph{Cohen--Macaulay
  rings, Revised ed.}, Cambridge studies in advanced mathematics {\bf 39}.
  Cambridge University Press, Cambridge, 1998.
\bibitem[BruVet88]{Bruns.Vetter} W.~Bruns and U.~Vetter, \emph{Determinantal rings},
Monograf\'{i}as de Matem\'{a}tica, {\bf 45}. Instituto de Matem\'{a}tica Pura e
Aplicada, Rio de Janeiro, 1998.
\bibitem[BucRim04]{Buch.Rimanyi} A.~Buch and R.~Rim\'{a}nyi,
\emph{Specializations of Grothendieck polynomials}, C.~R.~Acad.~Sci.~Paris,
Ser.~I {\bf 339}(2004), 1-4.
\bibitem[FomKir94]{Fomin.Kirillov} S.~Fomin and A.~N.~Kirillov,
\emph{Grothendieck polynomials and the Yang-Baxter equation},
Proc.~6th Intern.~Conf.~on Formal Power Series and Algebraic
Combinatorics, DIMACS, 1994, 183--190.
\bibitem[FomKir96]{Fomin.Kirillov:B} \bysame,
\emph{Combinatorial $B_n$-analogues of Schubert polynomials},
Trans.~A.~M.~S., {\bf 348}(1996), 3591--3620.
\bibitem[Ful92]{Fulton:Duke92} W.~Fulton, \emph{Flags, Schubert polynomials,
degeneracy loci, and determinantal formulas}, Duke Math.~J.~{\bf 65} (1992),
 no.~3, 381--420.
\bibitem[Ful97]{Fulton:YT} \bysame, \emph{Young tableaux},
London Mathematical Society Student Texts, 35. Cambridge University Press,
Cambridge, 1997.
\bibitem[GhoRag06]{Ghorpade.Raghavan} S.~Ghorpade and K.~N.~Raghavan,
  \emph{Hilbert functions of points on Schubert varieties in the symplectic
    Grassmannian}, Trans. Amer. Math. Soc. {\bf 358} (2006), 5401-5423.
\bibitem[Gol01]{Goldin}
R.~F.~Goldin, \emph{The cohomology ring of weight varieties and polygon spaces},
Adv.~Math.~{\bf 160}(2001), no. 2, 175--204.
\bibitem[HamPit08]{Pittel} A.~Hammett and B.~Pittel, \emph{How often are two permutations comparable?}
Trans.~A.~M.~S. {\bf 360}(2008), no.~9, 4541--4568.
\bibitem[Hoc77]{Hochster} M.~Hochster, \emph{Cohen-Macaulay Rings,
  combinatorics, and simplicial complexes}, Ring theory, II (Proc. Second
    Conf., Univ. Oklahoma, Norman, OK, 1975) (B.R.~Macdonald and R.~Morris,
    eds.) Lecture Notes in Pure and Applied Mathematics {\bf 26}, Marcel
    Dekker, New York (1977), 171--223.
\bibitem[IkeNar07]{Ikeda.Naruse} T.~Ikeda and H.~Naruse, \emph{Excited
  Young diagrams and equivariant Schubert calculus},
  Trans. Amer. Math. Soc.~{\bf 361} (2009), no. 10, 5193-5221.
\bibitem[KazLus79]{Kazhdan.Lusztig} D.~Kazhdan and G.~Lusztig,
\emph{Representations of Coxeter Groups and Hecke Algebras},
Invent.~Math.~{\bf 53} (1979), 165--184.
\bibitem[Knu08]{Knutson:patches}
\bysame, \emph{Schubert patches degenerate to subword complexes},
Transform. Groups {\bf 13} (2008), 715--726.
\bibitem[Knu09]{Knutson:frob} \bysame, \emph{Frobenius splitting,
  point counting, and degeneration}, preprint 2009.  \textsf{arXiv:0911.4941}
\bibitem[KnuMil04]{Knutson.Miller:subword} A.~Knutson and E.~Miller,
\emph{Subword complexes in
Coxeter groups}, Adv.~Math.~{\bf 184}(2004), no.~1, 161--176.
\bibitem[KnuMil05]{Knutson.Miller} \bysame,
\emph{Gr\"{o}bner geometry of
Schubert polynomials}, Ann.~of Math. (2) {\bf 161} (2005), no.~3, 1245--1318.

\bibitem[KnuMilYon05]{KMY} A.~Knutson, E.~Miller and A.~Yong,
\emph{Gr\"{o}bner geometry of vertex decompositions, and of flagged tableaux},
J.~Reine Agnew Math.~{\bf 630} (2009), 1--31.
\bibitem[KosKum90]{Kostant.Kumar} B.~Kostant and S.~Kumar,
\emph{$T$-equivariant $K$-theory of generalized flag varieties},
J.~Differential Geom.~{\bf 32}(1990), no.~2, 549--603.
\bibitem[Kra01]{Krattenthaler} C.~Krattenthaler,
\emph{On multiplicities of points on Schubert varieties in Grassmannians},
S\'{e}minaire~Lotharingien~Combin.~{\bf 45}(2001), Article B45c, 11 pp.
\bibitem[Kre08]{Kreiman} V.~Kreiman, \emph{Schubert classes in the equivariant
$K$-theory and equivariant cohomology of the Grassmannian}, preprint 2006.  \textsf{arXiv:math.AG/0512204}
\bibitem[KreLak04]{Kreiman.Lakshmibai} V.~Kreiman and V.~Lakshmibai,
\emph{Multiplicities at Singular Points of Schubert Varieties in the
Grassmannian}, Algebra, arithmetic and geometry with applications,
553--563, Springer, Berlin, 2004.
\bibitem[Kum96]{Kumar} S.~Kumar, \emph{The nil Hecke ring and singularity of Schubert varieties}, Invent.~Math.~{\bf 123} (1996), no. 3, 471--506.
\bibitem[LakWey90]{Lakshmibai.Weyman} V.~Lakshmibai and J.~Weyman,
\emph{Multiplicities of points on a Schubert variety in a minuscule $G/P$},
Adv.~Math.~{\bf 84}(1990), no. 2, 179--208.
\bibitem[LasSch82a]{Lascoux.Schutzenberger1} A.~Lascoux and M.~P.~Sch\"{u}tzenberger,
\emph{Polyn\^{o}mes de Schubert}, C.~R.~Acad.~Sci.~Paris S\'{e}r.~I Math.~{\bf 294} (1982),
no.~13, 447--450.
\bibitem[LasSch82b]{Lascoux.Schutzenberger2} \bysame, \emph{Structure de Hopf de l'anneau
de cohomologie et de l'anneau de Grothendieck d'une vari\'{e}t\'{e} de drapeaux},
C.~R.~Acad.~Sci.~Paris S\'{e}r.~I Math. {\bf 295} (1982), no.~11m 629--633.
\bibitem[LiYon10]{LiYong} L.~Li and A.~Yong, \emph{Some degenerations of Kazhdan-Lusztig ideals and multiplicities of Schubert varieties},
Adv.~Math., Volume 229, Issue 1, {\bf 15}(2012), 633--667. \textsf{arXiv:1001.3437}
\bibitem[Man01]{Manivel} L.~Manivel, \emph{Symmetric functions, Schubert polynomials and
degeneracy loci}, American Mathematical Society, Providence 2001.
\bibitem[MilStu05]{Miller.Sturmfels} E.~Miller and B.~Sturmfels, \emph{Combinatorial commutative
algebra}, Springer-Verlag, 2005.
\bibitem[Mir07]{MiroRoig} R.~Mir\'{o}-Roig, \emph{Determinantal ideals}, Progress in Mathematics,
{\bf 264}, Birkh\"{a}user Verlag, basel, 2008.
\bibitem[RagUpa07]{Raghavan.Upadhyay} K.~N.~Raghavan and S.~Upadhyay,
  \emph{Hilbert functions of points on Schubert varieties in the orthogonal
    Grassmannian}, J. Algebraic Combin.~{\bf 31} (2010), no. 3, 355--409.
\bibitem[RosZel01]{Rosenthal.Zelevinsky} J.~Rosenthal and A.~Zelevinsky,
\emph{Multiplicities of points on Schubert varieties in Grassmannians},
J.~Algebraic Combinatorics, {\bf 13}(2001), 213--218.
\bibitem[Ste01]{Stembridge} J.~Stembridge, \emph{Minuscule elements of
Weyl groups}, J.~Algebra {\bf 235} (2001), 722--743.
\bibitem[War10]{Warrington:KLExper} G.~Warrington, \emph{Equivalence
  classes for the mu-coefficient of Kazhdan--Lusztig polynomials in
  $S_n$}, to appear in Exp. Math. \textsf{arXiv:1010.3961}
\bibitem[Wil06]{Willems} M.~Willems, \emph{$K$-th\'{e}orie \'{e}quivariante des tours de
Bott. Application \'{a} la structure multiplicative de la $K$-th\'{e}orie \'{e}quivariante
des vari\'{e}t\'{e}s de drapeaux}, Duke Math.~J.~{\bf 132}(2006), no.~2, 271--309.
\bibitem[WooYon08]{WYII} A.~Woo and A.~Yong, \emph{Governing singularities of Schubert varieties},
J.~Algebra, {\bf 320}(2008), no.~2, 495--520.
\end{thebibliography}
\end{document}